\documentclass[ejs]{imsart}

\RequirePackage[colorlinks,citecolor=blue,urlcolor=blue]{hyperref}
\usepackage{amsthm,amsmath,amssymb, bbm}

\usepackage[usenames,dvipsnames,svgnames,table]{xcolor}

\newcommand{\sgn}{\mathop{\mathrm{sign}}}

\usepackage{smile}

\newcommand{\nn}{\nonumber}

\newcommand \btt{\bbeta}
\newcommand \hbt{\hat{\btt}}
\newcommand \bttc{\bbeta^*}

\newcommand \tbt{\tilde{\btt}}

\newcommand \blam{\blambda}

\def\##1\#{\begin{align}#1\end{align}}
\def\$#1\${\begin{align*}#1\end{align*}}

\usepackage{relsize}
\newcommand{\T}{{\mathsmaller {\rm T}}}
\def\sn{\sum_{i=1}^n}

\newcommand{\BB}{\mathbb{B}}

\newcommand{\wt}{\widetilde}

\newcommand{\bfsym}[1]{\ensuremath{\boldsymbol{#1}}}
       \def \bbeta    {\bfsym{\beta}}
       \def \bdelta   {\bfsym{\delta}}

\def \soft   {\textnormal{soft}}     
\newcommand{\ora}{{\rm ora}}
 
\newcommand{\Rom}[1]{\text{\uppercase\expandafter{\romannumeral #1\relax}}}

\usepackage{enumitem}
\newcommand{\hlasso}{\mbox{{\tiny H-Lasso}}}
\newcommand{\lasso}{\mbox{{\tiny Lasso}}}
\newcommand{\cc}{{\rm c}}

\doi{10.1214/154957804100000000}
\pubyear{0000}
\volume{0}
\firstpage{1}
\lastpage{1}

\startlocaldefs
\endlocaldefs

\begin{document}

\begin{frontmatter}

\title{Iteratively Reweighted $\ell_1$-Penalized Robust Regression}
\runtitle{Iteratively Reweighted Penalized Robust Regression}

\author{\fnms{Xiaoou} \snm{Pan}\ead[label=e1]{xip024@ucsd.edu}}
\address{Department of Mathematics, University of California, San Diego \\ La Jolla, CA 92093, USA \\ \printead{e1}}

\author{\fnms{Qiang} \snm{Sun}\ead[label=e2]{qsun@utstat.toronto.edu}}
\address{Department of Statistical Sciences, University of Toronto \\ Toronto, ON M5S 3G3, Canada \\ \printead{e2}}
\and
\author{\fnms{Wen-Xin} \snm{Zhou}\corref{} \ead[label=e3]{wez243@ucsd.edu}}
\address{Department of Mathematics, University of California, San Diego \\ La Jolla, CA 92093, USA \\ \printead{e3}}

\runauthor{Pan, Sun and Zhou}

\begin{abstract}
This paper investigates tradeoffs among optimization errors, statistical rates of convergence and the effect of heavy-tailed  errors for high-dimensional robust regression with nonconvex regularization. When the additive errors in linear models have only bounded second moment,  we show that iteratively reweighted $\ell_1$-penalized adaptive Huber regression estimator satisfies exponential deviation bounds and oracle properties,  including the oracle convergence rate and variable selection consistency, under a weak beta-min condition. 
Computationally, we need as many as $\cO( \log s + \log\log d )$ iterations to reach such an oracle estimator, where $s$ and $d$ denote the sparsity and ambient dimension, respectively. Extension to a general class of robust loss functions is also considered.
Numerical studies lend strong support to our methodology and theory. 
\end{abstract}




\tableofcontents

\end{frontmatter}

\section{Introduction}
\label{sec:1}

Suppose we observe independent and identically distributed (i.i.d.) data vectors  $\{ (y_i, \bx_i): 1\leq i \leq n \}$ from $(y,\bx)$ that follows the linear model
\#
	 y   =   \bx^\T  \bbeta^*   + \varepsilon  = \sum_{j=1}^d \beta^*_j x_j + \varepsilon  ,  \label{linear.model}
\#
where  $\bx= (x_1, \ldots, x_d)^\T \in \RR^d$ with $x_1\equiv 1$ is the predictor, $\bbeta^* = (\beta^*_1,\ldots, \beta_d^*)^\T \in \RR^{d}$ is the vector of regression coefficients with $\beta^*_1$ denoting the intercept, and $\varepsilon$ is an error term satisfying $\EE(\varepsilon  | \bx ) =0$. This setting includes  the location-scale model in which $\varepsilon= \sigma(\bx ) e$,  $\sigma(\cdot): \RR^d \mapsto \RR$ is an unknown function,  and  $e$ is independent of $\bx$ and satisfies $\EE(e)=0$.  We considers the high-dimensional regime, where the number of features $d$ exceeds the sample size $n$ and $\bbeta^*$ is $s$-sparse.   Of particular interest is the case where  the error variable is  {\it asymmetric} and {\it heavy-tailed} with only bounded second moment.

Since the invention of Lasso  two decades ago \citep{Tib1996}, a variety of variable selection methods have been developed for finding a small group of covariates that are associated with the response from a large pool. The Lasso estimator $\hat \bbeta^{\lasso}$  solves the convex optimization problem $\min_{ \bbeta   \in \RR^d } \, (2n)^{-1} \sn (y_i - \bx_i^\T \bbeta)^2 + \lambda \| \bbeta \|_1$, where $\lambda >0$ is the regularization parameter.  The Lasso is an $\ell_1$-penalized least squares method in nature: the quadratic loss is used as a goodness of fit measure and the $\ell_1$-norm induces sparsity.
To achieve better performance under different circumstances, several Lasso variants have been proposed and studied; see,  \cite{FL2001}, \cite{ZH2005}, \cite{Z2006}, \cite{YL2006}, \cite{BCW2011}, \cite{SZ2012} and \cite{BvdBSSC2015}, to name a few. We refer to \cite{BvG2011}, \cite{HTW2015} and \cite{W2019} for comprehensive and systematic introductions of high-dimensional statistical methods and theory.

As a general regression analysis method, the Lasso, along with many of its variants, has two potential downsides. First, the regularized least squares methods are sensitive to the tails of error distributions, even though various alternative penalties have been proposed to achieve better model selection performance. Consider a Lasso-type estimator that solves the penalized  empirical risk minimization  $ \min_{\bbeta \in \RR^p } \{  (1/n) \sn \ell(y_i -  \bx_i^\T \bbeta  ) + \lambda \| \bbeta  \|_1 \}$, 
where $\ell(x): \RR \mapsto [0,\infty)$ is a general loss function. The effects of the loss and noise on estimation error are coded in the vector $\{ \ell'(\varepsilon_i)\}_{i=1}^n$.  If $\ell$ is the quadratic loss, this vector is likely to have relatively many large coordinates when $\varepsilon$ is heavy-tailed. As a result, the combination of the rapid growth of $\ell$ with heavy-tailed sampling distribution inevitably leads to outliers, which will eventually be translated into spurious discoveries. 
Secondly,   the $\ell_1$-penalty introduces nonnegligible estimation bias \citep{FL2001, Z2006}.   For correlated designs, the bias of the Lasso may offset true signals and creates spurious effects,  leading to inconsistency in support recovery.   Technically, this is expressed as the irrepresentable condition for the selection consistency of the Lasso \citep{ZY2006}. Under restricted eigenvalue type conditions, the Lasso and its sorted variant Slope \citep{BLT2018, ACL2019} do achieve rate optimality for prediction and coefficient estimation. However, they do not benefit much from strong signals because the bias does not diminish as signal strengthens.  Under the restricted isometry property (on the design) and Gaussian errors,   \cite{N2018} derived the lower bound for the minimax risk: $\inf_{\hat \bbeta } \sup_{\bbeta^* \in \Omega(s,a) } \mathbb E \| \hat \bbeta - \bbeta^* \|_2^2 \gtrsim \sigma^2 s / n$ when $a\gtrsim  \sigma \sqrt{ \log(ed/s)/n}$,  where $\Omega(s,a) = \{ \bbeta\in \RR^d: \| \bbeta\|_0 \leq s, \min_{j: \beta_j\neq 0} |\beta_j | \geq a \}$.  For estimating such sparse vectors with sufficiently strong signals,  Lasso can not achieve the oracle rate without the strong irrepresentable condition \citep{MB2006,ZY2006}, which is a condition on how strongly the important and unimportant variables can be correlated. This condition, however, is in general very restrictive; see \cite{Z2006} for counterexamples and numerical demonstrations.

In the presence of heavy-tailed noise, outliers occur more frequently and may have a significant impact on (regularized) empirical risk minimization when the loss grows quickly.  When the regression error $\varepsilon$ only has finite second moment,  the Lasso still achieves the minimax rate $\sqrt{s\log(d)/n}$ (under $\ell_2$-norm) but with worse deviations \citep{LM2018}. 
To reduce the ill-effects of outliers, a widely recognized strategy is to use a robust loss function that is globally Lipschitz continuous and locally quadratic. A prototypical example is the Huber loss \citep{H1964}:
\begin{align} \label{Huber.loss}
	\ell_\tau(x) =
	\left\{\begin{array}{ll}
	 x^2 /2     ~~& \mbox{if } | x | \leq  \tau ,  \\
	  \tau | x | -  \tau^2/2     ~~&  \mbox{if }  | x | > \tau ,
	\end{array}  \right.
\end{align} 
where $\tau>0$ is a robustification parameter that controls the tradeoff between the robustness and  bias.
The second important issue is the choice of sparsity-inducing penalty. In order to eliminate the nonnegligible estimation bias introduced  by convex regularization, \cite{FL2001} introduced a family of folded-concave penalties, including the smoothly clipped absolute deviation (SCAD) penalty \citep{FL2001}, minimax concave (MC+) penalty \citep{Z2010}, and the capped $\ell_1$-penalty \citep{Z2010b, SPZ2012}.  These ideas motivate the following nonconvex (folded concave) regularized $M$-estimator
\# \label{nonconvex.huber}
	\hat \bbeta  \in \argmin_{ \bbeta \in \RR^d }  \bigg\{ \hat  \cL_\tau(\bbeta ) + \sum_{j=1}^d p_\lambda(\beta_j)  \bigg\} ,
\#
where $ \hat \cL_\tau (\bbeta) := (1/n) \sn \ell_\tau(y_i -   \bx_i^\T \bbeta   )$ is the empirical loss, $\tau>0$ is a robustification parameter, and $p_\lambda : \RR  \mapsto [0, \infty)$ is a  concave penalty function with a regularization parameter $\lambda>0$.  We refer to \cite{ZZ2012} for a comprehensive survey of folded concave regularized methods.

In practice, it is inherently difficult to solve the nonconvex optimization problem \eqref{nonconvex.huber} directly.  Statistical properties, such as the rate of convergence under various norms and oracle properties, are  established for either the hypothetical global optimum that is unobtainable by any practical algorithm in polynomial time, or some local optimum that exists somewhere like a needle in a haystack.  To mitigate the gap between statistical theory and algorithmic complexity, we apply an iteratively reweighted $\ell_1$-penalized algorithm, which originates from  \cite{ZL2008},  to adaptive Huber regression.  This multi-step regularized robust regression procedure (provably) yields an estimator with desired oracle properties, and is computationally efficient because it only involves solving a sequence of (unconstrained) convex programs. Our theoretical analysis is based on and improves upon \cite{FLSZ2018}, who established the statistical and algorithmic theory for the iteratively reweighted $\ell_1$-penalized least squares regression estimator. The aim of this paper is to explore a general class of robust loss functions, typified by the Huber loss, not merely for the purpose of generality but owing to a real downside of the quadratic loss.  Typified by the Huber loss, our general principle applies to a class of robust loss functions as will be discussed in Section~\ref{sec:ext}. Software implementing the proposed procedure and reproducing our computational results is available at \href{https://github.com/XiaoouPan/ILAMM}{{\small \textsf{https://github.com/XiaoouPan/ILAMM}}}.

\subsection{Related literature}

Nonasymptotic or finite-sample analysis of regularized regression methods beyond least squares,  such as regularized empirical risk minimization (ERM) or $M$-estimation with a non-quadratic loss,  can be divided into three categories depending on the form of the regularizer/penalty.

\medskip
\noindent
{\sc $\ell_1$-regularization}: For high-dimensional sparse  linear models, \cite{minsker2015geometric} and \cite{FLW2017}, respectively, proposed a robust version of Lasso based on geometric median and $\ell_1$-penalized Huber's $M$-estimator.  Both estimators achieve sub-Gaussian deviation bounds when the regression error only has finite variance.  In such a heavy-tailed case,  \cite{LM2018} showed that the Lasso still achieves the minimax rate under expectation but with much worse deviations.
When the regression error only has finite $(1+\delta)$-th absolute moment for some $\delta\in (0,1)$, \cite{SZF2017} established exponential deviation bounds for $\ell_1$-penalized adaptive Huber regression estimator with a more delicate choice of the robustification parameter. 
For more general penalized $M$-estimators with a convex and Lipschitz continuous loss,  \cite{ACL2019} established both estimation bounds and sharp oracle inequalities. Their results do not require a local strong convexity on the loss function,  thus also including the hinge loss and quantile regression loss.
For nonconvex loss functions with a redescending derivative,  typified by Tukey's bisquare loss,  \cite{MBM2018} proved the statistical consistency of the $\ell_1$-penalized estimator subject to an $\ell_2$-constraint to stationary points.

\medskip
\noindent
{\sc Folded concave regularization}: For folded concave penalized $M$-estimators subject to a convex side constraint,  \cite{LW2015} and \cite{LW2017} were among the first to provide rigorous statistical and algorithmic theory for local optima.  
They quantified statistical accuracy by providing bounds on  $\ell_1/\ell_2$- and prediction errors between stationary points and the
population-level optimum.  They also provided conditions under which the stationary point is unique, and proposed a composite gradient algorithm for provably
solving the constrained optimization problem efficiently. In the context of generalized linear models with a sufficiently smooth link function and bounded covariates, 
 \cite{LW2017} proved under the scaling $n\gtrsim s^3 \log(d)$ that the nonconvex regularized program subject to an $\ell_1$-ball constraint has a unique stationary point given by the oracle estimator with high probability.
For linear regression with symmetric heavy-tailed errors, \cite{L2017} studied statistical consistency and asymptotic normality of nonconvex regularized robust $M$-estimators (also subject to an $\ell_1$-ball constraint) with a locally strongly convex loss.  For sub-Gaussian covariates,   \cite{L2017} proved the uniqueness of a stationary point which has $\ell_2$- and $\ell_1$-error bounds in the order of $\sqrt{ s \log(d) /n}$ and $s \sqrt{\log(d)/n}$, respectively.  
Furthermore,  under the scaling $n\gtrsim \max\{ s^2, s\log(d) \}$ and the beta-min condition $\| \bbeta^*_{\cS}\|_{\min}\gtrsim  \lambda + \sqrt{\log(s)/n}$, this stationary point coincides with the oracle estimator.

In this paper, we address nonconvex regularized robust regression from a different angle. Motivated by the local linear approximation (LLA) algorithm proposed by \cite{ZL2008},  we apply an iteratively reweighted $\ell_1$-penalized algorithm to adaptive Huber regression, which involves solving a sequence of (unconstrained) convex programs. We simultaneously analyze the statistical property  and algorithmic complexity of the solutions produced by such an iterative procedure. 
For sub-exponential covariates and asymmetric error with finite variance, we show that the multi-step  penalized estimator,  after $\cO( \log s + \log(\log d) )$ iterations,  achieves exponential deviation bounds with  $\ell_2$- and $\ell_1$-errors in the order of $\sqrt{s/n}$ and $s/\sqrt{n}$, respectively, under the scaling $n\gtrsim s \log(d)$ and the above beta-min condition.  The strong oracle property can be obtained under slightly stronger moment condition and the scaling $n\gtrsim \max\{ s^2 , s\log(d) \}$.

\medskip
\noindent
{\sc $\ell_0$-regularization}: Another popular class of sparse recovery algorithms is based on directly solving $\ell_0$-constrained or $\ell_0$-penalized empirical risk minimizations, which naturally produces  sparse solutions.  Such a formation is NP-hard,  and believed to be intractable in practice.  Despite its computational hardness, many practically useful algorithms have been proposed to solve $\ell_0$-regularized ERM,  while the statistical properties beyond least squares regression are much less studied.   We refer to \cite{BPV2020} and \cite{HTT2020} for two comprehensive survey articles on $\ell_0$-regularized regression methods.

\medskip
The idea of having the robustification parameter grow with the sample size in order to achieve exponential deviations even when the sampling distribution only has finite variance dates back to \cite{C2012} in the context of mean estimation. Therefore,  the robustness considered in this paper is primarily about nonasymptotic exponential deviation of the estimator versus polynomial tail of the error distribution.  The resulting procedure does sacrifice a fair amount of robustness to adversarial contamination of the data.
To echo the message in \cite{C2012},  the motivation of this work is different from and should not be confused with the classical notion of robust statistics.

From a variable selection perspective, this paper focuses on oracle properties of multi-step penalized robust regression estimators when the signal is sufficiently strong.  While allowing for heavy-tailed noise,  the high-dimensional feature vector $\bx\in \RR^d$ is assumed to have either sub-exponential or sub-Gaussian tails.
For more complex problems in which both  the  covariates and noise can be (i) heavy-tailed and/or (ii) adversarially contaminated,  the estimator obtained by minimizing a robust loss function is still sensitive to outliers in the feature space.  
To achieve  robustness in both feature and response spaces, recent years have witnessed a rapid development of the ``median-of-means" (MOM) principle, which dates back to  \cite{NY1983} and \cite{JVV1986},  and a variety of MOM-based procedures for regression and classification in both low- an high-dimensional settings \citep{LL2018,LM2019,CLL2019, CLL2020,LM2020,LL2020}.
We refer to \cite{LM2019b} for a recent survey.   An interesting open problem is how to efficiently incorporate the MOM principle with nonconvex regularization or iteratively reweighted $\ell_1$-regularization so as to achieve high degree of robustness and variable selection consistency simultaneously.

\subsection{Notation}
Let us summarize our notation. For every integer $k\geq 1$, we use $\RR^k$ to denote the $k$-dimensional Euclidean space. The inner and Hadamard products of any two vectors $\bu=(u_1, \ldots, u_k)^\T, \bv=(v_1, \ldots ,v_k)^\T \in \RR^k$ are defined by $\bu^\T \bv = \langle \bu, \bv \rangle= \sum_{i=1}^k u_i v_i$ and $\bu \circ \bv = (u_1 v_1, \ldots, u_k v_k )^\T$, respectively.
We use $\| \cdot \|_p$ $(1\leq p \leq \infty)$ to denote the $\ell_p$-norm in $\RR^k$: $\| \bu \|_p = ( \sum_{i=1}^k | u_i |^p )^{1/p}$ and $\| \bu \|_\infty = \max_{1\leq i\leq k} |u_i|$. Moreover, we write $\| \bu \|_{\min} = \min_{1\leq i\leq k} |u_i|$.
For $k\geq 2$, $\mathbb{S}^{k-1} = \{ \bu \in \RR^k : \| \bu \|_2 = 1 \}$ denotes the unit sphere in $\RR^k$. 
For any function $f:\RR \mapsto \RR$ and vector $\bu = (u_1,\ldots, u_k)^\T \in \RR^k$, we write $f(\bu) = (f(u_1), \ldots, f(u_k))^\T \in \RR^k$.

For $k\geq 2$, ${\rm I}_k$ represents the identity/unit matrix of size $k$. For any $k\times k$ symmetric matrix $\Sigma \in \RR^{k\times k}$, $\| \Sigma \|_2$ is the operator norm of $\Sigma$, and we use $ \lambda_{\min}(\Sigma) $ and $\lambda_{\max}(\Sigma)$ to denote the minimal and maximal eigenvalues of $ \Sigma$, respectively.
For a positive semidefinite matrix $\Sigma \in \RR^{k\times k}$,  $\| \cdot \|_{\Sigma}$ denotes the norm linked to $\Sigma$ given by $\| \bu \|_{\Sigma} = \| \Sigma^{1/2} \bu \|_2$, $\bu \in \RR^k$.
For any two real numbers $u$ and $v$, we write $u\vee v = \max(u,v)$ and $u \wedge v = \min(u,v)$. 
For any integer $d\geq 1$, we write $[d]=\{1,\ldots, d\}$. For any set $\mathcal{S}$, we use $|\mathcal{S}|$ to denote its cardinality, i.e., the number of elements in $\mathcal{S}$.

\section{Regularized Huber $M$-estimation}
\label{sec:hd}

We first revisit the $\ell_1$-penalized Huber regression estimator in Section~\ref{sec:huberlasso}, and point out two different regimes for the robustification parameter $\tau$. In Section~\ref{sec:TAC}, we propose a multi-step procedure, which is closely related to folded concave regularized Huber regression, for fitting high-dimensional sparse models with heavy-tailed noise. 
This  multi-step penalized robust regression method not only is computationally efficient, but also achieves optimal rate of convergence and oracle properties, as will be studied in Sections~\ref{sec:det.analysis} and \ref{sec:random.analysis}. Throughout, $\cS = {\rm supp}(\bbeta^*) = \{ 1\leq j \leq d : \beta_j^* \neq 0\} \subseteq [d]$ denotes the active set and $s= |\cS| $ is the sparsity.

 \subsection{$\ell_1$-penalized Huber regression}
\label{sec:huberlasso}

Given i.i.d. observations $\{(y_i, \bx_i)\}_{i=1}^n$ from  the linear model \eqref{linear.model}, consider the $\ell_1$-regularized Huber $M$-estimator, which we refer to as the {\it Huber-Lasso},
\# \label{Huber-Lasso}
	\hat \bbeta^{\hlasso } \in \argmin_{ \bbeta   \in \RR^d }  \,  \{   \hat  \cL_\tau(\bbeta ) + \lambda \| \bbeta  \|_1  \},
\# 
where $ \hat \cL_\tau(\cdot)$ is the emprical loss function defined in \eqref{nonconvex.huber}.  Statistical properties of the penalized Huber $M$-estimator have been studied by \cite{LZ2011}, \cite{FLW2017}, \cite{L2017} and \cite{ACL2019} under different assumptions.
A less-noticed problem is the connection between the robustification parameter and the error distribution, which in turn quantifies the tradeoff between robustness and unbiasedness.
Recent studies by \cite{SZF2017} reveal that the use of Huber loss is particularly suited for heavy-tailed problems in both low and high dimensions. With a properly chosen robustification parameter, calibrated by  the noise level, sample size and parametric dimension, the effects of  the heavy-tailed noise can be removed or dampened.

\begin{remark}
In practice,  it is natural to leave the intercept or a given subset of the parameters unpenalized in the penalized $M$-estimation framework.  Denote by  $\cR$ be a user-specified index set of  unpenalized parameters, which contains at least index 1.  A modified Huber-Lasso  estimator is then defined as the solution to $\min_{ \bbeta   \in \RR^d }  \{   \hat \cL_\tau(\bbeta ) + \lambda \| \bbeta_{\cR^{{\rm c}}}  \|_1  \}$, where $\| \bbeta_{\cR^{{\rm c}}}  \|_1   = \sum_{j \in \cR^{{\rm c}} } |\beta_j | $. Similar theoretical analysis can be carried out with slight modifications, and thus will be omitted for ease of exposition. 
\end{remark}

We first impose the following assumptions on the data generating process. The (random) covaraite vectors are assumed to be {\it sub-exponential/sub-gamma} \citep{BLM2013},  and we allow the regression errors to be heavy-tailed and asymmetric.

\begin{cond} \label{moment.cond}
There exist some constant $\sigma_{\bx}, c_0  \geq 1$ such that $\PP(|  \bu^\T \bx  | \geq  \sigma_{\bx}   t) \leq  c_0  e^{-t }$ for all $\bu \in \mathbb{S}^{d-1}$ and $t\geq 0$.   For simplicity,  we set $c_0=1$. Moreover, $\Sigma =\EE(\bx \bx^\T)$ is positive definite with $\rho_l = \lambda_{\min} (\Sigma) >0 $.
The regression error $\varepsilon$ satisfies $\EE(\varepsilon | \bx ) = 0$ and $ \EE(\varepsilon^2 | \bx) \leq  \sigma_2 ^2$ almost surely.
\end{cond}

\begin{theorem} \label{thm:l1huber}
Assume that Condition~\ref{moment.cond} holds for model \eqref{linear.model}.  For every $t>0$,  any optimal solution $\hat{\bbeta}^{\hlasso}$ to the convex program \eqref{Huber-Lasso} with $\tau\asymp \sigma_2  \sqrt{n/(\log d + t)}$ and $\lambda \asymp   \sigma_2  \sqrt{(\log d + t)/n}$ satisfies
\begin{equation}
\begin{aligned}
	& \| \hat \bbeta^{\hlasso} - \bbeta^* \|_2  \leq c_1  \sigma_2 (\log d + t )^{1/2} \sqrt{\frac{s}{n}} ~~\mbox{ and} \\
	& \| \hat \bbeta^{\hlasso} - \bbeta^* \|_1 \leq c_2    \sigma_2  (\log d + t )^{1/2} \frac{s}{\sqrt{n}}  \label{l1huber.bounds}
\end{aligned}
\end{equation}
with probability at least $1-3 e^{-t}$ as long as $n \geq c_3  (s \log d  +t )$, where $c_1$--$c_3$ are constants that are independent of $(n,d,s)$ and $t$.

\end{theorem}

Theorem~\ref{thm:l1huber} provides the error bounds for the one-step penalized estimator,  and paves the way for our subsequent analysis for the multi-step procedure.  
Theorem~\ref{thm:l1huber} is a modified version of Theorem~B.2 in \cite{SZF2017} (when $\delta=1$) with an explicit relation between deviation bound and confidence level under slightly relaxed moment condition on the design.
When the (conditional) distribution of $\varepsilon$ is symmetric,  $\bbeta^*$ can be identified as $\bbeta^* \in \argmin_{\bbeta\in \RR^d} \EE \hat \cL_\tau(\bbeta)$.  Then,  with a fixed $\tau$ (e.g. $\tau \asymp \sigma_2$), Theorem~\ref{thm:l1huber} can also be obtained as a special case of Theorem~2.1 in \cite{ACL2019} when the feature vector $\bx$ is sub-Gaussian.

\subsection{Iteratively reweighted $\ell_1$-penalized Huber regression}
\label{sec:TAC}

For fitting sparse regression models,  the Lasso-type estimators typically exhibit a suboptimal rate of convergence,  as compared to the oracle rate achieved by nonconvex regularization methods,  under a minimum signal strength condition \citep{ZZ2012, N2018}, also known as the beta-min condition \citep[Section 7.4]{BvG2011}.  However, as noted previously, directly solving the nonconvex optimization problem \eqref{nonconvex.huber} is computationally challenging. Moreover, statistical properties are only established for the hypothetical global optimum (or some stationary point), which is typically unobtainable by any polynomial time algorithm.

Inspired by the local linear approximation to folded concave penalties \citep{ZL2008},  we consider a multi-stage procedure that solves a sequence of convex programs up to a prespecified optimization precision.  This is an iteratively reweighted $\ell_1$-penalized algorithm, which is similar in spirit to the iteratively reweighted basis-pursuit algorithms studied in \cite{GL2011}.
Let $p_\lambda(\cdot)$ be a differentiable penalty function as in \eqref{nonconvex.huber} and recall that $ \hat \cL_\tau(\cdot)$ is the empirical loss function.
Starting with an initial estimate $\hat \bbeta^{(0)} =  ( \hat \beta_1^{(0)},\ldots,\hat \beta_d^{(0)})^\T $, consider a sequence of convex optimization programs $\{  ({\rm P}_\ell ) \}_{\ell \geq 1}$:
\#
 \min_{\bbeta = ( \beta_1, \ldots,\beta_d)^\T  } \bigg\{    \hat  \cL_\tau(\bbeta) + \sum_{j=1}^d p'_\lambda( |\hat \beta_j^{(\ell-1)} | )  | \beta_j |  \bigg\}     ~~~~~~~({\rm P}_\ell)   \label{weighted.lasso}
\#
for $\ell =1 ,2 ,\ldots$, where $\hat \bbeta^{(\ell )}  = ( \hat \beta_1^{(\ell)},\ldots,\hat \beta_d^{(\ell)})^\T$ is the optimal solution to program $({\rm P}_\ell)$.  Following \cite{ZZ2012}, we assume the following conditions on the penalty function $p_\lambda$. 

\begin{cond}  \label{cond:penalty}
The penalty function $p_\lambda$ is of the form $p_\lambda(t) = \lambda^2 p(t/\lambda)$ for $t\in \RR$, where $p: \RR \mapsto [0, \infty)$ satisfies:
(i) $p(t) = p(-t)$ for all $t$ and $p(0)=0$;  (ii) $p$ is nondecreasing on $[0,\infty)$; (iii) $p$ is differentiable almost everywhere on $(0,\infty)$ and $\lim_{t \to 0^+} p'(t) = 1$; (iv) $p'(t_1) \leq p'(t_2)$ for all $t_1 \geq t_2 > 0$.
\end{cond} 

Prototypical examples of the penalty function $p(\cdot)$ in Condition~\ref{cond:penalty} include the $\ell_1$-function,  the SCAD penalty \citep{FL2001},  MC+  penalty \citep{Z2010}, and capped-$\ell_1$ function \citep{Z2010b}.
 
\begin{enumerate}
\item (SCAD) $p(t) = \int_0^{|t|} \min\{1 , (1-\frac{u-1}{a-1})_+ \} {\rm d} u$ and $p'(t) = \sign(t)\min\{1 , (1-\frac{|t|-1}{a-1})_+ \}$ for some $a>2$.  By a Bayesian argument, \cite{FL2001} suggested the choice of $a=3.7$.

\item (MC+)   $p(t) = \int_0^{|t|}  (1-u/a)_+  {\rm d} u$ and $p'(t) = \sign(t) (1- |t|/a)_+$ for some $a>1$.  

\item (Capped-$\ell_1$) $p(t) = \min(1, |t| )$ and $p'(t) = I(|t|\leq 1)$.
\end{enumerate} 

For each $\ell\geq 1$, program $({\rm P}_\ell)$ corresponds to a weighted $\ell_1$-penalized empirical Huber loss minimization of the form
\#
	 \min_{\bbeta  \in \RR^d  } \{  \hat  \cL_\tau(\bbeta) +  \| \blambda   \circ  \bbeta  \|_1 \},   \label{general.lasso}
\#
where $\blambda = (\lambda_1,\ldots, \lambda_d)^\T$ is a $d$-vector of regularization parameters with $\lambda_j \geq 0$.  
By convex optimization theory, any optimal solution $\hat \bbeta  $ to the convex program \eqref{general.lasso} satisfies the first-order condition
\#
	 	 \nabla  \hat \cL_\tau ( \hat \bbeta ) +  \blambda  \circ \bxi = \textbf{0}_d  ~\mbox{ for some }~ \bxi  = (\xi_1 ,\ldots, \xi_d )^\T \in \partial \| \hat  \bbeta  \|_1  \subseteq [-1,1]^d  , \nn
\#
where  $\nabla  \hat \cL_\tau(\bbeta) = (-1/n )\sn \ell_\tau'(y_i -  \bx_i^\T \bbeta )\bx_i$.

\begin{definition} \label{def:opt.solution}
Following the terminology in \cite{FLSZ2018}, for a prespecified tolerance level $\epsilon >0$, we say $\wt \bbeta$ is an $\epsilon$-optimal solution to \eqref{general.lasso} if $\omega_{\blambda}(\wt \bbeta) \leq \epsilon$, where
\#
	\omega_{\blambda}(\bbeta)  :=   \min_{\bxi \in \partial \| \bbeta  \|_1 }   \|  \nabla  \cL_\tau ( \bbeta ) +  \blambda \circ \bxi  \|_{\infty}  , \ \ \bbeta \in \RR^d . 
\#
\end{definition}

In view of Definition~\ref{def:opt.solution}, for a prespecified sequence of tolerance levels $\{ \epsilon_\ell \}_{\ell \geq 1}$, we use $\wt \bbeta^{(\ell)} = (  \wt \beta_1^{(\ell)},\ldots,\wt \beta_d^{(\ell)})^\T$ to denote an $\epsilon_\ell$-optimal solution to program $({\rm P}_\ell)$, that is,
\#
	\min_{\bbeta  \in \RR^d }  \{  \hat  \cL_\tau(\bbeta) +   \| \blambda^{(\ell-1)} \circ \bbeta  \|_1 \}  , \nn
\#
where $\blambda^{(\ell-1)}   := p_\lambda'(\wt \bbeta^{(\ell-1)})$.
For simplicity, we consider a trivial initial estimator $\wt \bbeta^{(0)}   = \textbf{0}$. Since $p_\lambda'(| \wt \beta_j^{(0)}|) = p'_\lambda(0) = \lambda$ for $j=1,\ldots, d$, the program $({\rm P}_1)$ coincides with that in \eqref{Huber-Lasso}.
 In Section~\ref{sec:LAMM}, we will describe an iterative local adaptive majorize-minimization (I-LAMM) algorithm which produces an $\epsilon$-optimal solution to \eqref{general.lasso} after a few iterations.

The above procedure is sequential, and can be categorized into two stages: contraction ($\ell=1$) and tightening ($\ell\geq 2$).
As we will see in the next subsection, even starting with a trivial initial estimator that is fairly remote from the true parameter, the contraction stage will produce a reasonably good estimator whose statistical error is of the order $\sqrt{  \log(d) \cdot  s/n }$. 
Essentially, the contraction stage is equivalent to the $\ell_1$-penalized Huber regression in \eqref{Huber-Lasso}. A tightening stage further refines this coarse contraction estimator consecutively, and eventually gives rise to an estimator that achieves the oracle rate $\sqrt{s /n }$ under a weak beta-min condition.

\subsection{Deterministic analysis}
\label{sec:det.analysis}

To analyze the statistical properties of $\{ \wt \bbeta^{(\ell)} \}_{\ell \geq 1}$, we first define a ``good" event regarding the restricted strong convexity (RSC) property of the empirical Huber loss over a local $\ell_1$-cone.

\begin{definition} \label{def:l1cone}
For some $r,  l, \kappa >0$,  define the event 
\#
	\cE_1(r,l, \kappa  ) = \left\{   \inf_{\bbeta \in \bbeta^* + \BB(r) \cap \CC(l)}\frac{\langle   \nabla \hat  \cL_\tau(\bbeta ) -  \nabla \hat   \cL_\tau( \bbeta^* ) , \bbeta  - \bbeta^*  \rangle  }{\| \bbeta  - \bbeta^*  \|_{2}^2 }  \geq \kappa \right\}  ,  \label{RSC.event1}
\#
where $\BB(r) = \BB^d(r) = \{ \bdelta \in \RR^d :  \| \bdelta  \|_{2} \leq r\}$ is an $\ell_2$-ball and  $\CC( l) := \{ \bdelta \in \RR^d : \|   \bdelta   \|_1   \leq l \| \bdelta  \|_{2} \}$ is an $\ell_1$-cone. Here $\bbeta +  \BB(r) \cap  \CC(l) = \{ \bbeta + \bdelta :  \bdelta \in \BB(r) \cap  \CC(l) \}$.
\end{definition}

Throughout the following, we assume that the penalty function $p_\lambda(\cdot)$ satisfies Condition~\ref{cond:penalty}.
Moreover, define 
\#
	\bw^* =  \nabla \hat  \cL_\tau(\bbeta^*)  -  \nabla \cL_\tau(\bbeta^*)    ~~\mbox{ and }~~ b^*_\tau =  \| \cL_\tau(\bbeta^* ) \|_2  , \label{def.bias}
\#
where $\cL_\tau(\bbeta) = \EE \hat \cL_\tau(\bbeta)$ is the population loss.  Here $\bw^* \in \RR^d$ is the centered gradient vector which corresponds to the stochastic error, and $b^*_\tau $ denotes the (deterministic) approximation bias  induced by the Huber loss. See Lemma~\ref{lem:bias} in the Supplementary Material for an upper bound on the bias.

\begin{remark} \label{rmk:bias}
In this paper,  we introduce the bias term $b^*_\tau$ into the results primarily because the error distribution,  if not specified, can be asymmetric.  This term is typically nullified in the literature due to two reasons. First, under the symmetry assumption that $\varepsilon$ (conditional on $\bx$) is symmetric around zero, then for any given $\tau>0$, $\EE \{ \ell'_\tau(\varepsilon) \bx \}=0$.  Secondly,  it is sometimes assumed that $\bx_i$ and $\varepsilon_i$ are independent, and both have zero means.  Again, for any $\tau>0$, it follows that $\EE \{ \ell'_\tau(\varepsilon) \bx \} = \EE\{ \ell'_\tau(\varepsilon) \} \cdot  \EE(\bx) = 0$. In these two scenarios,  the bias $b^*_\tau$ vanishes for any given $\tau$.
\end{remark}

\begin{proposition}  \label{prop:contraction}
Let $\lambda,   r, \kappa>0$ satisfy
\# \label{scaling.1}
    \lambda \geq s^{-1/2} b^*_\tau , \quad   r >   2.5 \kappa^{-1}  s^{1/2} \lambda  .
\# 
Then,  conditioned on the event $\cE_1(r,l ,  \kappa  ) \cap \{    \lambda \geq 2   (  \| \bw^*  \|_\infty  + \epsilon_1   )\}$ with $l  = 6 s^{1/2}$,  any $\epsilon_1$-optimal solution $\wt \bbeta^{(1)}$ of program $({\rm P}_1)$ satisfies
\# \label{step1.bound}
	 \| \wt \bbeta^{(1)} - \bbeta^* \|_2 \leq   2.5 \kappa^{-1}   s^{1/2} \lambda .
\#
\end{proposition}

Proposition~\ref{prop:contraction} is deterministic in the sense that the error bound \eqref{step1.bound} holds conditioning on the event $\cE_1(r,l ,  \kappa ) \cap \{ \lambda \geq 2 ( \| \bw^*  \|_\infty  + \epsilon_1 ) \}$.
Under Condition~\ref{moment.cond} (sub-exponential design and heavy-tailed error with finite variance),  we will establish the delicate choices of $\lambda, \epsilon_1, r$ and sample size requirement in order that this  event occurs with high probability. Specifically, we will show that 
$$
	 \| \wt \bbeta^{(1)} - \bbeta^* \|_2  \lesssim \sigma_2  \sqrt{\frac{s\log (d)}{n}}   \mbox{ with high probability as long as }  n\gtrsim s\log (d).
$$

Next, we investigate the statistical properties of $\{ \wt \bbeta^{(\ell)} \}_{\ell \geq 2}$ in the tightening stage. 
We impose a {\it minimum signal strength condition} on $\| \bbeta^*_{\cS} \|_{\min} = \min_{j \in \cS} |\beta^*_j | $, so that the error rate obtained in Proposition~\ref{prop:contraction} is improvable \citep{ZZ2012, N2018}. Recall that $s= |\cS|$.

\begin{proposition} \label{prop:tightening}
Assume there exists some   $\gamma >0$ such that $p'(\gamma) >0$.
Let 
\#
	\lambda \geq  s^{-1/2} b^*_\tau, \quad  \kappa  \gamma >  0.5 p'(\gamma)     ,  \label{lambda.kappa}
\#
and  choose $c>0$ so that
\#
  0.5 p'(\gamma)  (c^2+1)^{1/2}   + 2  = c \kappa  \gamma . \label{def.c}
\#
Set $l =  \{ 2 +\frac{2}{p'(\gamma) } \}   (c^2+1)^{1/2} s^{1/2} + \frac{2}{p'(\gamma) }s^{1/2}$ and let $r>0$ satisfy
\#
		 r^{{\rm crude}}   :=	 c \gamma  s^{1/2} \lambda  \leq r . \label{constraint.r}
\#
Under  the minimum signal strength condition $\| \bbeta^*_{\cS} \|_{\min} \geq \gamma \lambda $,  and conditioned on event $\cE_1(r,l,\kappa ) \cap \{  \lambda \geq   \frac{2}{p'(\gamma) }  ( \| \bw^*  \|_\infty + \max_{\ell \geq 1} \epsilon_\ell  ) \}$,  the $\epsilon_\ell$-optimal solutions $\wt \bbeta^{(\ell)}$ $(\ell \geq 2)$ satisfy
\# \label{contraction.inequality}
 & \| \wt \bbeta^{(\ell)} - \bbeta^* \|_2 \nn \\
 & \leq \delta \cdot   \| \wt \bbeta^{(\ell-1)} - \bbeta^* \|_2  +  \underbrace{   \kappa^{-1}  \big\{  \|  p_\lambda'(|\bbeta^*_\cS| -    \gamma \lambda ) \|_2 +  \| \bw^*_{  \cS} \|_2 +  s^{1/2} \epsilon_\ell  + b^*_\tau   \big\}   }_{=: r^{{\rm ora}}  }.
\#    
where $\delta =  0.5 p'(\gamma)/ (\gamma \kappa) \in (0,1)$.
Furthermore, it holds
\# \label{contraction.inequality2}
  \| \wt \bbeta^{(\ell)} - \bbeta^* \|_{2}  \leq     \delta^{\ell-1}   r^{{\rm crude}}    +  (1-\delta)^{-1} r^{{\rm ora}} ~\mbox{ for any}~ \ell \geq 2. 
\#
\end{proposition}

Proposition~\ref{prop:tightening} unveils how the tightening stage improves the statistical rate: every tightening step shrinks the estimation error from the previous step by a $\delta$-fraction. The second term on the right-hand side of \eqref{contraction.inequality} or \eqref{contraction.inequality2} dominates the $\ell_2$-error, and up to constant factors, consists of three components,
\#
	\underbrace{  \|  p_\lambda'(|\bbeta^*_\cS| - \gamma \lambda ) \|_2 }_{{\rm shrinkage~bias}} , \quad   \underbrace{  \| \bw^*_{ \cS } \|_2 + b^*_\tau }_{{\rm oracle~rate~plus~approx.~bias}} ~~\mbox{ and }~~  \underbrace{   s^{1/2} \epsilon_\ell }_{{\rm optimization~error}}.  \nn
\#
We identify $ \|  p_\lambda'(|\bbeta^*_\cS| -  \gamma \lambda ) \|_2$ as the shrinkage bias induced by the penalty function. This explains the limitation of  the $\ell_1$-penalty $p_\lambda(t) = \lambda |t|$ whose derivative $p_\lambda'(t) = \lambda \sgn(t)$ ($t\neq 0$) does not vanish regardless of the signal strength.  Intuitively, choosing a proper penalty function $p_\lambda(\cdot)$ with a  descending  derivative reduces the bias as signal strengthens.
The second term, $ \| \bw^*_{ \cS } \|_2 + b^*_\tau$, reveals the oracle property. To see this, consider the oracle estimator defined as
\#
	\hat \bbeta^{{\rm ora}} = \argmin_{\bbeta : \bbeta_{ \cS^{\cc} } = \textbf{0}  }  \hat   \cL_\tau(\bbeta) =  \argmin_{\bbeta  :   \bbeta_{ \cS^{\cc}} = \textbf{0} }\frac{1}{n} \sn \ell_\tau(y_i - \bx_{i  , \cS}^\T  \bbeta_{  \cS}   ) . \label{def:oracle}
\#
Since $s= |\cS | \ll n$, the finite sample theory for  Huber's $M$-estimation in low dimensions \citep{SZF2017} applies to $\hat \bbeta^{{\rm ora}}$, indicating that with high probability, 
$$
	\| \hat \bbeta^{{\rm ora}}  - \bbeta^* \|_2 \lesssim  \| \bw^*_{ \cS } \|_2 + b^*_\tau   .
$$
According to Definition~\ref{def:opt.solution}, the last term $s^{1/2} \epsilon_\ell$ demonstrates the optimization error, which will be discussed  in Section~\ref{sec:LAMM}.

The above results provide conditions under which the sequence of estimators $\{ \wt \bbeta^{(\ell)} \}_{\ell \geq 1}$ satisfy the contraction property and, meanwhile, fall in a local neighborhood of $\bbeta^*$.
Another important feature of the proposed procedure is that the resulting estimator satisfies the strong oracle property, as demonstrated by the following result. Let $\{ \hat \bbeta^{(\ell)} \}_{\ell\geq 1}$ be any optimal solutions to the convex programs $\{ ({\rm P}_\ell)\}_{\ell \geq 1}$ in \eqref{weighted.lasso} with $\hat \bbeta^{(0)} = \textbf{0}$. 
Similarly to Definition~\ref{def:l1cone},  we define the following event  in regard of the restricted strong convexity of the empirical Huber loss. 
For some $r,  l, \kappa >0$, 
\#
	\cE_2(r , l ,\kappa) :=  \left\{   \inf_{ (\bbeta' ,\bbeta'') \in \cC(r,l)   }  \frac{\langle  \nabla \hat  \cL_\tau(\bbeta' ) -  \nabla \hat  \cL_\tau( \bbeta'' ) , \bbeta' - \bbeta''  \rangle  }{\|  \bbeta'  - \bbeta''  \|_2^2 } 
	\geq \kappa \right\} . \label{RSC.event2}
\#
where $\cC(r,l)= \{ (\bbeta_1, \bbeta_2):  \bbeta_1  \in \bbeta_2  + \BB(r) \cap \CC(l)  ,  \bbeta_2 \in \bbeta^* + \BB_{\Sigma}(r ) ,   {\rm supp}(\bbeta_2) \subseteq  \cS \}$ and $\BB_{\Sigma}(r) = \{ \bbeta \in \RR^d : \| \bbeta \|_{\Sigma} \leq r\}$.
Moreover, define the ``oracle" score $\bw^{\ora} \in \RR^d$ as
\#
\bw^{\ora} = \nabla \hat  \cL_{\tau}(\hat \bbeta^\ora)   ,  \label{oracle.bias}
\#
which satisfies $\bw^{\ora}_{   \cS^{\cc}} = \textbf{0}$.

\begin{proposition} \label{prop:oracle}
Suppose there exist  constants $\gamma_1 > \gamma_0>0$ such that $p'(\gamma_0) \in (0, 1/2]$, $p'(\gamma_1)=0$.
For a prespecified $\delta \in (0,1)$,   let $\kappa  \geq 1.25 / (\delta \gamma_0)$ and choose $c_0> 0$ so that
\#
   1 + 0.5 p'(\gamma_0) (c_0^2+1)^{1/2}  = c_0 \kappa \gamma_0 .	  \label{def.c0}
\#
Moreover,  set $l=  \{ 2 +\frac{2}{p'(\gamma_0) } \}   (c_0^2+1)^{1/2} s^{1/2}$ and let $r\geq c_0  \gamma_0 s^{1/2} \lambda$.  Then, conditioned on  the  event 
\#
&\bigg\{     \|  \bw^{\ora}   \|_\infty  \leq \frac{p'(\gamma_0)}{2} \lambda \bigg\}  \cap  \bigg\{   \| \hat \bbeta^{{\rm ora}} - \bbeta^* \|_\infty \leq  \frac{\lambda}{5 \delta \kappa} \bigg\}  \nn \\
&~~~~~~~~~ \cap 	\big\{ \| \hat \bbeta^{{\rm ora}} - \bbeta^* \|_{\Sigma} \leq r  \big\}   \cap \cE_2(r,l,  \kappa )  ,  \label{oracle.constraint}
\#
the {\it strong oracle property} holds under the {\it minimum signal strength condition}  $\| \bbeta^*_{\cS} \|_{\min} \geq (\gamma_0 + \gamma_1 ) \lambda $:  $\hat \bbeta^{(\ell)} = \hat \bbeta^{{\rm ora}}$ for all $\ell \geq \lceil \log(s^{1/2}/\delta)/\log(1/\delta) \rceil $.
\end{proposition}

The proofs of Propositions~\ref{prop:contraction}, \ref{prop:tightening} and \ref{prop:oracle} are provided in the Supplementary Material.   

\subsection{Random analysis}
\label{sec:random.analysis}

In this section, we complement the previous deterministic results with probabilistic bounds on the random events of interest.
To be more specific, events $\cE_1(r,l,\kappa)$ and  $\cE_2(r,l,\kappa)$ correspond to the RSC properties of $\hat \cL_\tau(\cdot)$.
The order of the regularization parameter $\lambda$ depends on $\| \bw^*  \|_\infty$, where $\bw^* =  \nabla \hat  \cL_\tau(\bbeta^* ) -  \nabla \cL_\tau(\bbeta^* )$ is the centered score function evaluated at $\bbeta^*$. 
The oracle convergence rate depends on the $\ell_2$-norm of $\bw^*_{\cS} \in \RR^s$, the subvector of $\bw^*$ indexed by $\cS$.

Under Condition~\ref{moment.cond},  $\bx = (x_1, \ldots, x_d)^\T$ is sub-exponential and $\Sigma= \EE(\bx \bx^\T) = (\sigma_{jk})_{1\leq j,k\leq d}$ is positive definite.
Here we do not require the components of $\bx$ to have zero means. 
Moreover, given the true active set $\cS \subseteq [d]$ of $\bbeta^*$, we define the following $ s \times s$ principal submatrix of $\Sigma$:
\#
	{\rm S}  = \EE(\bx_{ \cS} \bx_{\cS}^\T) , ~\mbox{ where } \bx_{ \cS} \in \RR^s \mbox{ is the subvector of $\bx$ indexed by $\cS$}.  \label{def:S}
\#
Throughout, ``$\lesssim$'' and ``$\gtrsim$" stand for ``$\leq$" and ``$\geq$", respectively, up to
constants that are independent of $(n,d,s)$ but might depend on those in Condition~\ref{moment.cond}. In particular,  define 
\#
\rho_{\bx} &= \sup_{ \bu \in \RR^d}  \{ \EE(\bx^\T \bu)^4 \}^{1/4 } / \{ \EE( \bx^\T \bu)^2\}^{1/2} \nn \\
&= \sup_{\bu \in \mathbb S^{d-1}} \{ \EE(\bu^\T \Sigma^{-1/2} \bx)^4 \}^{1/4}  \geq 1, \label{def.rhox}
\#
which is a constant depending only on $\sigma_{\bx}$. 

\begin{proposition} \label{prop:RSC}
Assume Condition~\ref{moment.cond} holds, and let $\rho_u = \lambda_{\max}( \bSigma)$. Then, for any $t \geq 0$, 
\# \label{RSC.bound}
   \inf_{ \bbeta  \in   \bbeta^* +  \BB(r)  \cap  \CC(l ) }   & \frac{\langle  \nabla \hat \cL_\tau(\bbeta) -  \nabla  \hat \cL_\tau(\bbeta^*) , \bbeta - \bbeta^* \rangle  }{\| \bbeta - \bbeta^* \|_{2}^2 }   \nn \\
   &   \geq  \frac{3}{4} \rho_l  - C _0 \sigma_{\bx}\frac{  \tau  l }{r} \sqrt{\frac{\log(2d)}{n}} -    \rho_u  \rho_{\bx}^2  \sqrt{\frac{2  t}{n}} -  \bigg( \frac{\tau}{r} \bigg)^2 \frac{t}{3n}
\#
holds with probability at least $1- e^{-t}$ as long as $\tau \geq   \max\{ C_1 \sigma_2,  C_2 r \}$ and $n\geq \log(2d)$, where $C_0, C_1 $ are absolute constants and $C_2$ depends only on $\sigma_{\bx}$.
\end{proposition}

The next proposition provides high probability bounds on $  \| \bw^* \|_\infty$ and $\| \bw^*_{\cS} \|_2 $, where $\bw^* = \nabla \hat \cL_\tau(\bbeta^*) - \nabla    \cL_\tau(\bbeta^*)$.

\begin{proposition} \label{prop:score}
Assume Condition~\ref{moment.cond} holds.   For any $t>0$,    the centered score $\bw^* \in \RR^d$ satisfies
\#
	\| \bw^*  \|_\infty \leq  2 \sigma_{\bx}   \Biggl\{       \sigma_2   \sqrt{\frac{\log(2d ) + t  }{ n}} +   \tau  \frac{\log (2d) + t }{ 2n}  \Biggr\}  .  \label{score.max.bound} 
\#
with probability at least $1-   e^{-t}$,  and  
\#
	 \| \bw^*_\cS \|_2 \leq  3 \sigma_{\bx}  \Bigg(   \sigma_2 \sqrt{\frac{ 2 s +t}{ n}} +   \tau\frac{2 s+t}{2 n}  \Biggr) ,\label{score.l2.bound}
\#
with probability at least $1- e^{-  t}$.
\end{proposition}
 
 Similarly to Theorem~\ref{thm:l1huber}, Propositions~\ref{prop:RSC} and \ref{prop:score} are also modified versions of Lemmas~C.4 and C.6 in \cite{SZF2017} under a weaker sub-exponential condition on the feature vector $\bx$. Therefore in the proofs,  we only provide the necessary steps that help improve upon  the existing results.
Together, Propositions~\ref{prop:RSC} and \ref{prop:score} reveal the impact of the robustification parameter on the statistical properties of the resulting estimator.
As discussed in Section~\ref{sec:det.analysis} above, the order of $\|  \bw(\bbeta^*)_{ \cS} \|_2$ determines the oracle rate of convergence.
In Theorem~\ref{thm:random}, we show that after only a small number of iterations, the proposed procedure leads to an estimator that achieves the oracle rate of convergence.
Recall from Section~\ref{sec:TAC} that  $\{ \wt \bbeta^{(\ell)} \}_{\ell = 1,2 ,\ldots}$ is a sequence of $\epsilon_\ell$-optimal solutions of the convex programs \eqref{weighted.lasso}, initialized at $\wt \bbeta^{(0)} = \textbf{0}$.


\begin{theorem} \label{thm:random}
Assume Conditions~\ref{moment.cond} and \ref{cond:penalty} hold, and there exist some $\gamma_1 > \gamma_0 >0$  such that
\#
	 \gamma_0   >   \rho_l^{-1} p'(\gamma_0)  , \quad p'(\gamma_0 ) >0,    \quad  p'(t) =0 ~\mbox{ for all } t \geq \gamma_1 .  \label{gamma0.cond}
\#
Given $t\geq 0$,  suppose the sample size satisfies $n\gtrsim  s\log d  + t$,  and $\epsilon_\ell \leq \sqrt{1/n}$ for all $\ell \geq 1$.  
Moreover,  suppose that we  choose a regularization parameter $\lambda \asymp \sigma_2 \sqrt{(\log d + t)/n}$, and let $\tau$ satisfy
\#
 \sigma_2 \lesssim \tau \lesssim \sigma_2 \sqrt{\frac{n }{\log d + t}}~~\mbox{ and }~~ b_\tau^*  =    \| \EE \{ \ell'_\tau(\varepsilon)   \bx \} \|_2 \leq   s^{1/2} \lambda.  \label{tau.constraint}
\#
Then,   under the minimum signal strength condition $\| \bbeta^*_{\cS} \|_{\min } \geq (\gamma_0 + \gamma_1) \lambda$,  the multi-stage estimator $\wt \bbeta^{(T)}$ with $T \gtrsim \lceil \frac{\log(\log d  + t)}{\log(1/\delta)} \rceil$ satisfies the bounds
\begin{equation}
\begin{aligned}
	& \| \wt \bbeta^{(T)} - \bbeta^* \|_2  \lesssim  \sigma_2 \sqrt{\frac{s  + t }{n}}  +  \tau \frac{s+t}{n}+ b^*_\tau    ~~\mbox{ and} \\
	 & \| \wt \bbeta^{(T)} - \bbeta^* \|_1 \lesssim s^{1/2} \Bigg(     \sigma_2 \sqrt{\frac{s+t}{n}}  + \tau \frac{s+t}{n}  +b^*_\tau  \Bigg)   \label{oracle1}
\end{aligned}
\end{equation}
with probability at least $1- 3 e^{-t}$.
 
\end{theorem}

We refer to the conclusion of Theorem~\ref{thm:random} as the {\it weak oracle property} in the sense that the proposed estimator achieves the same rate of convergence as the oracle $\hat \bbeta^{{\rm ora}}$ which knows {\it a priori} the support $\cS$ of $\bbeta^*$. We keep the two terms $\tau(s+t)/n$ and $b^*_\tau$ in the upper bounds of \eqref{oracle1} to keep track the impact of $\tau$ on the estimator error: the former is part of the stochastic error and the latter characterizes the bias. 
Below are two cases that are of general interests. 

\begin{enumerate}
\item (Symmetry) As discussed in Remark~\ref{rmk:bias},  if $\varepsilon$ (conditional on $\bx$) is symmetric around zero, then $b^*_\tau = 0$ for any $\tau>0$. To certify \eqref{rmk:bias}, $\tau$ can be taken as a constant-multiple of $\sigma_2$, and the resulting error bounds become 
\#
	\| \wt \bbeta^{(T)} - \bbeta^* \|_2  \lesssim  \sigma_2\sqrt{ \frac{s  + t }{n}}     ~~\mbox{ and }~~ \| \wt \bbeta^{(T)} - \bbeta^* \|_1 \lesssim s^{1/2}    \sigma_2 \sqrt{\frac{s+t}{n}}    \nn
\#
with probability at least $1-3e^{-t}$.

\item (Asymmetry) When the conditional distribution of $\varepsilon_i$ is asymmetric,  there will be a bias-robustness tradeoff.  If $\varepsilon$ only has bounded second moment,  by Lemma~\ref{lem:bias} in the Supplementary Material we have $b^*_\tau \lesssim \sigma_2^2 \tau^{-1}$ although $\tau b^*_\tau \to 0$ as $\tau \to \infty$.  Then, the multi-step iterative estimator $\wt \bbeta^{(T)} $ with $\tau \asymp \sigma_2 \sqrt{n/(s+\log d + t)}$ satisfies, under the scaling $n\gtrsim s\log d + t$,  that
\#
	& \| \wt \bbeta^{(T)} - \bbeta^* \|_2  \lesssim  \sigma_2\sqrt{ \frac{s  + \log d + t }{n}}     ~~\mbox{ and} \nn \\
	& \| \wt \bbeta^{(T)} - \bbeta^* \|_1 \lesssim s^{1/2}    \sigma_2 \sqrt{\frac{s+ \log d +t}{n}}    \nn
\#
with probability at least $1-3e^{-t}$.
\end{enumerate}

A more intriguing result, as revealed by the following theorem, is that our estimator achieves the strong oracle property, namely, it coincides with the oracle with high probability. Here we need slightly stronger moment conditions than those in Condition~\ref{moment.cond}, that is, the random predictor $\bx$ is {\it sub-Gaussian} and the noise variable $\varepsilon$ satisfies an $L_{2+\eta}$-$L_2$ norm equivalence for some $\eta \in (0, 1]$.

\begin{cond} \label{moment.cond2}
There exists  $\sigma_{\bx} \geq 1$ such that $\PP(|  \bu^\T \bx  | \geq \sigma_{\bx}   t) \leq  2e^{-t^2/2 }$ for all $\bu \in \mathbb{S}^{d-1}$ and $t\geq 0$. Moreover, $\Sigma = (\sigma_{jk}) =  \EE(\bx \bx^\T)$ satisfies $ \rho_l= \lambda_{\min}(\Sigma) >0$ and 
\#
  \max_{  j \in  \cS^{\cc}} \| \Sigma_{j\cS} (\Sigma_{\cS\cS} )^{-1} \|_1  \leq A_0    \label{irr}
\#
for some $A_0 >0$. The random error $\varepsilon$ satisfies $\EE(\varepsilon | \bx ) = 0$ and $\EE(\varepsilon^2 | \bx) \leq \sigma_2^2$, and $\{ \EE(|\varepsilon|^{2 + \eta} | \bx ) \}^{1/(2+\eta) } \leq  a_\eta \{ \EE(\varepsilon^2 | \bx) \}^{1/2}   $ (almost surely) for some $\eta \in (0,1]$ and $a_\eta >1$.  Moreover, $\varepsilon$ satisfies the {\it anti-concentration} property: there exists a constant $a_0 >0$ such that 
\#
	 \PP( a \leq \varepsilon \leq b | \bx ) \leq a_0 (b-a) ~~\mbox{ for all }  a\leq b . \label{anti-concentration}
\#
\end{cond}

\begin{theorem} \label{thm:oracle}
Assume Conditions~\ref{cond:penalty} and \ref{moment.cond2} hold, and there exist some $\gamma_1 > \gamma_0 > 2.5/  \rho_l$ such that $p' (\gamma_0   ) >0$, $p'(t) = 0$ for all $t\geq \gamma_1 $.  For any $t \geq 0$ and $q \geq \max(s, \log d)$,  let $\lambda \asymp \sigma_2 \sqrt{(\log d + t )/n}$ and $\tau \asymp \sigma_2 \sqrt{n/(q+t)}$. 
Moreover,  assume the sample size satisfies $n\gtrsim \max\{  s \log d + t ,   (q+t)^{1+1/\eta} (\log d)^{-1/\eta} \}$,  and the beta-min condition $\| \bbeta^*_{\cS} \|_{\min} \geq (\gamma_0 + \gamma_1 ) \lambda$.
Then, with probability at least $1-8 e^{-t}$,  $\hat \bbeta^{(\ell)} = \hat \bbeta^{{\rm ora}}$ provided $\ell \geq  \lceil \log(s^{1/2}/\delta)/\log(1/\delta) \rceil $, where $\delta := 2.5/(\rho_l \gamma_0)  \in (0,1)$.
\end{theorem}

Theorem~\ref{thm:oracle} provides a useful complement to Theorem~2 in \cite{L2017}, and differs from it in two aspects.   First,  the latter studies the estimator obtained by solving the folded concave penalized optimization program in \eqref{nonconvex.huber} subject to an $\ell_1$-ball constraint in order to ensure the existence of local/global optima.
Secondly,   Theorem~2 in \cite{L2017} establishes the strong oracle property for any stationary point $\wt \bbeta$ of the program \eqref{nonconvex.huber} (with an $\ell_1$-ball constraint) that falls inside a local neighborhood of $\bbeta^*$.  In contrast,  Theorem~\ref{thm:oracle}  concerns the strong oracle property of the proposed iteratively reweighted $\ell_1$-penalized estimator  obtained by solving a sequence of (unconstrained) convex programs \eqref{weighted.lasso}.

\begin{remark}
A direct consequence of the strong oracle property is {\it variable selection consistency}, saying that
\$
	 \PP\big\{ {\rm supp} (\hat \bbeta^{(\ell)} ) = \cS \big\} \to 1 ~~\mbox{ as }~ n, d \to \infty. 
\$
In particular, assume Condition~\ref{moment.cond2} holds with $\eta=1$,  implying that $\varepsilon$ satisfies an $L_3$-$L_2$ norm equivalence. Then,  Theorem~\ref{thm:oracle} implies that the multi-step estimator $\hat \bbeta^{(\ell)}$ with $\lambda \asymp \sigma_2 \sqrt{ (\log d)/n}$,  $\tau \asymp \sigma_2 \sqrt{n/(s + \log d)}$ and $\ell \asymp \log s$ achieves variable selection consistency as $n, d \to \infty$ under the scaling $n\gtrsim \max( s \log d,  s^2 )$ and the necessary beta-min condition $\| \bbeta^*_{\cS} \|_{\min} \gtrsim  \sigma_2 \sqrt{(\log d)/n}$ \citep{N2018}.

As discussed earlier, Lasso \citep{Tib1996} achieves desirable risk properties,  in terms of both estimation and prediction, under mild conditions, yet its variable selection consistency requires much stronger assumptions \cite{MB2006, ZY2006, W2009}.  In addition to sub-Gaussian errors, it requires a stronger beta-min condition---$\| \bbeta^*_{\cS} \|_{\min} \gtrsim  \sigma_2   \sqrt{(s \log d)/n}$, and the {\it irrepresentable condition}
\#
  \max_{  j \in  \cS^{\cc}} \| \Sigma_{j\cS} (\Sigma_{\cS\cS} )^{-1} \|_1  \leq a_0 < 1  .    \label{irr0}
\# 
See, for example,  Chapter 7 in \cite{BvG2011} and Section 7.5 in \cite{W2019}.
\end{remark}

\section{Optimization Algorithm}
\label{sec:LAMM}

In this section, we use the local adaptive majorize-minimize (LAMM) principal \citep{FLSZ2018} to derive an iterative algorithm for solving 
each subproblem $({\rm P}_\ell)$ in \eqref{weighted.lasso}: 
\$
	\min_{\bbeta \in \RR^d  }  \{  \hat  \cL_\tau(\bbeta) +   \| \blambda^{(\ell-1)} \circ \bbeta  \|_1 \} ,  \ \  \ell = 1, 2 , \ldots , 
\$
where $ \blambda^{(\ell-1)} = (\lambda^{(\ell-1)}_1, \ldots, \lambda^{(\ell-1)}_d)^\T \in \RR^d$ with $\lambda^{(\ell-1)}_j \geq 0$. Specifically, $\lambda^{(\ell-1)}_j =0$ for some $j$ means that the $j$-th coefficient is not penalized.

\subsection{LAMM algorithm}
To minimize a nonlinear function $f(\cdot)$ on $\RR^d$, at a given point $\bbeta^{(k)}$, the majorize-minimize (MM) algorithm first majorizes it by another function $g(\cdot \, |\btt^{(k)})$, which satisfies
$$
g(\btt|\btt^{(k)}) \geq f(\btt) \quad \mbox{and} \quad
g(\btt^{(k)}|\btt^{(k)}) = f(\btt^{(k)}) ~~\textnormal{for any}~ \bbeta \in \RR^d , 
$$
 and then compute $
\btt^{(k+1)} :=\argmin_{\btt\in \RR^d}  g(\btt|\btt^{(k)}) 
$  \citep{lange2000optimization}.
The objective value of such an algorithm is non-increasing in each step,  because
\#\label{0311.3}
   f(\bbeta^{(k+1)}) \stackrel{{\rm (i)}}{\leq }   g(\bbeta^{(k+1)} \,|\, \bbeta^{(k)})   \stackrel{{\rm (ii)}}{\leq } 
    g(\bbeta^{(k)}\,|\,\bbeta^{(k)}) =f(\bbeta^{(k)}) ,
\#
where inequality (i) is due to the marization property of $g(\cdot | \bbeta^{(k)})$ and inequality (ii) follows from the definition $\btt^{(k+1)}$.
\cite{FLSZ2018} observed that the global majorization requirement is not necessary.  It only requires the local properties
\begin{equation} \label{0311.b}
   f(\bbeta^{(k+1)})\leq g(\btt^{(k+1)}|\btt^{(k)})~~\text{and}~~g(\btt^{(k)}|\btt^{(k)})=f(\bbeta^{(k)})
\end{equation}
for the inequalities in \eqref{0311.3} to hold.

Using the above principle, it suffices to locally majorize  the objective function $\hat \cL_\tau(\bbeta)$ in the penalized optimization problem. At the $k$-th step with working parameter vector $\bbeta^{(\ell, k-1)}$, we use an isotropic quadratic function, that is,
\#
& F (\bbeta; \phi, \bbeta^{(\ell, k-1)} ) := \hat \cL_\tau(\bbeta^{(\ell, k-1)} ) \nn \\
&~~~~~~ +  \langle\nabla  \hat \cL_\tau(\bbeta^{(\ell,k-1)}) , \bbeta -\bbeta^{(\ell,k-1)} \rangle + \frac{\phi}{2} \|\bbeta-\bbeta^{(\ell,k-1)} \|_2^2 ,  \label{def:F}
\#
to locally majorize $\hat \cL_\tau(\bbeta)$ such that 
\#\label{eq:localmajor}
F (\bbeta^{(\ell, k)}; \phi^{(\ell, k)}, \bbeta^{(\ell, k-1)} ) \geq \hat \cL_\tau (\bbeta^{(\ell, k)} ),
\#
where $\phi^{(\ell, k)}$ is a proper quadratic coefficient at the $k$-th update, and $\bbeta^{(\ell,k)}$ is the solution to   
\$
\min_{\bbeta} \bigl\{F (\bbeta; \phi^{(\ell,k)}, \bbeta^{(\ell, k-1)} )+   \|\blambda^{(\ell-1)}\circ\bbeta   \|_1 \bigr\}. 
\$
It is easy to see that $\bbeta^{(\ell,k)}$ takes a simple explicit form
\#
\bbeta^{(\ell,k)}   =S_\textnormal{soft} \big(  \bbeta^{(\ell, k-1)}-\nabla  \hat  \cL_\tau(\bbeta^{(\ell,k-1)})/\phi^{(\ell, k)}, \blambda^{(\ell-1)}/\phi^{(\ell,k)} \big) , 
  \label{eq:iteration}
\#
where $S_\soft(\btt, \blambda) := (\sign(\beta_j) \max\{|\beta_j|-\lambda_j, 0\})_{j=1, \ldots, d}$ is the soft-thresholding operator. For simplicity, we summarize and define the above update as $\bbeta^{(\ell, k)}=T_{\blambda^{(\ell-1)}, \phi^{(\ell, k)}} (\bbeta^{(\ell, k-1)}  )$.  Using this simple update formula of $\bbeta$, we iteratively search for the pair $(\phi^{(\ell, k)}, \bbeta^{(\ell, k)})$ that ensures the local majorization \eqref{eq:localmajor}.  Starting with an initial quadratic coefficient $\phi=\phi_0$, say $10^{-4}$, we iteratively increase $\phi$ by a factor of $\gamma_u>1$ and compute 
\$
\bbeta^{(\ell, k)}=T_{\blambda^{(\ell-1)}, \phi^{(\ell, k)}}(\bbeta^{(\ell, k-1)}) ~~\textnormal{with}~~\phi^{(\ell, k)}=\gamma_u^{k-1}\phi_0, 
\$
until the local property \eqref{eq:localmajor} holds. This routine is summarized in Algorithm \ref{alg:ls}.

\begin{algorithm}[!t]
	\caption{LAMM algorithm at the $k$-th iteration of the $\ell$-th subproblem. }\label{alg:ls}
	\begin{algorithmic}[1]
		\STATE{{\bf Algorithm}: $\{\btt^{(\ell,k)},\phi^{(\ell,k)} \} \leftarrow \mbox{LAMM}(\blam^{(\ell-1)},\btt^{(\ell, k-1)}, \phi_{0},\phi^{(\ell, k-1)}) $  }
		\STATE{\textbf{Input}: $\blam^{(\ell-1)}, \btt^{(\ell, k-1)}, \phi_{0}, \phi^{(\ell,k-1)}$   }
		\STATE{\textbf{Initialize}: $\phi^{(\ell,k)}\leftarrow \max\{\phi_{0},\gamma_u^{-1}\phi^{(\ell,k-1)}\}$  }
		\STATE{\textbf{Repeat}}
		\STATE{$\quad{}$ $\btt^{(\ell,k)}\leftarrow T_{\blam^{(\ell-1)},\phi^{(\ell,k)}}(\btt^{(\ell,k-1)})$ }
		\STATE{$\quad{}$ \textbf{If} $F(\btt^{(\ell,k)},{\blam^{(\ell-1)}})  <  \hat \cL_\tau(\btt^{(\ell,k)})$ \textbf{ then } $\phi^{(\ell,k)}\leftarrow \gamma_u\phi^{(\ell,k)}$ }
		\STATE{ \textbf{Until}  $F(\btt^{(\ell,k)},\blam^{(\ell-1)})\geq\hat   \cL_\tau(\btt^{(\ell,k)})$ }
		\STATE{\textbf{Return} $\{\btt^{(\ell,k)},\phi^{(\ell,k)}\}$}
	\end{algorithmic}
\end{algorithm}

\subsection{Complexity theory}

To investigate the complexity theory of the proposed algorithm, we first impose the following standard regularity conditions on the objective function.

\begin{cond}\label{ass:lip_g}
$\nabla \hat \cL_\tau(\cdot)$ is $L$-Lipschitz continuous for some $L>0$, that is,  $
\|\nabla \hat \cL_\tau(\btt_1)-\nabla  \hat \cL_\tau(\btt_2)\|_\infty \leq  L\|\btt_1-\btt_2\|_2$ for any $\bbeta_1 ,\bbeta_2 \in \RR^d$.
\end{cond}

Our next theorem characterizes the computational complexity in the contraction stage.  Recall that $\blambda^{(0)}= (\lambda, \ldots, \lambda)^\T \in \RR^d$. 

\begin{theorem}\label{thm:comp:1}
Assume Condition~\ref{ass:lip_g} holds and the optimal solution $\hat \bbeta^{(1)}$ satisfies $\| \hat \bbeta^{(1)} - \bbeta^* \|_2 \lesssim s^{1/2} \lambda$. Then, to attain an $ \epsilon_{{\rm c}}$-optimal solution $\widetilde \bbeta^{(1)}$, i.e. $\omega_{\blambda^{(0)}} (\widetilde \bbeta^{(1)})\leq \epsilon_{{\rm c}}$, in the contraction stage, we need as many as $C_1  L^2 (1+\gamma_u)^2(\|\bbeta^*\|_2 + s^{1/2} \lambda )^2 /{\epsilon^2_{{\rm c}}}$ LAMM iterations in \eqref{eq:iteration}, where $C_1>0$ is a constant independent of $(n,d,s)$.
\end{theorem}

The sublinear rate in the contraction stage is due to the lack of global strong convexity of the loss function in this stage, because we start with a naive initial value ${\bf 0}$.  
Once we enter the contracting region where the estimator is relatively closer to the underlying true parameter vector, the problem becomes strongly convex (at least with high probability).  This endows   the algorithm a linear rate of convergence. Our next theorem provides a formal statement on the geometric   convergence rate of LAMM for solving each subproblem in the tightening stage.  To this end,  we describe a variant of the sparse eigenvalue condition. 

\begin{definition}[LSE---Localized Sparse Eigenvalue]\label{lse}
Given  $r, \tau>0$ and an integer $m\geq |\cS|$,  the localized sparse eigenvalues are defined as
	\$
	\kappa_+(m,r, \tau) =\sup \big\{{  \bdelta^\T \nabla^2 \hat \cL_\tau (\btt)  \bdelta }:  \bdelta  \in \mathbb C_0  (m ), \btt \in  \bttc+ \BB (r) \big\} \\
	\mbox{ and }~\kappa_-(m,r,\tau) =\inf \big\{{  \bdelta^\T \nabla^2  \hat  \cL_\tau (\btt) \bdelta }:    \bdelta \in \mathbb C_0 (m) , \btt \in \bttc + \BB (r) \big\}  ,
\$
where $ \mathbb C_0 (m ) : = \{  \bu  \in \mathbb{S}^{d-1}  :   \cS \subseteq {\rm supp}(\bu), \, | {\rm supp}(\bu) | \leq m  \}$ denotes a sparse cone.
\end{definition}

\begin{cond}\label{assume:LSE}
	 We say an LSE$(C_0)$ condition holds for some $C_0\geq 1$ if  there exist an integer $s'  \lesssim s$ and constants $\kappa^* , \kappa_* , C_1>0$ such that
	\begin{flalign*}
	  0 < \kappa_*\leq \kappa_-( C_0 s + 2 s' ,r , \tau)< \kappa_+( C_0 s +2 s' ,r, \tau)\leq \kappa^*  \notag\\
	  \mbox{ and }~ \kappa_+(s' , r, \tau) / \kappa_-(C_0 s +2s' , r, \tau)\leq  1 + C_1 s'/s .
	\end{flalign*}
\end{cond}

Note that if a vector $\bu \in \RR^d$ belongs to the sparse cone $\CC_0(m)$ for some $m\geq 1$,  by Cauchy-Schwarz inequality we have $\| \bu \|_1 \leq m^{1/2} \| \bu \|_2$. This implies that $\bu$ also falls into the $\ell_1$-cone $\CC(m^{1/2})$ defined in \eqref{RSC.event1}.  
Proposition~\ref{prop:RSC} will remain valid,  possibly with different constants, if the $\ell_1$-cone $\CC(l)$ therein is replaced by a sparse cone.  Note that Proposition~\ref{prop:RSC} controls the minimum LSE.  Similar results can be obtained to bound the maximum  LSE from above.

\begin{theorem}\label{thm:comp:2}
Assume LSE$(C_0)$ condition  holds for a sufficiently large $C_0>1$ and $\tau\gtrsim r\gtrsim \lambda\sqrt{s}$.
To obtain an $\epsilon_{{\rm t}}$-optimal solution $\tbt^{(\ell)}$, i.e. $\omega_{\blam^{(\ell-1)}} (\tbt^{(\ell)} ) \leq  \epsilon_{{\rm t}} $,  in  the $\ell$-th subproblem for  $\ell\geq 2$, we need as many as $C_1\log(C_2 s^{1/2} {\lambda }/{\epsilon_{{\rm t}} })$ LAMM iterations in \eqref{eq:iteration}, where $C_1$ and $C_2$ are positive constants.
\end{theorem}

We summarize the above two theorems in the following result,  which characterizes the computational complexity of the whole algorithm. 
\begin{corollary}\label{thm:comp:all}
 Assume that the conditions in Theorems~\ref{thm:comp:1} and \ref{thm:comp:2} hold.
 To achieve a sequence of  approximate  solutions $\{ \tbt^{(\ell)}\}_{\ell=1}^T$ such that $\omega_{\blam^{(0)}} (\tbt^{(1)}) \leq \epsilon_{{\rm c}} \lesssim  \lambda$ and $\omega_{\blam^{( \ell-1)}}(\tbt^{( \ell)}) \leq  \epsilon_{{\rm t}} \lesssim \sqrt{1/n}$ for $2 \leq \ell \leq  T$, the required number of  LAMM iterations  is of the order $C_1  \epsilon_{{\rm c}}^{-2} + C_2 (T-1) \log( \epsilon_{{\rm t}}^{-1}  )$,
where $C_1$ and $C_2$ are positive constants independent of $(n,d,s)$.  
\end{corollary}

\section{Extension to General Robust Losses}
\label{sec:ext}

Thus far, we have restricted our attention to the Huber loss. As a representative robust loss function, the Huber loss has the merit of being (i) globally $\tau$-Lipschitz continuous, and (ii) locally quadratic. A natural question arises that whether similar results, both statistical and computational, remain valid for more general loss functions that possess the above two features. In this section, we introduce a class of loss functions  which, combined with folded concave regularization, leads to statistically optimal estimators that are robust against heavy-tailed errors.

\begin{cond}[Globally Lipschitz and locally quadratic loss functions] \label{def:general.loss}
Consider a general loss function $\ell_\tau(\cdot)$ that is of the form $\ell_\tau(x) = \tau^2 \ell(x/\tau)$ for $x\in \RR$, where $\ell: \RR \mapsto [0,\infty)$ is convex and satisfies: (i) $\ell'(0)=0$ and $| \ell'(x) | \leq c_1 $ for all $x\in \RR$; (ii) $\ell''(0)=1$ and $\ell''(x) \geq c_2$ for all $|x| \leq c_3$; and (iii) $|\ell'(x) - x| \leq c_4 x^2$ for all $x\in \RR$,  
where $c_1$--$c_4$ are positive constants. 
\end{cond}

Note that Condition~\ref{def:general.loss} excludes some important Lipschitz continuous functions, such as the check function  for quantile regression and the hinge loss for classification, which do not have a local strong convexity.  The recent works \cite{ACL2019},  \cite{CLL2019} and \cite{CLL2020} established optimal estimation and excess risk bounds for (regularized) empirical risk minimizers and MOM-type estimators based on general convex and Lipschitz loss functions even  without a local quadratic behavior.
Our work complements the existing results on $\ell_1$-regularized ERM by showing oracle properties of nonconvex regularized methods under stronger signals.
For this reason, we need an additional local strong convexity condition on the loss.  It remains unclear whether the oracle rates or variable selection consistency can still be achieved without such a local curvature of the loss function.

We now discuss the implications of the three properties in Condition~\ref{def:general.loss}. First, since $\ell_\tau'(x) = \tau \ell'(x/\tau)$, it follows from property (i) that $\sup_{x\in \RR }|\ell_\tau'(x)| \leq c_1 \tau$. The boundedness of $|\ell_\tau'|$ facilitates the use of Bernstein's inequality on deriving upper bounds for the random quantities $ \| \bw^*  \|_\infty$ and $\|  \bw^*_{ \cS}   \|_2$ as in Proposition~\ref{prop:score}, where $\bw^* =\nabla \hat  \cL_\tau(\bbeta^*) -  \nabla \cL_\tau(\bbeta^*)$ with $\hat \cL_\tau(\bbeta) = (1/n) \sn \ell_\tau(y_i -  \bx_i^\T \bbeta )$.  Next, note that $\ell_\tau''(x) = \ell''(x/\tau)$. Property (ii) indicates that $\ell_\tau$ is strongly convex on $[-c_3 \tau, c_3 \tau]$, which turns out to be the key factor in establishing the restricted strong convexity condition on $\hat \cL_\tau$. See Proposition~\ref{prop:RSC} and Lemma~\ref{lem:oracle.RSC}.
Lastly, property (iii) is particularly useful when the error distribution is asymmetric.  Even though it can be shown under property (i) that $\nabla \hat  \cL_\tau(\bbeta^*)$ is concentrated around its expected value $ \nabla \cL_\tau(\bbeta^*)$ with high probability,  $  \nabla \cL_\tau(\bbeta^*) =  -   \EE\{ \ell'_\tau(\varepsilon) \bx \}$ is typically nonzero when the conditional distribution of $\varepsilon$  is asymmetric.   However, since $\EE(\varepsilon|\bx) = 0$, we have $\EE \{ \ell_\tau'(\varepsilon) | \bx \} = \EE \{  \ell_\tau'(\varepsilon)  - \varepsilon  | \bx \} = \tau  \EE \{\ell'(\varepsilon/\tau )  - \varepsilon/\tau | \bx \}$. Together with property (iii), this implies
\#
 | \EE \{ \ell_\tau'(\varepsilon) | \bx \} | \leq   c_4 \tau  \EE\{ (\varepsilon/\tau)^2 | \bx \} = c_4 \sigma_2^2 \tau^{-1} . \nn
\#
We thus use $b^*_\tau = \|  \nabla \cL_\tau(\bbeta^*) \|_2$ to quantify the bias; see Lemma~\ref{lem:bias} and Theorem~\ref{thm:random}.

Below we list five examples of $\ell(\cdot )$ (including the Huber loss) that satisfy Condition~\ref{def:general.loss}.
 
\begin{enumerate}
\item (Huber loss): $\ell(x)=x^2/2 \cdot I(|x| \leq 1) + (|x| - 1/2)\cdot  I(|x| >1)$ with $\ell'(x)= x I(|x| \leq 1) + \sgn(x) I(|x|>1)$ and $\ell''(x) = I(|x|\leq 1)$. Moreover,
\#
	|\ell'(x) - x | = |x-\sgn(x) | I(|x|>1)  \leq x^2. \nn
\#

\item (Pseudo-Huber loss I): $\ell(x) = \sqrt{1+x^2} - 1$, whose first and second derivatives are
\#
	 \ell'(x) = \frac{x}{\sqrt{1+x^2}} ~~\mbox{ and }~~ 	 \ell''(x) =  \frac{1}{(1+x^2)^{3/2}} , \nn
\#
respectively. It is easy to see that $\sup_{x \in \RR}| \ell'(x)|\leq 1$ and $ \ell''(x) \geq (1+c^2)^{-3/2}$ for all $|x| \leq c$ and $c>0$. Moreover, since $\ell'''(x) = -3x (1+x^2)^{-5/2}$ satisfies $|\ell'''(x)|<0.9$ for all $x$, it follows from Taylor's theorem and Lagrange error bound that $|\ell'(x) - x | = |\ell'(x) - \ell'(0) - \ell''(0) x| \leq  0.45 x^2$.

\item (Pseudo-Huber loss II): $\ell(x) = \log\{ (e^x + e^{-x})/2 \}$, whose first and second derivatives are, respectively,
\#
	 \ell'(x) =  \frac{e^x - e^{-x}}{e^x + e^{-x}} ~~\mbox{ and }~	~ \ell''(x)  = \frac{4}{(e^x + e^{-x})^2}  .\nn
\#
It follows that $\sup_{x \in \RR}| \ell'(x)|\leq 1$ and $ \ell''(x) \geq 4( e^c+e^{-c})^{-2}$ for all $|x| \leq c$ and $c>0$.
Moreover, we calculate the third derivative $\ell'''(x) =   - 8 ( e^{x} - e^{-x} ) (e^x + e^{-x})^{-4}$ that satisfies $|\ell'''(x)| < 0.4$. Again, by Taylor's theorem and Lagrange error bound, $|\ell'(x) - x | \leq 0.2 x^2$.

\item (Smoothed Huber loss I): The Huber loss is twice differentiable in $\RR$, except at $\pm 1$. Modifying the Huber loss gives rise to the following function that is twice differentiable everywhere:
$$
\ell(x) = \begin{cases}
 x^2/2 - |x|^3/6   & \mbox{ if } |x| \leq  1, \\
  |x| /2 - 1/6   & \mbox{ if } |x| > 1 ,
\end{cases}
$$
whose first and second derivatives are
$$
\ell'(x) = \begin{cases}
 x -   \sign(x) \cdot x^2/2 & \mbox{if } |x| \leq 1 , \\
 \sgn(x)  /2   & \mbox{if } |x| > 1 ,
\end{cases} \quad  
\ell''(x) = \begin{cases}
 1 -  |x|   & \mbox{if } |x| \leq 1 , \\
 0   & \mbox{if } |x| > 1.
\end{cases}
$$
Direct calculations show that $\sup_{x \in \RR}| \ell'(x)|\leq 1/2$ and $ \ell''(x) \geq 1- c$ for all $|x| \leq c$ and $0<c<1$. Since $\ell''$ is 1-Lipschitz continuous, we have $|\ell'(x) - x| \leq  x^2/2 $.

\item (Smoothed Huber loss II): Another smoothed version of the Huber loss function is
$$
\ell(x) = \begin{cases}
 x^2/2 -  x^4/24   & \mbox{ if } |x| \leq \sqrt{2} , \\
 (2\sqrt{2}/3) |x| -1  /2 & \mbox{ if } |x| > \sqrt{2} .
\end{cases}
$$
The derivative of this function is used in \cite{CG2017} for mean vector estimation. We compute 
$$
\ell'(x) = \begin{cases}
 x -  x^3/6 & \mbox{if } |x| \leq \sqrt{2} , \\
  (2\sqrt{2}/3)  \sgn(x)     & \mbox{if } |x| > \sqrt{2} ,
\end{cases} \quad  
\ell''(x) = \begin{cases}
 1 - x^2/2   & \mbox{if } |x| \leq \sqrt{2} , \\
 0   & \mbox{if } |x| > \sqrt{2} .
\end{cases}
$$
It is easy to see that $\sup_{x \in \RR}| \ell'(x)|\leq 2\sqrt{2}/3$ and $ \ell''(x) \geq 1-c^2/2$ for all $|x| \leq c$ and $0<c<\sqrt{2}$. Noting that $\ell''$ is $\sqrt{2}$-Lipschitz continuous, it holds $|\ell'(x) - x| \leq  x^2/\sqrt{2}$.
\end{enumerate}
 
The  loss functions discussed above, along with their derivatives up to order three, are plotted in Figure~\ref{loss.der} except for the Huber loss.
Provided that the loss function $\ell_\tau(\cdot)$ satisfies Condition~\ref{def:general.loss}, all the theoretical results in Sections~\ref{sec:det.analysis} and \ref{sec:random.analysis} remain valid only with different constants.
It is worth noticing that the four loss functions described in examples 2--5 also have Lipschitz continuous second derivatives; see Figure~\ref{loss.der}.  In fact, if the function $\ell$ satisfies $\ell'(0)=0$, $\ell''(0)=1$ and has $L_2$-Lipschitz second derivative, then property (iii) in Condition~\ref{def:general.loss} holds with $c_2 = L_2/2$.
The Lipschitz continuity of $\ell''(\cdot)$ also helps remove the anti-concentration condition \eqref{anti-concentration} on $\varepsilon$.
 
\begin{figure}[!t]
  \centering
  \subfigure[Loss function $\ell$]{\includegraphics[width=0.48\textwidth]{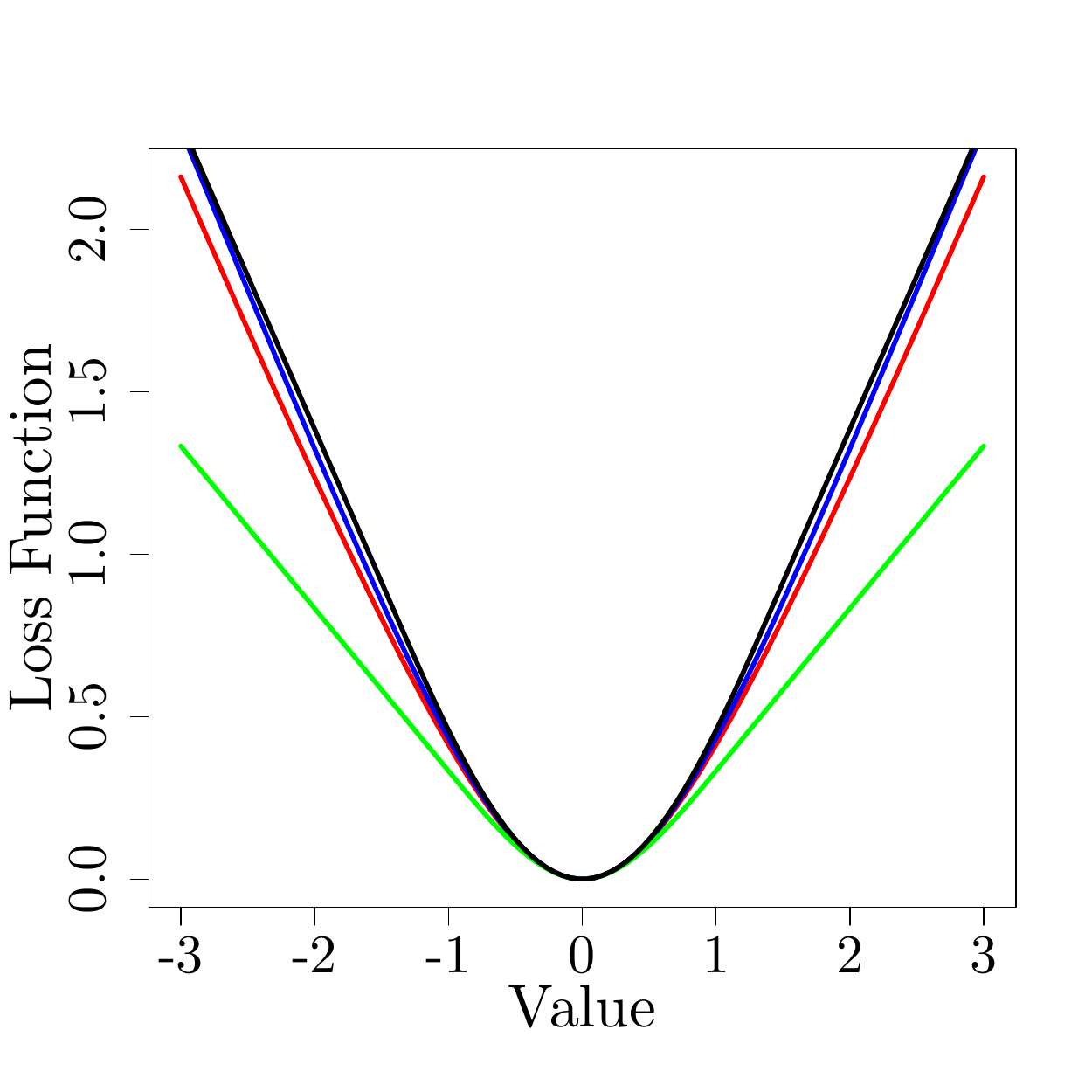}} 
  \subfigure[First derivative $\ell'$] {\includegraphics[width=0.48\textwidth]{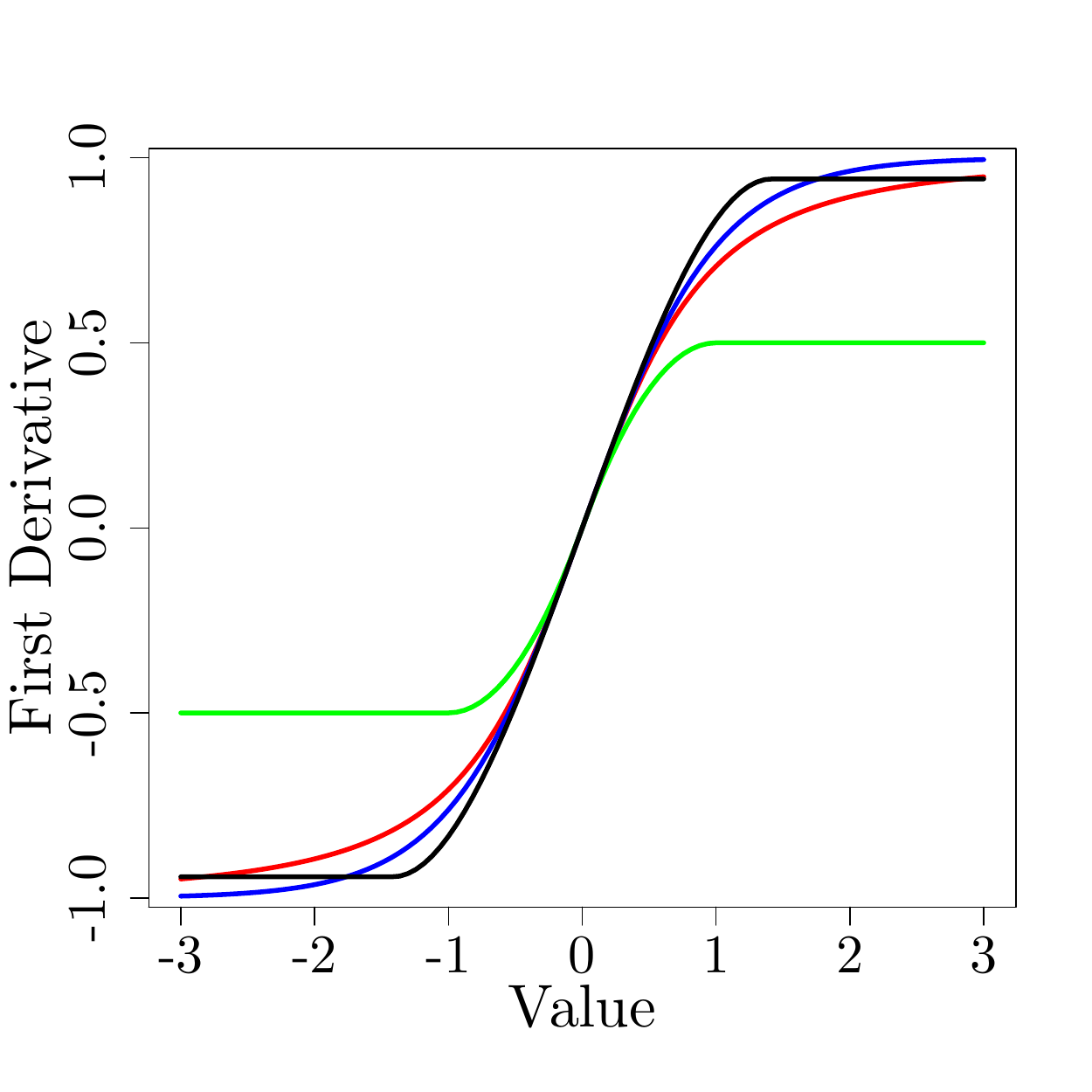}} 
  \subfigure[Second derivative $\ell''$]{\includegraphics[width=0.48\textwidth]{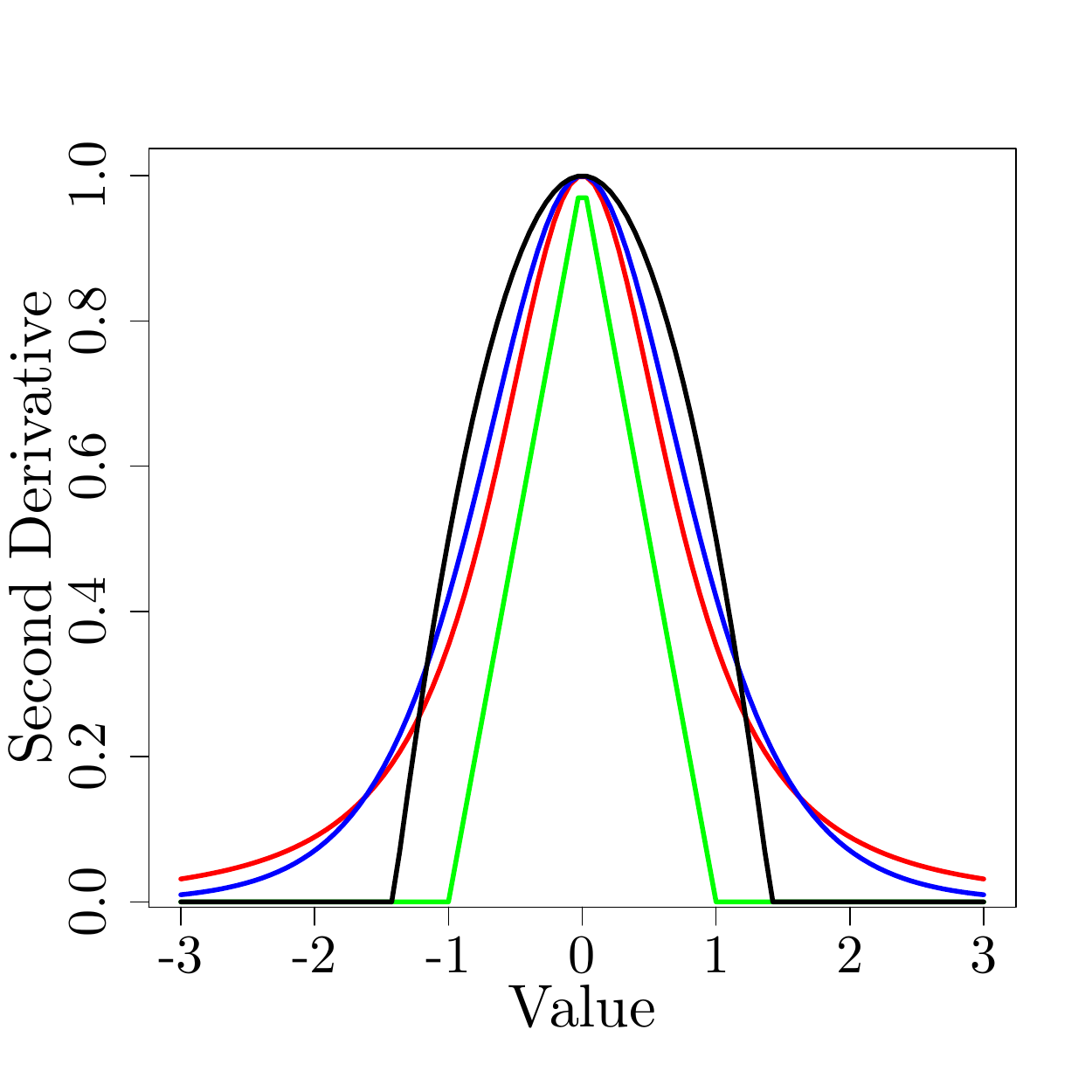}} 
  \subfigure[Third derivative $\ell'''$]{\includegraphics[width=0.48\textwidth]{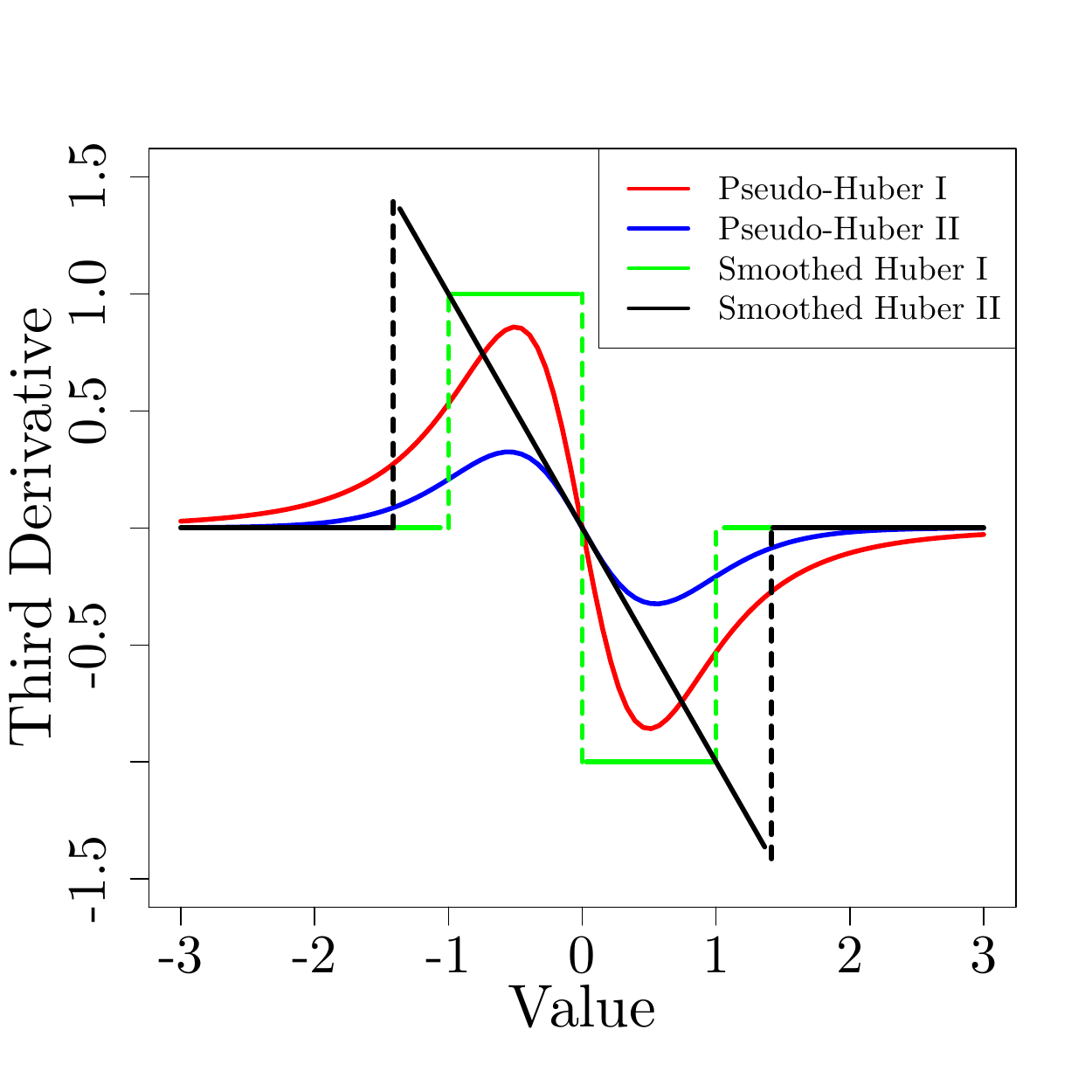}}
\caption{Examples of robust loss functions and their derivatives.}
  \label{loss.der}
\end{figure}

\section{Numerical Study}
\label{sec:sim}

In this section, we compare the empirical performance of the proposed multi-step penalized robust regression estimator with several benchmark methods, such as the Lasso \citep{Tib1996}, the SCAD and MC+ penalized least squares \citep{FL2001, Z2010}.
All the  computational results presented below are reproducible using software available at  \href{https://github.com/XiaoouPan/ILAMM}{{\small \textsf{https://github.com/XiaoouPan/ILAMM}}}.

We generate  data vectors $\{ (y_i, \bx_i) \}_{i=1}^n$ from two types of linear models: 
\begin{enumerate}
\item (Homoscedastic model): $ y_i = \bx_i^\T \bbeta^* + \varepsilon_i$ with $\bx_i \sim \mathsf{N}(0,  {\rm I}_d)$, $i=1,\ldots, n$.

\item (Heteroscedastic model): $y_i = \bx_i^\T \bbeta^* + c^{-1}( \bx_i^\T  \bbeta^*)^2 \varepsilon_i$ with $\bx_i \sim \mathsf{N}(0, {\rm I}_d)$ for $i=1,\ldots , n$,
where the constant $c$ is chosen as $c = \sqrt{3}\, \|\bbeta^*\|_2^2$ such that $\EE \{c^{-1}( \bx_i^\T  \bbeta^*)^2\}^2 = 1$, and therefore the variance of the noise is the same as that of $\varepsilon_i$. 
\end{enumerate}
In addition, we consider the following four error distributions:
\begin{enumerate}
\item Normal distribution $\mathsf{N}(\mu,\sigma^2)$ with mean $\mu = 0$ and standard deviation $\sigma = 1.5$;
\item Skewed generalized $t$ distribution ${\sf sgt}(0, 5, 0.75, 2, 2.5)$ \citep{Theo1998} with mean $\mu = 0$, variance $\sigma^2 = q / (q - 2) = 5$, $q = 2.5$, skewness parameter $\lambda = 0.75$ and shape parameter $p = 2$;
\item Lognormal distribution $\mathsf{LN}(\mu, \sigma^2)$ with log location parameter $\mu = 0$ and log shape parameter $\sigma = 1.2$;
\item Pareto distribution ${\sf Par}(x_m, \alpha)$ with scale parameter $x_m = 2$ and shape parameter $\alpha = 2.2$.
\end{enumerate}
Except for the normal distribution, all the other three  are skewed and heavy-tailed. To meet our model assumption, we subtract the mean from the lognormal and Pareto distributions.

In both homoscedastic and heteroscedastic models, the sample size $n=100$, the ambient dimension $d=1000$ and the sparsity parameter $s=6$. The true vector of regression coefficients is $\bbeta^* = (4, 3, 2, -2, -2, 2, 0, \dots, 0)^\T$,
where the first $6$ elements are non-zero and the rest are all equal to $0$. We apply the proposed TAC (\underline{T}ightening \underline{A}fter \underline{C}ontraction) algorithm to compute all the estimators with tuning parameters $\lambda$ and $\tau$ chosen via three-fold cross-validation. To be more specific, we  first choose a sequence of $\lambda$ values the same way as in the \textbf{glmnet} algorithm \citep{FHT2010}.  Guided by its theoretically ``optimal" magnitude,  the candidate set for $\tau$ is taken to be $\{ 2^j   \hat{\sigma}^{\mbox{{\tiny MAD}}} \sqrt{n / \mathrm{log}(nd)} : j=-2, -1, 0, 1, 2\}$, where $ \hat{\sigma}^{\mbox{{\tiny MAD}}} : = \mathrm{median}\{|\hat \bR - \mathrm{median}(\hat \bR)|\} / \Phi^{-1}(3/4)$ is the median absolute deviation (MAD) estimator using the residuals $\hat \bR=(\hat r_1, \ldots, \hat r_n)^\T$ obtained from the Lasso.

To highlight the tail robustness and oracle property of our algorithm, we consider the following four measurements to assess the empirical performance:
\begin{enumerate}
\item True positive, TP, which is the number of signal variables that are selected;
\item False positive, FP, which is the number of noise variables that are selected;
\item Relative error, RE$_1$ and RE$_2$, which is the relative error of an estimator $\hat \bbeta$ with respect to the Lasso under $\ell_1$- and $\ell_2$-norms:
\end{enumerate}
$$
\mathrm{RE}_1 = \frac{|| \hat \bbeta - \bbeta^*||_1}{||  \hat \bbeta^{\lasso}  - \bbeta^*||_1} ~~\mbox{ and }~~ \mathrm{RE}_2 = \frac{|| \hat \bbeta - \bbeta^*||_2}{|| \hat \bbeta^{\lasso} - \bbeta^*||_2}.
$$

\begin{table}[!h]
\centering
\begin{tabular}{c | c | c c c c c }
Error dist. & & Lasso & SCAD & Huber-SCAD & MC+ & Huber-MC+ \\
\hline
Normal & TP & 6.00(0) & 6.00(0) & 6.00(0) & 6.00(0) & 6.00(0) \\
 & FP & 24.44(14.25) & 3.11(4.53)  & 2.19(3.87) & 0.84(1.91) & 0.53(1.27)  \\
 & RE$_1$ & 1.00 & 0.23(0.12) & 0.22(0.11) & 0.19(0.09) & 0.19(0.09) \\
 & RE$_2$ & 1.00 & 0.32(0.13) & 0.33(0.13) & 0.30(0.12) & 0.30(0.12) \\
\hline
Skewed $t$ & TP & 4.74(1.37)  & 4.87(1.35) & 4.74(1.39) & 3.97(1.67) & 3.97(1.62) \\
 & FP & 20.78(17.10) & 18.49(9.65)  & 11.48(8.82)  & 4.28(4.24)  & 2.76(3.23)  \\
 & RE$_1$ & 1.00 & 0.88(0.22) & 0.73(0.23) & 0.73(0.21) & 0.65(0.22) \\
 & RE$_2$ & 1.00 & 0.91(0.17) & 0.86(0.19) & 0.94(0.20) & 0.88(0.23) \\
\hline
Lognormal  & TP & 5.68(0.87)  & 5.71(0.84) & 6.00(0.07) & 5.49(1.14) & 5.97(0.37) \\
 & FP & 29.70(16.66) & 16.75(8.70) & 3.80(4.52) & 4.32(4.62) & 0.91(1.95)  \\
 & RE$_1$ & 1.00 & 0.54(0.26) & 0.15(0.12) & 0.42(0.32) & 0.13(0.11) \\
 & RE$_2$ & 1.00 & 0.62(0.26) & 0.23(0.14) & 0.60(0.30) & 0.22(0.14) \\
\hline
Pareto & TP & 5.64(1.09) & 5.67(1.01) & 6.00(0) & 5.44(1.35) & 5.98(0.35) \\   
 & FP & 28.30(16.21)  & 14.69(8.97)  & 2.91(4.34) & 3.48(3.39)  & 0.71(1.71) \\ 
 & RE$_1$ & 1.00 & 0.51(0.30) & 0.14(0.08) & 0.40(0.25) & 0.13(0.17) \\
 & RE$_2$ & 1.00 & 0.58(0.26) & 0.21(0.11) & 0.57(0.28) & 0.22(0.22) \\
\end{tabular}
\caption{Simulation results for the Lasso, SCAD, Huber-SCAD, MC+ and Huber-MC+ estimators under the homoscedastic model.}
\label{simu.homo}
\end{table}

\begin{table}[!h]
\centering
\begin{tabular}{c | c | c c c c c }
Error dist. &   & Lasso & SCAD & Huber-SCAD & MC+ & Huber-MC+ \\
\hline
Normal & TP & 6.00(0) & 6.00(0) & 5.96(0.40) & 6.00(0) & 5.98(0.28) \\
 & FP & 22.71(16.51) & 3.29(5.76)  & 0.31(1.68)   & 0.88(2.03)  & 0.13(0.70) \\ 
 & RE$_1$ & 1.00 & 0.28(0.17) & 0.21(0.16) & 0.24(0.14) & 0.16(0.18) \\
 & RE$_2$ & 1.00 & 0.38(0.19) & 0.31(0.16) & 0.36(0.18) & 0.25(0.14) \\
\hline
Skewed $t$ & TP & 4.93(1.59) & 5.04(1.53) & 5.83(0.65) & 4.58(1.76)  & 5.52(1.17)  \\  
 & FP & 22.99(18.62)  & 18.21(10.83) & 2.71(4.29) & 4.99(5.14) & 0.92(2.42) \\
 & RE$_1$ & 1.00 & 0.83(0.30) & 0.26(0.26) & 0.69(0.29) & 0.27(0.27) \\
 & RE$_2$ & 1.00 & 0.87(0.28) & 0.34(0.26) & 0.87(0.29) & 0.35(0.28) \\
\hline
Lognormal & TP & 5.74(0.96)  & 5.77(0.91)  & 6.00(0)  & 5.65(1.14) & 6.00(0)  \\  
 & FP & 26.61(16.51) & 11.28(9.50)  & 1.23(3.55)  & 2.62(3.45) & 0.30(0.86)  \\  
 & RE$_1$ & 1.00 & 0.45(0.28) & 0.14(0.13) & 0.35(0.23) & 0.12(0.09) \\
 & RE$_2$ & 1.00 & 0.53(0.26) & 0.21(0.15) & 0.50(0.26) & 0.19(0.12) \\
\hline
Pareto & TP & 5.67(1.19) & 5.67(1.18)  & 5.97(0.42)  & 5.59(1.29)  & 5.95(0.55) \\  
 & FP & 25.56(16.04)  & 10.13(10.31) & 0.61(2.03) & 2.80(4.06) & 0.23(0.91)  \\  
 & RE$_1$ & 1.00 & 0.46(0.29) & 0.14(0.12) & 0.39(0.28) & 0.15(0.23) \\
 & RE$_2$ & 1.00 & 0.55((0.29) & 0.22(0.16) & 0.54(0.31) & 0.23(0.35) \\
\end{tabular}
\caption{Simulation results for the Lasso, SCAD, Huber-SCAD, MC+ and Huber-MC+ estimators under the heteroscedastic model.}
\label{simu.hetero}
\end{table}

Tables~\ref{simu.homo} and \ref{simu.hetero} summarize the averages of each measurement, TP, FP, RE$_1$, and RE$_2$ with their standard deviations in brackets, over 200 replications under both homoscedastic and heteroscadastic models.  
RE$_1$ and RE$_2$ for Lasso are defined to be one, so we omit their standard deviations.
Here, Huber-SCAD and Huber-MC+ signify the proposed two-stage algorithm using the SCAD and MC+ penalties, respectively.
When the noise distributions are heavy-tailed and/or skewed, we see that Huber-SCAD and Huber-MC+ outperform SCAD and MC+, respectively, with fewer spurious discoveries (false positives), smaller estimation errors and less variability. Under the homoscedastic normal model, Huber-SCAD and Huber-MC+ perform similarly to their least squares counterparts; while under heteroscedasticity, the proposed algorithm exhibits a notable advantage over existing methods on selection consistency even though the error is normally distributed.
In summary, these numerical studies validate our expectations that the proposed robust regression algorithm improves the Lasso as a general regression analysis method on two aspects: robustness against heavy-tailed (and even heteroscedastic) noise and selection consistency.

To further visualize the advantage of the multi-step penalized robust regression methods over the existing ones (e.g. Lasso, SCAD and MC+), we draw the receiver operating characteristic  (ROC) curve, which is the plot of true positive rate (TPR) against false positive rate (FPR) at various regularization parameters. Specifically, TPR and FPR are defined, respectively, as the ratio of true positive to $s$ and the ratio of false positive to $d-s$.
We generate data vectors $\{ (y_i, \bx_i) \}_{i=1}^n$ from both homoscedastic and heteroscedastic models with sample size $n=100$, dimension $d=1000$ and sparsity $s=10$. The true vector of regression coefficients is $\bbeta^* = (1.5, 1.5, \dots, 1.5, 0, \dots, 0)^\T$, where the first $10$ elements are non-zero with weaker signals than the previous experiment, and the rest are all equal to $0$. We apply the proposed TAC algorithm to implement all the five methods, Lasso, SCAD, MC+, Huber-SCAD and Huber-MC+, with a sequence of $\lambda$ values chosen as before and $\tau$ as $\hat{\sigma}^{\mbox{{\tiny MAD}}} \sqrt{n / \mathrm{log}(nd)}$. For each combination of $\lambda$ and $\tau$, the empirical FPR and TPR are computed based on 200 simulations.

\begin{figure}[!t]
  \centering
  \subfigure[$\mathsf{N}(0, 2)$]{\includegraphics[width=0.48\textwidth]{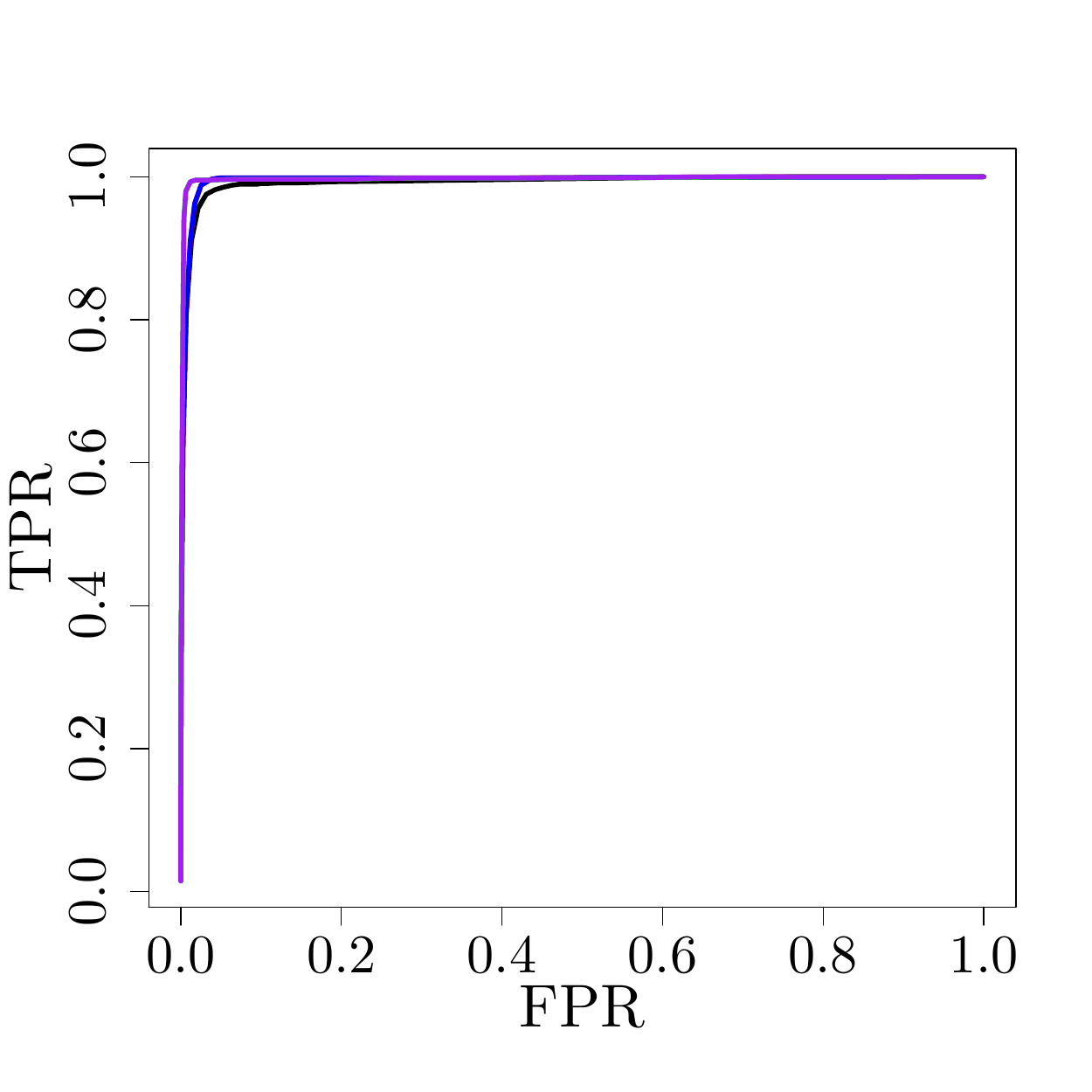}} 
  \subfigure[${\sf t}_2$] {\includegraphics[width=0.48\textwidth]{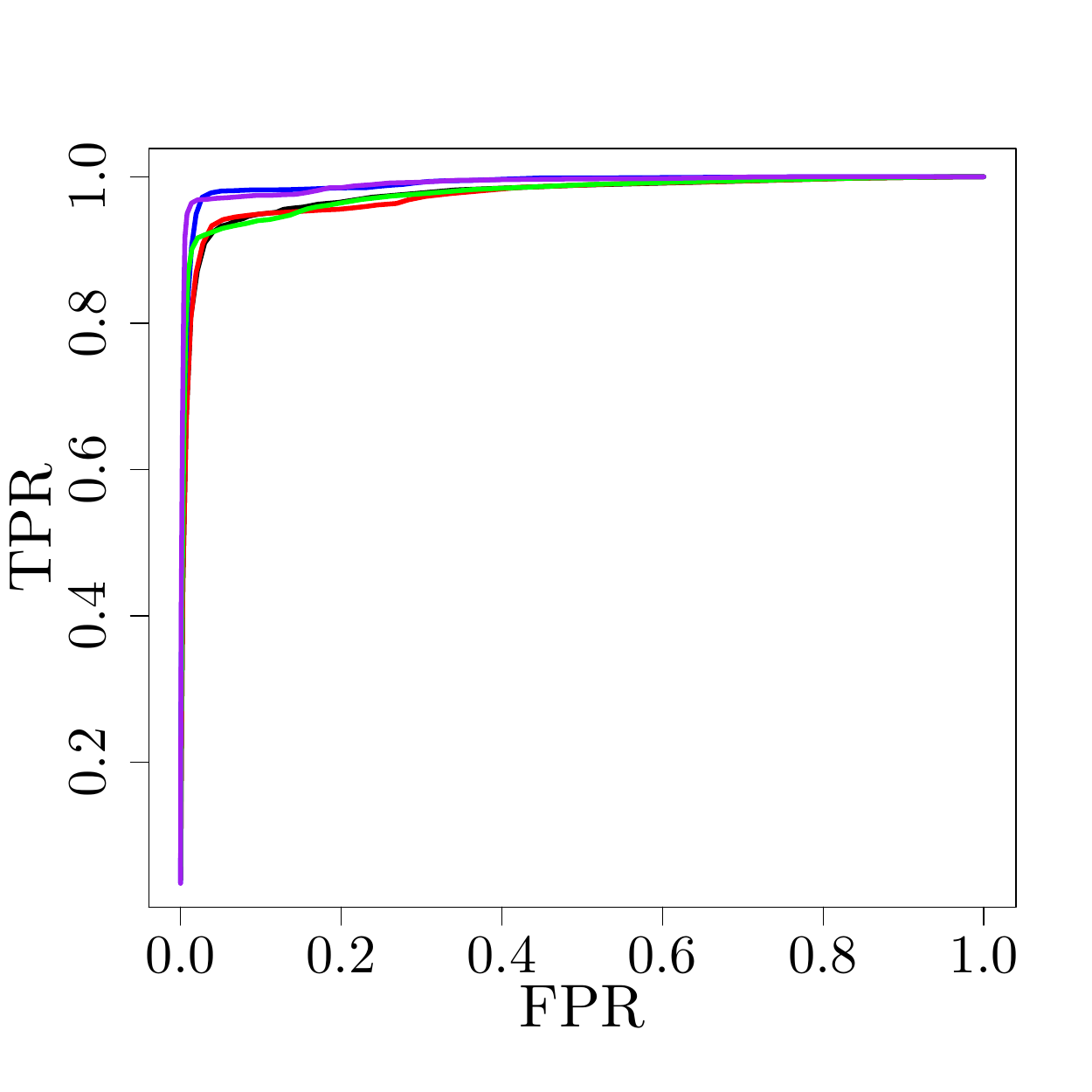}} 
  \subfigure[$\mathsf{LN}(0, 1.25^2)$]{\includegraphics[width=0.48\textwidth]{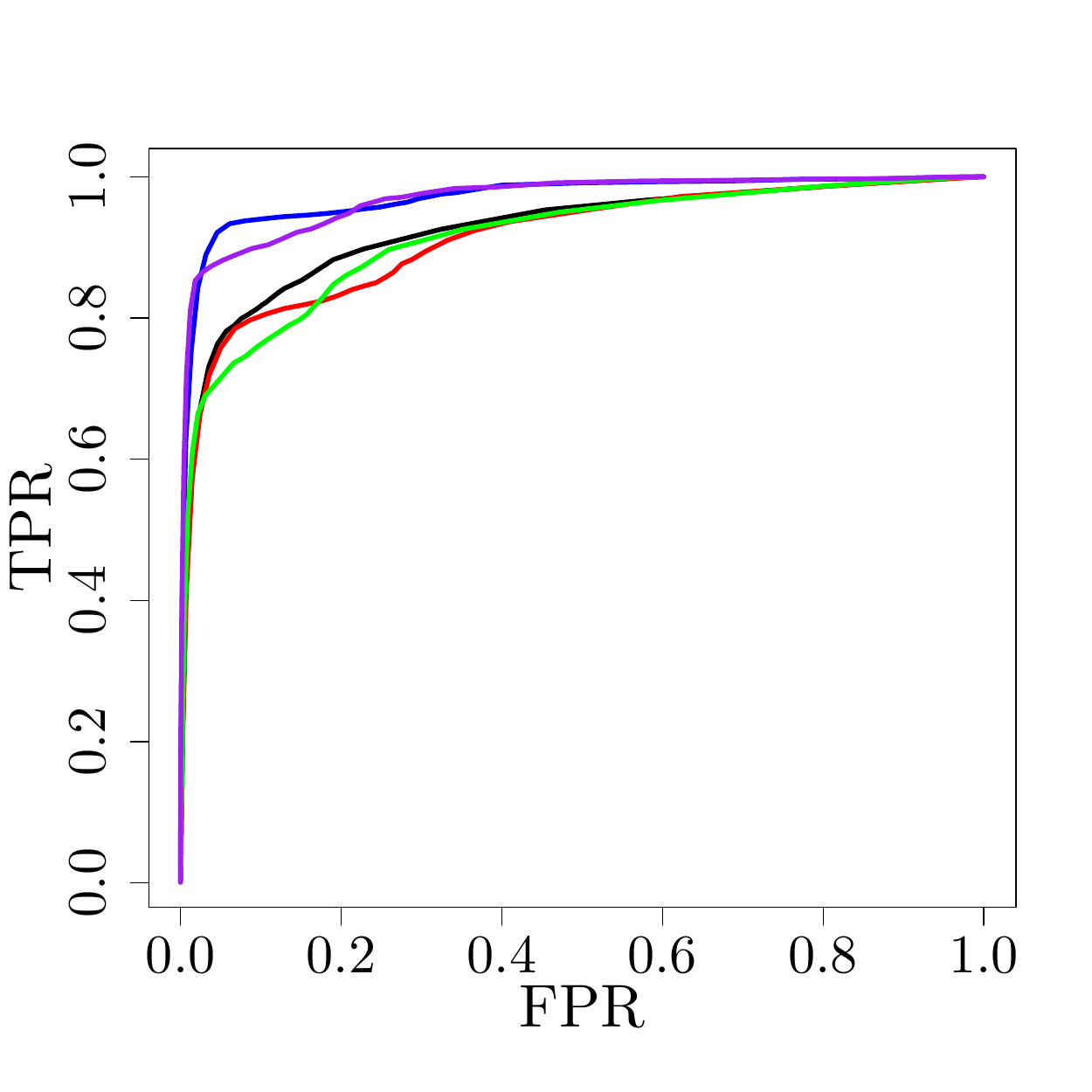}} 
  \subfigure[${\sf Par}(2, 2)$ ]{\includegraphics[width=0.48\textwidth]{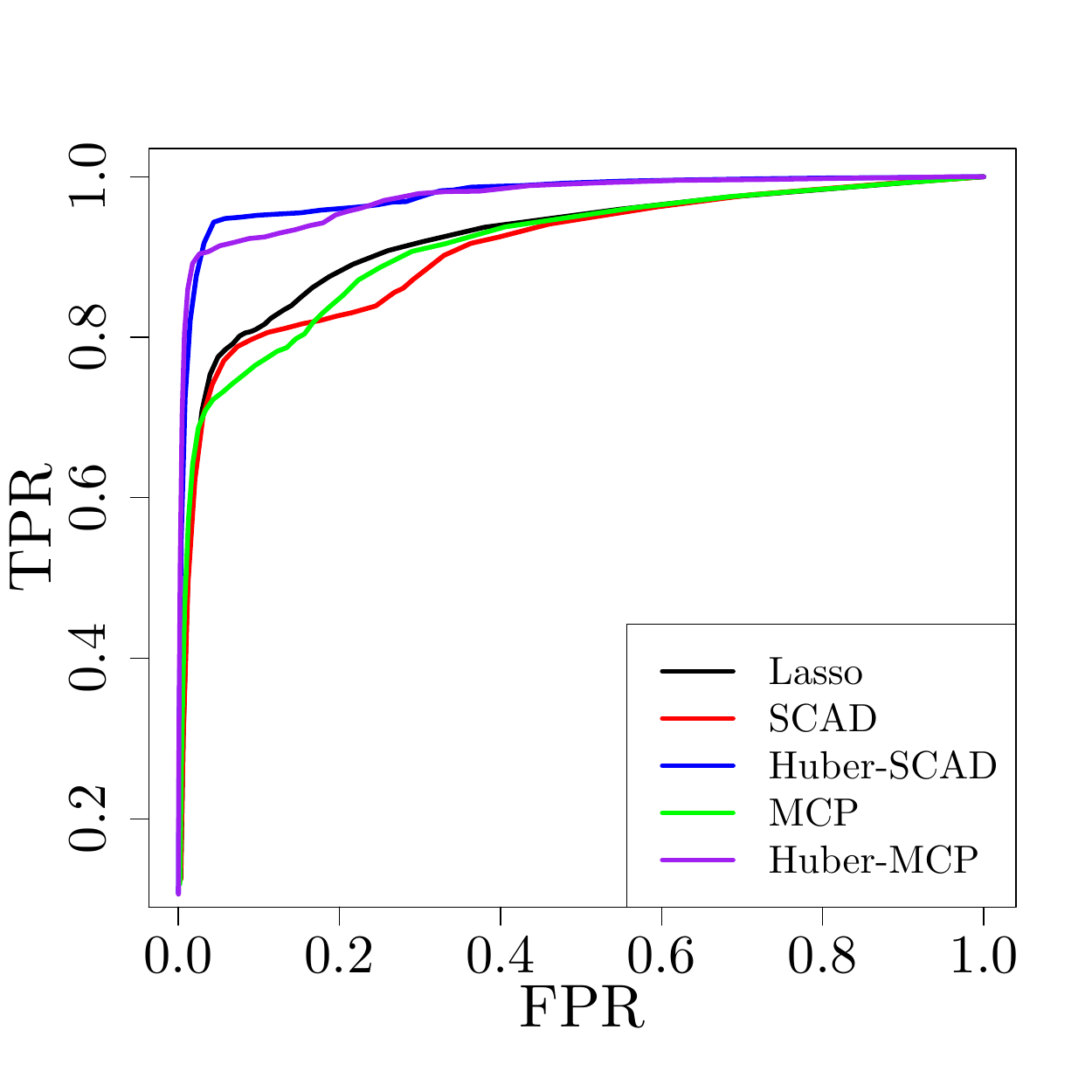}}
\caption{Plots of ROC curves of the five methods under the homoscedastic model with errors generated from four distributions: normal, Student's $t$, lognormal and Pareto.}
  \label{homo.roc}
\end{figure}

\begin{figure}[!t]
  \centering
  \subfigure[$\mathsf{N}(0, 2)$]{\includegraphics[width=0.48\textwidth]{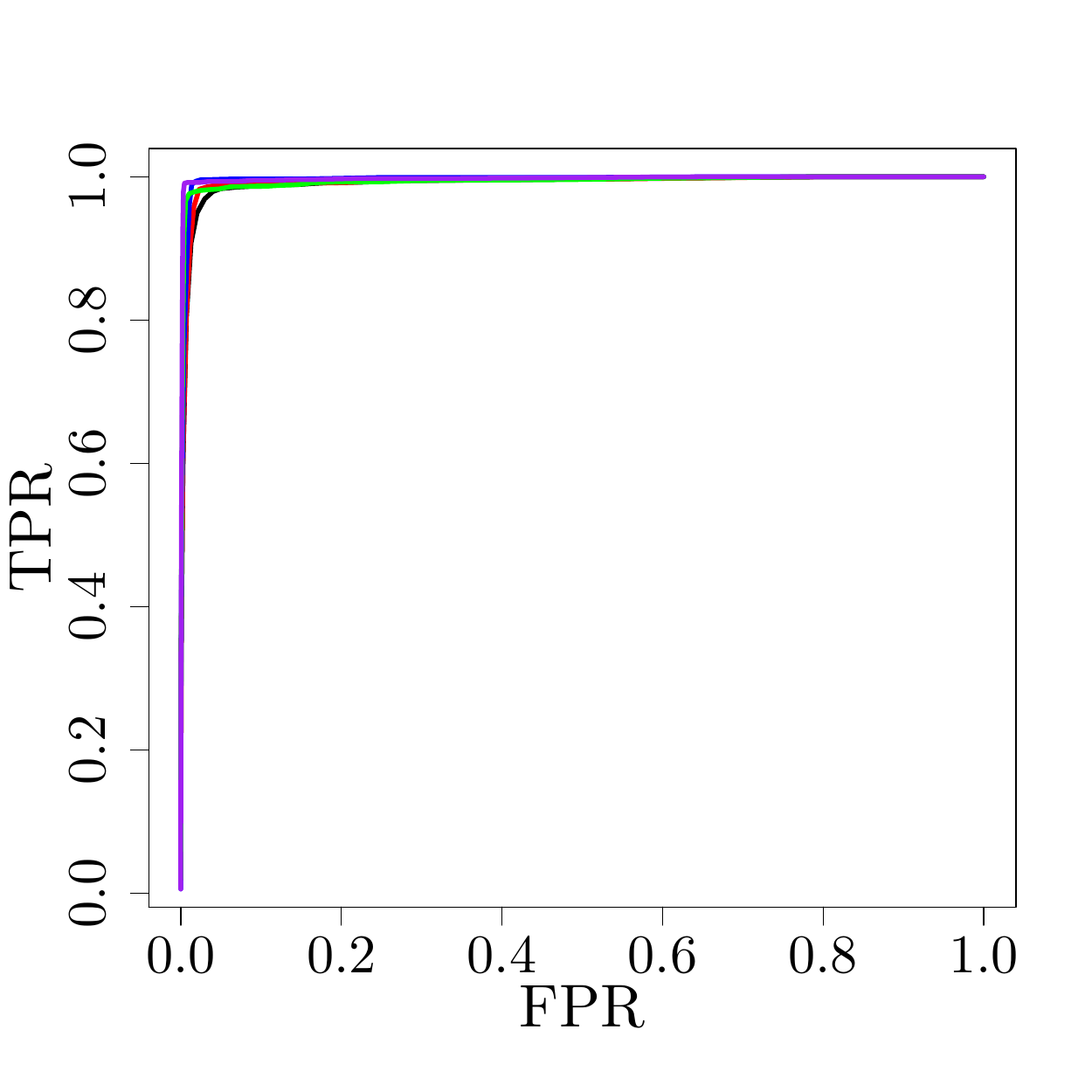}} 
  \subfigure[ ${\sf t}_2$] {\includegraphics[width=0.48\textwidth]{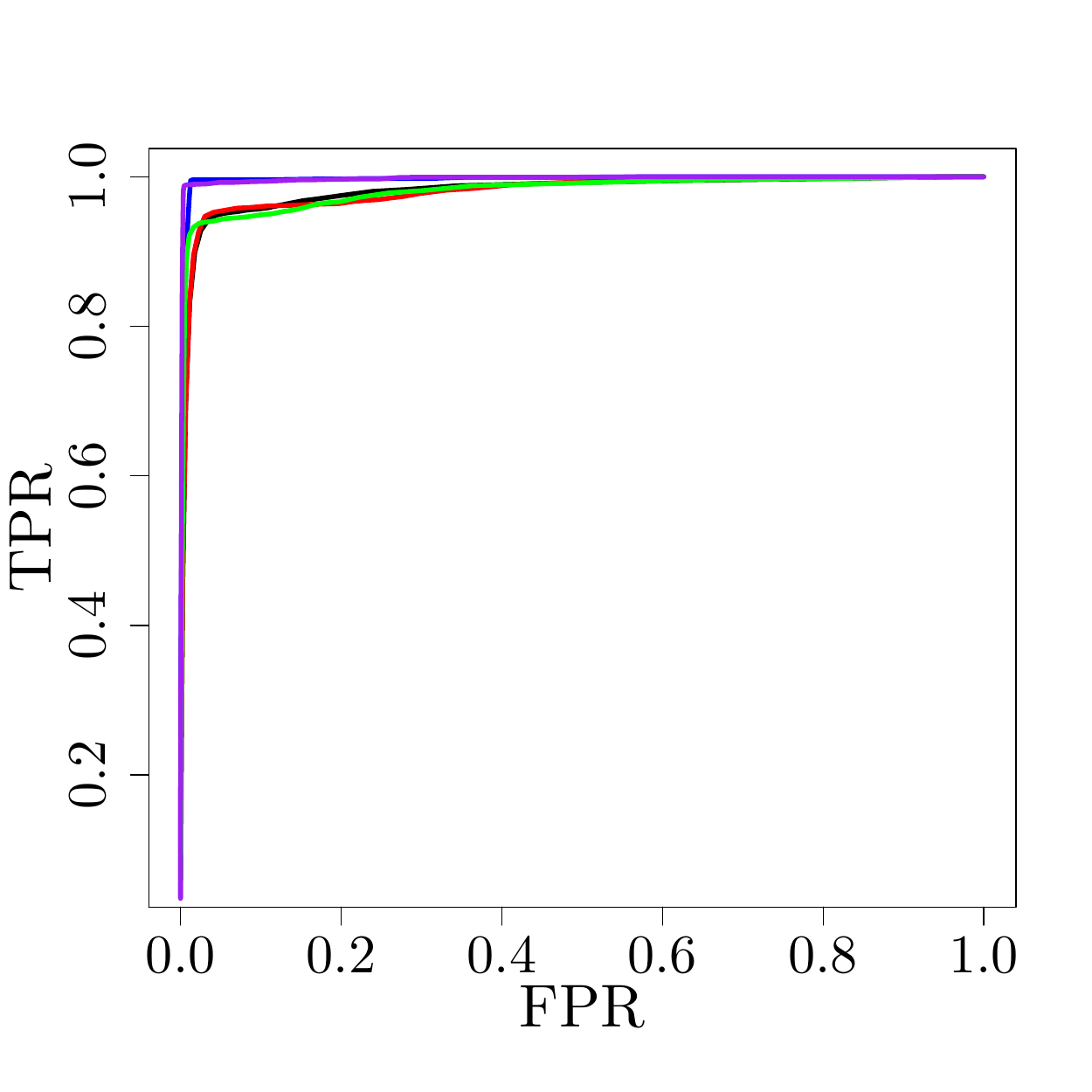}} 
  \subfigure[$\mathsf{LN}(0, 1.25^2)$]{\includegraphics[width=0.48\textwidth]{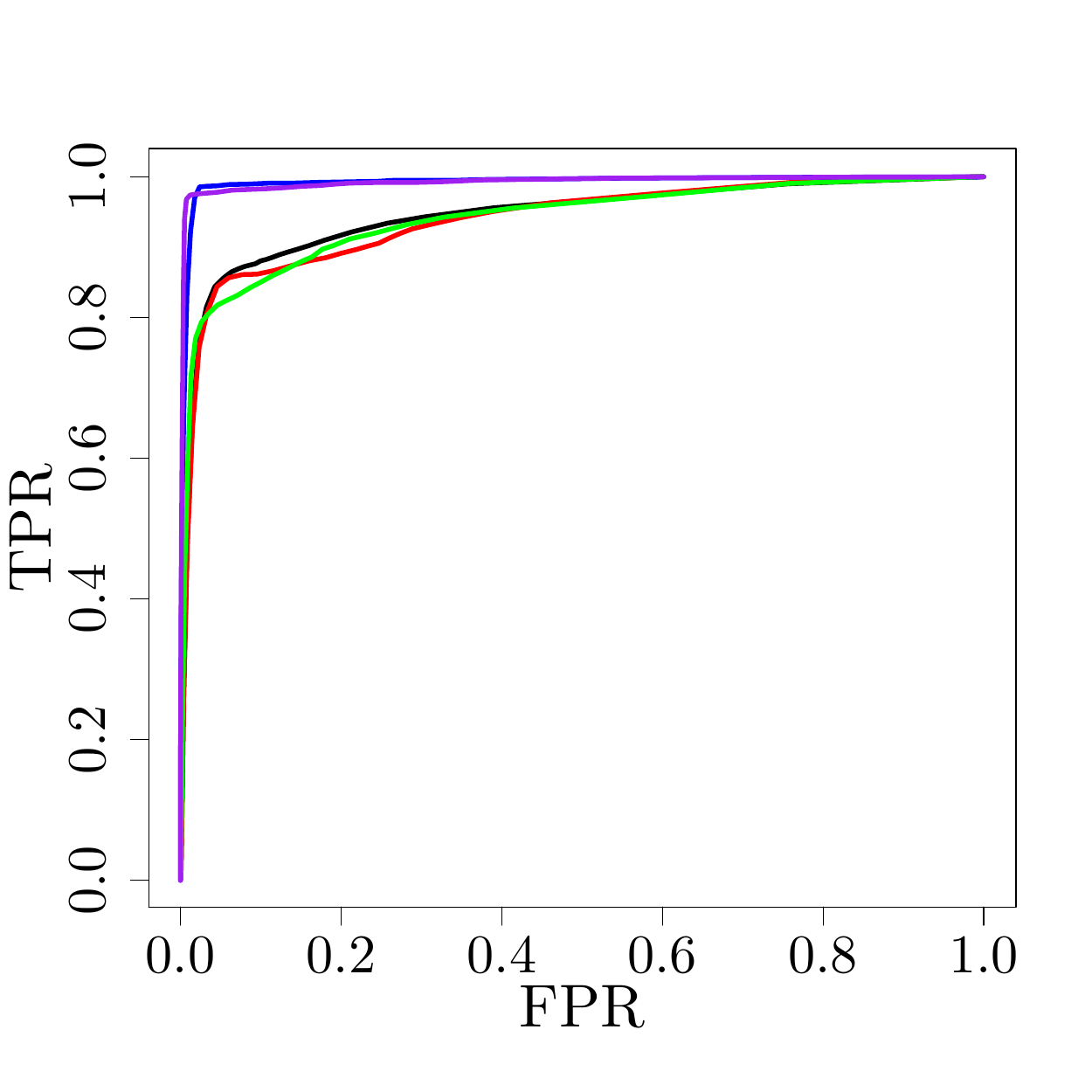}} 
  \subfigure[ ${\sf Par}(2, 2)$ ]{\includegraphics[width=0.48\textwidth]{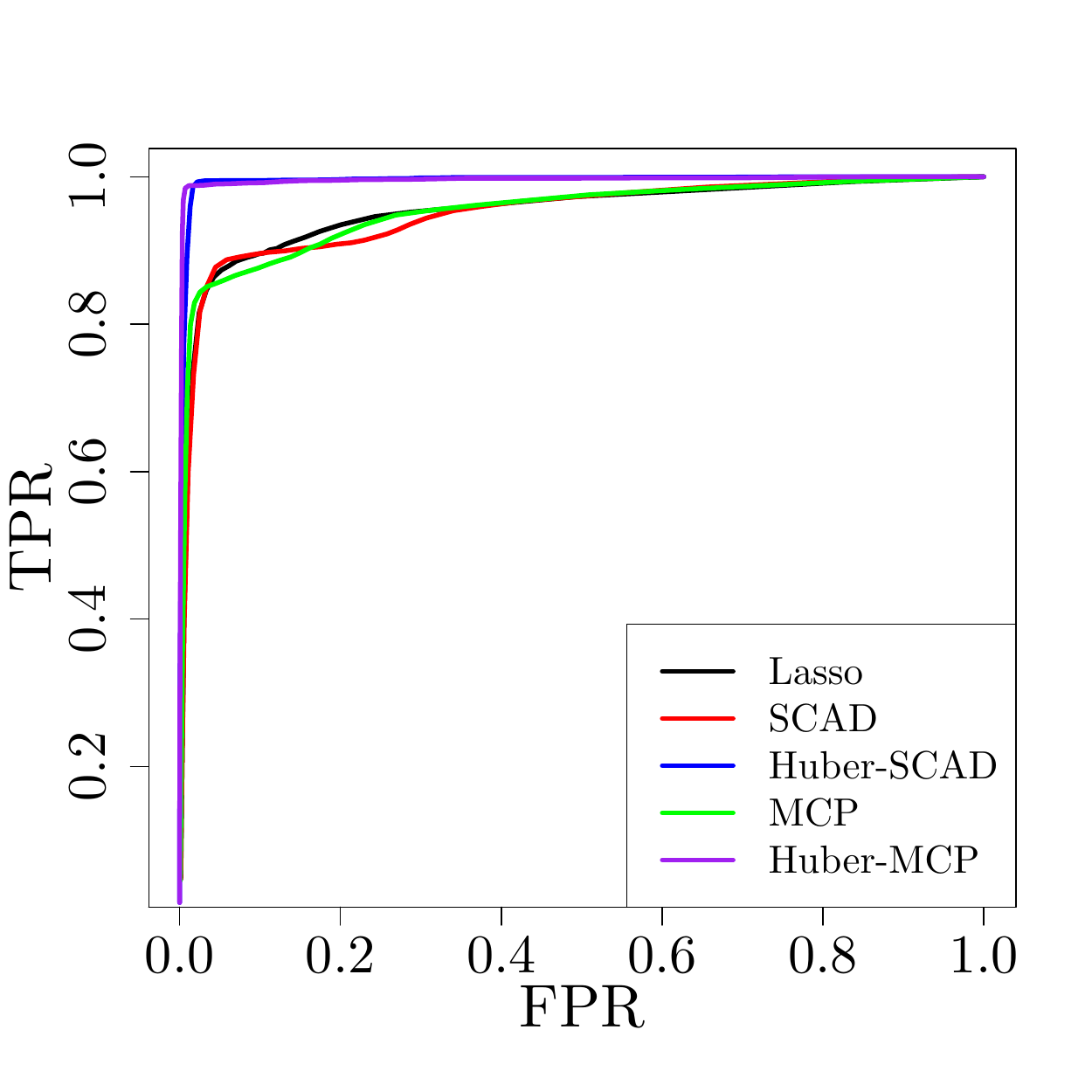}}
\caption{Plots of ROC curves of the five methods under the heteroscedastic model with errors generated from four distributions: normal, Student's $t$, lognormal and Pareto.}
  \label{hetero.roc}
\end{figure}

Figures~\ref{homo.roc} and \ref{hetero.roc}  indicate evident advantage of Huber-SCAD and Huber-MC+ over their least squares counterparts: the robust methods have a greater area under the curve (AUC) when the noise distribution is heavy-tailed and/or skewed in both homoscedastic and heteroscedastic models.
Surprisingly, even in a normal model, the proposed methods still outperform the competitors by a visible margin.

\begin{supplement}

\stitle{}

\section{Preliminaries}
Assume we observe independent data $\{ (y_i, \bx_i)\}_{i=1}^n$ from the linear model $y_i =  \bx_{ i}^\T \bbeta^* + \varepsilon_i $.  Let $\blambda = (\lambda_1,\ldots, \lambda_d)^\T$ be a $d$-vector of regularization parameters with $\lambda_j\geq 0$. Consider the optimization problem
\#
	 \min_{\bbeta  \in \RR^{d}}  \{  \hat \cL_\tau(\bbeta) +  \| \blambda   \circ  \bbeta \|_1   \},   \label{general.lasso2}
\#
where $ \hat \cL_\tau(\bbeta) = (1/n) \sn \ell_\tau(y_i -  \bx_i^\T \bbeta )$ and $ \blambda   \circ  \bbeta  = (\lambda_1 \beta_1, \ldots, \lambda_d \beta_d )^\T $. 
Moreover, define the population loss $\cL_\tau(\bbeta) = \EE \hat \cL_\tau(\bbeta)$.

The following result provides conditions under which an $\epsilon$-optimal solution to the convex program \eqref{general.lasso2} falls in an $\ell_1$-cone. Recall that $\cS = {\rm supp}(\bbeta^*)$ and $\cS^{\cc} = [d] \setminus \cS$. Moreover, define
\#
	 \bw(\bbeta) = \nabla  \hat\cL_\tau(\bbeta) -  \nabla \cL_\tau(\bbeta) ~ ~\mbox{ and }~ ~ b(\bbeta) = \|   \nabla \cL_\tau (\bbeta )\|_2 , \nn
\#
which are, respectively, the centered score function and the approximation bias. 
First, we characterize the magnitude of the bias  $b_\tau^* := b(\bbeta^*)$, as a function of $\tau$.

\begin{lemma} \label{lem:bias}
Assume $\mu_1 =  \sup_{\bu \in \mathbb{S}^{d-1}}   \EE  |  \bu^\T      \bx | <\infty$, $\EE (\varepsilon | \bx) = 0$ and $\EE(\varepsilon^2 | \bx ) \leq \sigma_2^2$ almost surely.  Then  $| b_\tau^* | \leq  \mu_1 \sigma^2_2 \tau^{-1}$ for any $\tau>0$, and $ \tau |b^*_\tau | \to 0$ as $\tau \to \infty$.
\end{lemma}

\begin{proof}[Proof of Lemma~\ref{lem:bias}]
Note that $ \nabla \cL_\tau(\bbeta^*)  = \EE \{ \ell_\tau'(\varepsilon ) \bx \}$, where $\ell_\tau'(u) = uI(|u| \leq \tau ) + \tau \sign(u)I(|u|>\tau)$. Recalling $\EE (\varepsilon | \bx) =0$, it follows that
\#
	|  \EE \{ \ell_\tau'(\varepsilon ) | \bx \} |  & = |  \EE [ \{ \varepsilon - \tau \sign(\varepsilon)\} I\{ |\varepsilon |> \tau \}  | \bx ] |  \nn \\
& \leq  \EE [ \{  | \varepsilon |  - \tau \sign(\varepsilon)\} I\{ |\varepsilon |> \tau \}  | \bx ] \nn \\
& \leq \frac{\EE  \{ ( \varepsilon^2 - \tau^2 ) I(|\varepsilon |> \tau) | \bx  \} }{\tau} \leq \frac{\sigma_2^2 -   \EE \{ \ell'_\tau(\varepsilon )^2  | \bx \}}{\tau} ~\mbox{ almost surely.} \nn
\#
By the variational representation of the $\ell_2$-norm, we have
\#
	\|    \nabla \cL_\tau(\bbeta^*) \|_2  & = \sup_{\bu \in \mathbb{S}^{d-1}} \EE \{ \ell_\tau'(\varepsilon )  \bu^\T   \bx   \}    \leq  \sigma_2^2\tau^{-1} \cdot   \sup_{\bu \in \mathbb{S}^{d-1}}   \EE  |  \bu^\T      \bx | =  \mu_1   \sigma_2^2 \tau^{-1} ,\nn 
\#
as claimed.  The second claim  follows from the fact that $\EE\{ \varepsilon^2 I(|\varepsilon | > \tau) |\bx\} \to 0$ as $\tau \to \infty$.
\end{proof}

\begin{lemma} \label{lem:l1cone}
Let $\cE $  be a subset of $[d]$ that contains $\cS$.  For any $\bbeta  \in \RR^d$ satisfying $\bbeta_{ \cE^{\cc}} = \textbf{0}$  and $\epsilon>0$, provided $\blambda = (\lambda_1,\ldots, \lambda_d)^\T  $ satisfies $ \| \blambda_{ \cE^{\cc}} \|_{\min}  >  \|  \bw(\bbeta)   \|_{\infty} + \epsilon$, any $\epsilon$-optimal solution $\wt \bbeta $ to \eqref{general.lasso2} satisfies
\#
	  \| (\wt \bbeta - \bbeta )_{ \cE^{{\rm c}}} \|_1 \leq  \frac{    \{   \| \blambda \|_{\infty} +  \|  \bw(\bbeta)   \|_{\infty} + \epsilon  \} \| (\wt \bbeta - \bbeta)_{\cE } \|_1  +  b(\bbeta) \| \wt \bbeta - \bbeta \|_2 }{ \| \blambda_{ \cE^{\cc}} \|_{\min}   -   \|  \bw(\bbeta)  \|_{\infty} - \epsilon} . \nn
\#
\end{lemma}

\begin{proof}[Proof of Lemma~\ref{lem:l1cone}]
 For any $\bxi    \in \partial \| \wt \bbeta  \|_1$, define $  \bu_{\bxi}=  \nabla \hat \cL_\tau(\wt \bbeta)  + \blambda \circ \bxi   \in \RR^d$. Note that
 \#
 &  \|   \bu_{\bxi}  \|_{\infty } \| \wt \bbeta - \bbeta  \|_1 \geq 	 \langle  \bu_{\bxi}  , \wt \bbeta - \bbeta  \rangle \nn \\
 & =  \underbrace{  \langle   \nabla   \hat\cL_\tau(\wt \bbeta) - \nabla  \hat \cL_\tau( \bbeta  ) , \wt \bbeta - \bbeta  \rangle   }_{\geq \, 0} +  \langle \nabla \hat \cL_\tau(\bbeta ) -  \nabla \cL_\tau(\bbeta)  , \wt \bbeta - \bbeta  \rangle  \nn \\
& ~~~~ + \langle  \nabla \cL_\tau(\bbeta)  , \wt \bbeta - \bbeta \rangle   + \langle   \blambda \circ \bxi  , \wt \bbeta - \bbeta  \rangle    \nn \\
   & \geq  - \| \bw(\bbeta ) \|_\infty    \| \wt \bbeta - \bbeta \|_1 -  b(\bbeta) \| \wt \bbeta - \bbeta \|_2
   + \langle   \blambda \circ \bxi  , \wt \bbeta  - \bbeta  \rangle . \nn
 \#
 Moreover, we have
\#
\langle   \blambda \circ \bxi  , \wt \bbeta  - \bbeta  \rangle & = \langle   ( \blambda \circ \bxi)_{ \cE^{{\rm c}}}  ,  (\wt \bbeta - \bbeta )_{  \cE^{{\rm c}}} \rangle +  \langle   ( \blambda \circ \bxi)_{\cE}   ,  (\wt \bbeta - \bbeta  )_{\cE}  \rangle  \nn \\
& \geq \| \blambda_{ \cE^{\cc} } \|_{\min}  \|   (\wt \bbeta - \bbeta  )_{ \cE^{{\rm c}}}  \|_1  - \| \blambda_{\cE} \|_{\infty}\|   (\wt \bbeta - \bbeta )_{\cE}  \|_1 . \nn
\#
Together, the last two displays imply
\#
 \| \bu_{\bxi} \|_{\infty } \| \wt \bbeta - \bbeta \|_1 & \geq  -  \| \bw(\bbeta) \|_{\infty} \| \wt \bbeta -\bbeta  \|_1 - b(\bbeta) \| \wt \bbeta -\bbeta \|_2 \nn \\
 & ~~~~ + \| \blambda_{ \cE^{\cc} } \|_{\min}  \|   (\wt \bbeta - \bbeta )_{ \cE^{{\rm c}}}  \|_1  - \| \blambda_{\cE} \|_{\infty}\|   (\wt \bbeta - \bbeta  )_{\cE}  \|_1  . \nn
\#
Since the right-hand side of this inequality does not depend on $\bxi$, taking the infimum with respect to $\bxi    \in \partial \| \wt \bbeta  \|_1$ on both sides to reach 
\#
	\omega_{\blambda} (\wt \bbeta )  \| \wt \bbeta - \bbeta  \|_1  &\geq   -  \| \bw(\bbeta) \|_{\infty} \| \wt \bbeta -\bbeta  \|_1 - b(\bbeta) \| \wt \bbeta -\bbeta \|_2   \nn \\
	& ~~~~ +   \| \blambda_{\cE^{\cc} } \|_{\min}  \|   (\wt \bbeta - \bbeta  )_{\cE^{{\rm c}}}  \|_1  - \| \blambda_{\cE} \|_{\infty}\|   (\wt \bbeta - \bbeta  )_{\cE}  \|_1 . \nn
\#
By definition, $\wt \bbeta$ is an $\epsilon$-optimal solution so that $\omega_{\blambda} (\wt \bbeta ) \leq \epsilon$. Putting together the pieces, we obtain
\#
& \{  \epsilon + \| \bw(\bbeta)  \|_{\infty} \} \| \wt \bbeta - \bbeta  \|_1 + b(\bbeta) \| \wt \bbeta -\bbeta \|_2   \nn \\
& ~~~ \geq  \| \blambda_{ \cE^{\cc} } \|_{\min}  \|   (\wt \bbeta - \bbeta )_{ \cE^{{\rm c}}}  \|_1  - \| \blambda_{\cE} \|_{\infty}\|   (\wt \bbeta - \bbeta  )_{\cE}  \|_1 . \nn
\#
Decompose $\| \wt \bbeta - \bbeta  \|_1$ as $\|   (\wt \bbeta - \bbeta )_{\cE}  \|_1 + \|   (\wt \bbeta - \bbeta  )_{ \cE^{{\rm c}}}  \|_1 $, the stated result follows immediately. 
\end{proof}

\begin{lemma} \label{lem:deterministic.error.bound}
Consider some $\bbeta \in \RR^d$ satisfying $\bbeta_{ \cS^{\cc}} = \textbf{0}$,  and  let $\cE \subseteq [d]$ be a subset that contains $\cS$ and has cardinality $ |\cE |=k$.  
Assume  that  $\blambda=(\lambda_1, \ldots, \lambda_d)^\T$ satisfies $\| \blambda \|_\infty \leq \lambda$ and  $\| \blambda_{ \cE^{\cc} } \|_{\min} \geq \rho \lambda   >0$ for some $\rho \in (0,1 ]$ and $\lambda \geq  s^{-1/2} b(\bbeta)$.
Conditioned on event $\{   \| \bw(\bbeta) \|_\infty + \epsilon \leq 0.5 \rho \lambda  \}$,  any $\epsilon$-optimal solution $\wt \bbeta$ to \eqref{general.lasso2} satisfies $\wt \bbeta \in \bbeta+ \CC(l)$, where $l  =( 2+  \frac{2}{\rho} )   k^{1/2} + \frac{2}{\rho} s^{1/2}$.
Moreover,  let $r, \kappa >0$ satisfy
$$
   r >  \kappa^{-1}  (    0.5 \rho  k^{1/2}    +  2 s^{1/2}  )  \lambda  .
$$
Then, conditioned on the event $\cE_1(r,l,\kappa) \cap \{   \| \bw(\bbeta) \|_\infty + \epsilon \leq 0.5 \rho \lambda  \}$,  
\#
	 \| \wt \bbeta - \bbeta \|_{2}  &  \leq 
\kappa^{-1}   \bigl\{ \| \blambda_{\cS} \|_2  +  \|  \bw(\bbeta)_{  \cE}    \|_2 +  k^{1/2} \epsilon \bigr\}	 + \kappa^{-1} b(\bbeta)      \label{det.error.bound1} \\
 & \leq  \kappa^{-1}   \big(     0.5 \rho  k^{1/2}    +  2 s^{1/2}  \big) \lambda < r   . \label{det.error.bound2}
\#
\end{lemma}

\begin{proof}[Proof of Lemma~\ref{lem:deterministic.error.bound}]
 For some $r>0$ to be specified,  define  $\eta =  \sup\{ u \in [0,1] : (1-u)\bbeta + u  \wt \bbeta \in \BB(r) \}$, where $\BB(r) = \{ \bdelta \in \RR^d : \| \bdelta \|_2 \leq r\}$.  Note that $\eta = 1$ if $\wt \bbeta \in \bbeta+  \BB(r)$ and $\eta \in (0,1)$ otherwise.  
 Then, the intermediate estimate $\wt \bbeta_\eta := \eta \wt \bbeta + (1-\eta) \bbeta$ satisfies (i) $\wt \bbeta_\eta \in  \bbeta + \BB(r)$, (ii) $\wt \bbeta_\eta$ lies on the boundary of $\bbeta+ \BB(r)$ with $0<\eta <1$ if $\wt \bbeta  \notin \bbeta +\BB(r)$, and (iii) $\wt \bbeta_\eta = \wt \bbeta$ with $\eta=1$ if $\wt \bbeta \in \bbeta +  \BB(r)$.

By the convexity of Huber loss and Lemma~F.2 in \cite{FLSZ2018}, we have
\#
	\langle \nabla \hat  \cL_\tau(\wt \bbeta_\eta ) - \nabla \hat    \cL_\tau(\bbeta ) ,  \wt \bbeta_\eta - \bbeta   \rangle \leq \eta \langle \nabla   \hat   \cL_\tau(\wt \bbeta  ) - \nabla  \hat    \cL_\tau(\bbeta ) ,  \wt \bbeta  - \bbeta \rangle . \label{basic.inequality}
\#
First we bound the left-hand side of \eqref{basic.inequality} from below.
Conditioned on the stated event,  Lemma~\ref{lem:l1cone} indicates 
$$
  \| (\wt \bbeta - \bbeta )_{ \cE^{{\rm c}}} \|_1 \leq  (1+ 2/\rho)  \| (\wt \bbeta - \bbeta )_{\cE } \|_1   +  2 ( \rho \lambda)^{-1} b (\bbeta)  \| \wt \bbeta - \bbeta \|_2 ,
$$
from which it follows that $ \| \wt \bbeta - \bbeta  \|_1   \leq  (2+\frac{2}{\rho})   k^{1/2} \| \wt \bbeta - \bbeta  \|_2+  \frac{2}{\rho}   \lambda^{-1}  b(\bbeta)  \| \wt \bbeta - \bbeta \|_2$.  Provided that $\lambda \geq s^{-1/2} b(\bbeta)$, 
this implies $\wt \bbeta \in  \bbeta+ \CC(l)$ with $l= (2+\frac{2}{\rho})  k^{1/2} +   \frac{2}{\rho  }   s^{1/2}$. Since $\wt \bbeta_\eta - \bbeta = \eta(\wt \bbeta - \bbeta )$, we have $\wt \bbeta_\eta \in  \bbeta + \BB(r) \cap \CC(l ) $ and conditioned on event $\cE_1(r, l, \kappa)$,
\#
	\langle \nabla  \hat \cL_\tau(\wt \bbeta_\eta ) - \nabla 
 \hat  \cL_\tau(\bbeta ) ,  \wt \bbeta_\eta - \bbeta  \rangle \geq \kappa   \| \wt \bbeta_\eta - \bbeta  \|_2^2 .  \label{lbd.1}
\#

Next we upper bound the right-hand side of \eqref{basic.inequality}. For any $\bxi    \in \partial \| \wt \bbeta  \|_1$, write
\#
	& \langle \nabla \hat  \cL_\tau(\wt \bbeta  ) - \nabla  \hat \cL_\tau(\bbeta ) ,  \wt \bbeta  - \bbeta  \rangle  \nn \\
	& = \langle \bu ,  \wt \bbeta  - \bbeta \rangle   - \langle  \blambda \circ \bxi ,  \wt \bbeta   - \bbeta   \rangle   - \langle   \nabla \hat  \cL_\tau(\bbeta ) ,  \wt \bbeta  - \bbeta  \rangle  \nn \\
	& := \Pi_1 - \Pi_2 - \Pi_3, \label{ubd.1}
\#
where  $\bu =  \nabla \hat  \cL_\tau(\wt \bbeta )  +  \blambda \circ \bxi   $.
For $\Pi_3 = \langle   \nabla  \hat  \cL_\tau(\bbeta ) ,  \wt \bbeta  - \bbeta  \rangle$, by the decomposition $\nabla \hat \cL_\tau(\bbeta )= \bw(\bbeta) + \nabla \cL_\tau(\bbeta)$ we have
\#
	| \Pi_3 | & \leq  \|  \bw(\bbeta)_{  \cE}   \|_2  \| (\wt \bbeta - \bbeta  )_{ \cE}  \|_2 \nn \\
	&~~~~  +  \|   \bw(\bbeta)_{ \cE^{{\rm c} }}   \|_{\infty}  \| (\wt \bbeta - \bbeta  )_{  \cE^{{\rm c}}}  \|_1  + b(\bbeta) \| \wt \bbeta - \bbeta \|_2  .  \label{ubd.2}
\#
Turning to $\Pi_2$, decompose $\blambda \circ \bxi$ and $\wt \bbeta  - \bbeta $ according to $ \cS  \cup ( \cE \setminus \cS )  \cup  \cE^{\cc}$ to reach
$$
	\Pi_2 = \langle  (\blambda \circ \bxi)_{\cS}   ,   ( \wt \bbeta  - \bbeta )_{\cS}  \rangle  + \langle  (\blambda \circ \bxi)_{\cE \setminus \cS}   ,   ( \wt \bbeta  - \bbeta )_{\cE \setminus \cS}  \rangle  + \langle (\blambda \circ \bxi)_{ \cE^{\cc}}   , ( \wt \bbeta  - \bbeta )_{\cE^{\cc}}  \rangle .
$$
Since $\bbeta_{\cE^{\cc}}   = \textbf{0}$ and $\bxi \in  \partial \| \wt \bbeta  \|_1$, we have $\langle (\blambda \circ \bxi)_{\cE^{\cc}}   , ( \wt \bbeta  - \bbeta  )_{\cE^{\cc}}  \rangle = \langle \blambda_{\cE^{\cc}}, | \wt \bbeta_{\cE^{\cc}} | \rangle =  \langle \blambda_{\cE^{\cc}}, | (\wt \bbeta - \bbeta )_{\cE^{\cc}} | \rangle $. 
Also, $ \langle  (\blambda \circ \bxi)_{\cE \setminus \cS}   ,   ( \wt \bbeta  - \bbeta )_{\cE \setminus \cS}  \rangle = \langle  (\blambda \circ \bxi)_{\cE \setminus \cS}   ,   \wt \bbeta_{\cE \setminus \cS}  \rangle \geq 0$. Therefore,
\# \label{ubd.3}
	\Pi_2 \geq -  \| \blambda_{\cS} \|_2  \| (\wt \bbeta - \bbeta  )_{\cS} \|_2 +  \| \blambda_{\cE^{\cc} } \|_{\min}   \| (\wt \bbeta - \bbeta  )_{\cE^{\cc}} \|_1 .
\#
Similarly, $\Pi_1$ satisfies the bound
\# \label{ubd.4}
	| \Pi_1 | \leq \| \bu_{  \cE} \|_2 \| (\wt \bbeta - \bbeta )_{ \cE} \|_2 + \| \bu  \|_{\infty} \| (\wt \bbeta - \bbeta )_{  \cE^{\cc}} \|_1.
\#
Together, \eqref{ubd.1}--\eqref{ubd.4} yield
\#
	& \langle \nabla \hat \cL_\tau(\wt \bbeta  ) - \nabla  \hat \cL_\tau(\bbeta ) ,  \wt \bbeta  - \bbeta  \rangle  \nn \\
& \leq  -  \{  \| \blambda_{\cE^{\cc} } \|_{\min}  -  \|  \bw(\bbeta) \|_\infty \}  \| ( \wt \bbeta - \bbeta  )_{\cE^{\cc}} \|_1  + \| \bu  \|_{\infty} \| (\wt \bbeta - \bbeta )_{ \cE^{\cc}} \|_1 \nn \\
& ~\quad +  \{   \|  \bw(\bbeta) _{  \cE}   \|_2 +  \| \bu_{   \cE } \|_2  \} \| (\wt \bbeta - \bbeta  )_{  \cE}  \|_2 
+  \| \blambda_{\cS} \|_2 \| (\wt \bbeta - \bbeta )_{\cS} \|_2  + b(\bbeta) \| \wt \bbeta - \bbeta \|_2  . \nn
\#
Taking the infimum over $\bxi \in \partial \| \wt \bbeta \|_1$ on both sides, it follows that
\#
& \langle \nabla  \hat   \cL_\tau(\wt \bbeta  ) - \nabla \hat \cL_\tau(\bbeta ) ,  \wt \bbeta  - \bbeta  \rangle \nn \\
& \leq  -  \{  \| \blambda_{\cE^{\cc} } \|_{\min}  -  \|   \bw(\bbeta)  \|_\infty  - \epsilon  \}  \| ( \wt \bbeta - \bbeta )_{\cE^{\cc}} \|_1   \nn \\
& ~\quad +  \{   \|   \bw(\bbeta) _{   \cE}   \|_2 + k^{1/2} \epsilon \} \| (\wt \bbeta - \bbeta  )_{   \cE}  \|_2 \nn \\
& ~~~~ +  \| \blambda_{\cS} \|_2 \| (\wt \bbeta - \bbeta  )_{\cS} \|_2  + b(\bbeta) \| \wt \bbeta - \bbeta \|_2 . \label{ubd.5}
\#
It follows from \eqref{basic.inequality}, \eqref{lbd.1} and \eqref{ubd.5} that conditioned on  $\cE_1(r, l, \kappa) \cap \{  \|\bw(\bbeta) \|_{\infty } + \epsilon \leq 0.5\rho \lambda \}$,
\#
 \kappa 	 \| \wt \bbeta_\eta - \bbeta  \|_{2}^2   \leq  \{   \| \blambda_{\cS} \|_2+  \| \bw(\bbeta) _{  \cE}    \|_2+   k^{1/2}  \epsilon  \}  \|  \wt \bbeta_\eta - \bbeta  \|_{2}  + b(\bbeta) \| \wt \bbeta_\eta - \bbeta \|_{2}  , \label{l2.error.bound1}
\# 
On the same event,  note that
$$
 \| \bw(\bbeta) _{  \cE}  \|_2 +  k^{1/2}\epsilon  \leq k^{1/2}  \{   \| \bw(\bbeta) \|_\infty + \epsilon  \} \leq 0.5 \rho k^{1/2} \lambda .
$$
Moreover,  recall that $\| \blambda_{\cS} \|_2 \leq s^{1/2} \lambda$ and $b(\bbeta) \leq s^{1/2} \lambda$.
Plugging these bounds into \eqref{l2.error.bound1} yields
\#
 \| \wt \bbeta_\eta - \bbeta  \|_2 \leq   \kappa^{-1}   \{  (  s^{1/2} +  0.5 \rho  k^{1/2} )   \lambda + b(\bbeta) \} \leq   \kappa^{-1}    (  2 s^{1/2} +  0.5 \rho  k^{1/2} )    \lambda < r. \nn
\#
Hence, $\wt \bbeta_\eta$ falls in the interior of $\bbeta + \BB(r)$. Via proof by contradiction, we must have $\eta=1$ and   $\wt \bbeta_\eta = \wt \bbeta$. Consequently, \eqref{det.error.bound1} and \eqref{det.error.bound2} follow, respectively, from \eqref{l2.error.bound1} and the last display.
\end{proof}

 In Lemma~\ref{lem:deterministic.error.bound}, we need $\lambda$ to be sufficiently large in the sense that for some $s$-sparse vector $\bbeta$,
 \$
 	 \lambda \gtrsim \| \bw(\bbeta) \|_\infty~~\mbox{ and }~~ \lambda \geq   s^{-1/2} b(\bbeta),
 \$
 where $\bw(\bbeta)= \nabla \hat \cL_\tau(\bbeta) - \nabla \cL_\tau(\bbeta)$ is the centered gradient at $\bbeta$ and $b(\bbeta) = \| \nabla  \cL_\tau(\bbeta) \|_2$ quantifies the bias.   To prove the weak oracle property (Proposition~\ref{prop:tightening}), we will take $\bbeta =\bbeta^*$ and control the stochastic term $\| \bw^*\|_\infty = \|  \bw(\bbeta^*) \|_\infty$ and bias term $b(\bbeta^*)$ separately.    For some $\bbeta$ which has vanishing or negligible bias,  we will only focus on the stochastic term $\| \nabla \hat \cL_\tau(\bbeta)  \|_\infty$, as described in the next lemma. Recall the event $\cE_2(r, l,\kappa)$ defined in \eqref{RSC.event2}, and  $\BB_{\Sigma}(r) = \{ \bbeta \in \RR^d: \| \bbeta\|_{\Sigma} \leq r\}$.

\begin{lemma} \label{lem:rsc.oracle}
Consider some $\bbeta \in \bbeta^* + \BB_{\Sigma}(r)$ ($r>0$) satisfying $\bbeta_{ \cS^{\cc}} = \textbf{0}$,  and  let $\cE \subseteq [d]$ be a subset that contains $\cS$ and has cardinality $ |\cE |=k$.  
Assume  that  $\blambda=(\lambda_1, \ldots, \lambda_d)^\T$ satisfies $\| \blambda \|_\infty \leq \lambda$ and  $\| \blambda_{ \cE^{\cc} } \|_{\min} \geq \rho \lambda   >0$ for some $\rho \in (0,1 ]$.
Conditioned on event $\{   \| \hat \cL_\tau(\bbeta) \|_\infty  \leq 0.5 \rho \lambda  \}$,  any  optimal solution $\hat \bbeta$ to \eqref{general.lasso2} satisfies $\hat \bbeta \in \bbeta+ \CC(l)$, where $l  =( 2+  \frac{2}{\rho} )    k^{1/2} + \frac{2}{\rho} s^{1/2}$.
Moreover,   let $r, \kappa>0$ satisfy $   r >  \kappa^{-1}  (    0.5 \rho  k^{1/2}    +  2 s^{1/2}  )  \lambda  $. Then, conditioned on the event $\cE_2(r,l,\kappa) \cap \{   \| \hat \cL_\tau(\bbeta) \|_\infty  \leq 0.5 \rho \lambda  \}$,
\#
	 \| \hat  \bbeta - \bbeta \|_2   &  \leq 
\kappa^{-1}   \bigl\{ \| \blambda_{\cS} \|_2  +  \|  \hat \cL_\tau(\bbeta)_{  \cE}    \|_2  \bigr\}  . \nn
\#
\end{lemma}

The proof of Lemma~\ref{lem:rsc.oracle} is based on the same arguments from the proof of Lemmas~\ref{lem:l1cone} and \ref{lem:deterministic.error.bound}, and thus is omitted. To prove the strong oracle property (Proposition~\ref{prop:oracle}), we will apply Lemma~\ref{lem:rsc.oracle} with $\bbeta =\hat \bbeta^{{\ora}}$, the oracle estimator defined in \eqref{def:oracle}.

\section{Proofs of Propositions}

\subsection{Proof of Proposition~\ref{prop:contraction}}

With the initial estimate $\wt \bbeta^{(0)} = \textbf{0}_{d}$, we have $\blambda^{(0)} = p_\lambda'( \textbf{0}_d ) = (\lambda, \ldots, \lambda)^\T \in \RR^d$. Then \eqref{step1.bound} follows immediately from Lemma~\ref{lem:deterministic.error.bound} with $\bbeta = \bbeta^*$, $\cE = \cS$ and $\rho =1$. \qed

\subsection{Proof of Proposition~\ref{prop:tightening}}

In order to improve the statistical rate at step $\ell\geq 1$, we need to control the magnitude of the spurious discoveries from the last step, that is, $\max_{j \in \cS^{\cc}} | \wt \beta_j^{(\ell-1)}|$. 
Recall that $\blambda^{(\ell-1)}   = (\lambda_1^{(\ell-1)} , \ldots, \lambda_d^{(\ell-1)} )^\T = ( p_\lambda'( |\wt \beta_1^{(\ell-1)}| )  ,  \ldots  , p_\lambda'( |\wt \beta_d^{(\ell-1)}| ) )^\T$ and $p_\lambda(t) = \lambda^2 p(t/\lambda)$ for $t\in \RR$. Intuitively, the larger $ | \wt \beta_j^{(\ell-1)}|$ is, the smaller $\lambda^{(\ell-1)}_j$ is.
Motivated by this observation, we construct an augmented set $\cE_\ell$ of $\cS$ in each step and control the magnitude of $\| \blambda^{(\ell-1)}_{\cE_\ell^{{\cc}}} \|_{\min}$.

Starting from $\wt \bbeta^{(0)} = \textbf{0}$, we have $\blambda^{(0)} = (\lambda, \ldots, \lambda)^\T \in \RR^d$.
Recall from \eqref{lambda.kappa} that $\lambda \geq s^{-1/2} b^*_\tau$,  or equivalently, $\lambda^{-1} b^*_\tau   \leq s^{1/2}$.
Then, applying Lemma~\ref{lem:deterministic.error.bound} with $\cE= \cS$ and $l_0 = 6 s^{1/2}$ we obtain that,  conditioning on $\cE_1(r, l_0 , \kappa )\cap \{ \lambda \geq \frac{2}{p'( \gamma)}  ( \|  \bw^*\|_\infty + \epsilon_1 ) \}$,
\#
	\| \wt \bbeta^{(1)} - \bbeta^* \|_2  & \leq   \kappa^{-1} \big\{       \big(  \| \blambda_{\cS}^{(0)} \|_2 + \| \bw^*_{  \cS } \|_2 + s^{1/2} \epsilon_1 \big)  +  b^*_\tau   \big\}   \nn \\
	& \leq \kappa^{-1}   \{ 1 +  0.5   p'( \gamma)     \}   s^{1/2}\lambda +  \kappa^{-1} b^*_\tau  \nn \\
	& \leq \kappa^{-1}   \{ 2 +  0.5   p'( \gamma)     \}    s^{1/2}\lambda   ,   \label{bound.step1} 
\# 
where the last inequality is due to \eqref{lambda.kappa}.
For $\ell\geq 1$, define the augmented set
\#
	\cE_\ell = \cS \cup  \{1 \leq   j \leq d :  \lambda_j^{(\ell-1)}  < p'( \gamma ) \lambda  \} ,  \label{def:El}
\#
which depends on the solution $\wt \bbeta^{(\ell-1)}$ from the previous step. We claim that the above constructed sets satisfy
\#
	 | \cE_\ell | < (c^2 +1 )s ~~\mbox{ and }~~  \| \blambda^{(\ell-1)}_{\cE_\ell^{\cc}} \|_{\min} \geq  p'( \gamma) \lambda   ,  \label{scaling.stepl}
\#
where $c$ is the constant determined by \eqref{def.c}.  If these were true, 
it follows from Lemma~\ref{lem:deterministic.error.bound} with $\rho=p'( \gamma)$, $ k = (c^2+1)s$ and $l = (  2 + \frac{2}{ \rho})  (c^2+1)^{1/2} s^{1/2} + \frac{2}{ \rho }s^{1/2} $ that, conditioned on $\cE_1(r, l  , \kappa )\cap \{ \lambda \geq \frac{2}{p'( \gamma)}  ( \|  \bw^*\|_\infty + \epsilon_1 ) \}$,
\#
	\| \wt \bbeta^{(\ell)} - \bbeta^* \|_2  & \leq   \kappa^{-1} \big(  \| \blambda_{\cS}^{(\ell-1)} \|_2 + \|\bw^*_{  \cE_\ell } \|_2 + |   \cE_\ell|^{1/2} \epsilon_\ell   \big)  + \kappa^{-1} b^*_\tau  \label{bound.stepl} \\
	& <   \kappa^{-1} \big\{ 1 + 0.5  (c^2+1)^{1/2} p'(\gamma)  \big\} s^{1/2} \lambda + \kappa^{-1} s^{1/2} \lambda    \nn  \\
& =    \kappa^{-1} \big\{2 + 0.5  (c^2+1)^{1/2} p'(\gamma)  \big\} s^{1/2} \lambda   \nn \\
	& =  \underbrace{  c \gamma  s^{1/2} \lambda  }_{=: r^{{\rm crude}}}\leq r ,     \label{crude.bound.stepl}
\#
where the last two steps follow from \eqref{def.c} and  \eqref{constraint.r}.

We prove the earlier claim \eqref{scaling.stepl}  by the method of induction. For $\ell=1$, we have $\blambda^{(0)} = (\lambda, \ldots, \lambda)^\T \in \RR^{d}$. Thus, \eqref{scaling.stepl} holds with $\cE_1 =\cS$. Next, assume \eqref{scaling.stepl} holds for some $\ell\geq 1$, from which \eqref{crude.bound.stepl} follows. To bound the cardinality of $\cE_{\ell+1}$, note that for any $j\in \cE_{\ell+1} \setminus \cS$, $ p'_\lambda( | \wt \beta_j^{(\ell)}|  )  = \lambda_j^{(\ell)}   <  p'( \gamma ) \lambda = p_\lambda'( \gamma \lambda ) $. This, together with the monotonicity of $p_\lambda'$ on $\RR_+$, implies $| \wt \beta_j^{(\ell)}|  >  \gamma \lambda$.
Recalling that $\beta^*_j = 0$ for $j\in  \cE_{\ell+1} \setminus \cS$, we obtain
\#
	 | \cE_{\ell+1} \setminus \cS |^{1/2}  & <  \frac{ 1 }{   \gamma \lambda } \| \wt \bbeta^{(\ell )}_{ \cE_{\ell+1} \setminus \cS} \|_2 =\frac{ 1 }{   \gamma \lambda } \|  ( \wt \bbeta^{(\ell )} - \bbeta^* )_{ \cE_{\ell+1} \setminus \cS} \|_2  \nn \\
	&  \stackrel{({\rm i})}{\leq }   \frac{  c \gamma s^{1/2}  \lambda }{  \gamma \lambda  }   =   c  s^{1/2}, \label{set.bound}
\#
where inequality (i) applies the bound \eqref{crude.bound.stepl}.
Hence, $|\cE_{\ell+1}| \leq | \cS | + |\cE_{\ell+1} \setminus \cS| < (c^2+1)s$. By \eqref{def:El} and the property $p_\lambda'(t) = \lambda p'(t/\lambda)$, we are guaranteed that
$$
	\lambda^{(\ell)}_j  \geq p'( \gamma) \lambda \geq  2 ( \|  \bw^* \|_\infty + \epsilon_{\ell + 1} ) ~\mbox{ for }~ j \in \cE_{\ell+1}^{\cc}.
$$
The two hypotheses in \eqref{scaling.stepl} then hold for $\ell+1$, which completes the induction step. Consequently,  the bounds \eqref{bound.stepl} and \eqref{crude.bound.stepl} hold for any $\ell\geq 1$.
 
We have shown that under proper conditions, all the estimates $\wt \bbeta^{(\ell)}$ fall in a local neighborhood of $\bbeta^*$, i.e.,$\| \wt \bbeta^{(\ell)} - \bbeta^* \|_{\Sigma} \leq r^{{{\rm crude}}} =  c \gamma s^{1/2} \lambda$. To further refine this bound as signal strengthens,  on the right-hand side of  \eqref{bound.stepl}, we need to establish sharper bounds on 
$$
	\| \blambda_{\cS}^{(\ell-1)}\|_2 = \sqrt{ \sum_{j\in \cS} \{ \lambda_j^{(\ell -1)} \}^2 } ~~\mbox{ and }~~ \|\bw^*_{  \cE_\ell } \|_2 +  |  \cE_\ell|^{1/2} \epsilon_\ell , 
$$ 
and maintain the bias term $b^*_\tau $,  instead of replacing it with an upper bound $s^{1/2} \lambda$.
For each $j\in [d]$, $\lambda_j^{(\ell-1)} = p'_\lambda( | \wt \beta_j^{(\ell-1)} |)$. If $|  \wt \beta_j^{(\ell-1)}  - \beta_j^*  | \geq    \gamma \lambda$, then $\lambda_j^{(\ell-1)} \leq \lambda \leq   \gamma^{-1}|  \wt \beta_j^{(\ell-1)}  - \beta_j^*  |$; otherwise if $|  \wt \beta_j^{(\ell-1)}  - \beta_j^*  | \leq    \gamma \lambda$, $\lambda_j^{(\ell-1)} \leq p_\lambda'(|\beta_j^*| -   \gamma \lambda )$ due to monotonicity of $p_\lambda'$.
Putting together the pieces, we conclude that 
\#
	\| \blambda_{\cS}^{(\ell-1)}\|_2 \leq  \|  p_\lambda'(|\bbeta^*_\cS| -    \gamma \lambda ) \|_2 +   \gamma^{-1}  \| ( \wt \bbeta^{(\ell-1)} - \bbeta^* )_\cS \|_2 . \nn
\#
For the remaining terms that involve $\cE_\ell $, by  the triangle inequality and \eqref{set.bound} we obtain that
\#
	&   \|\bw^*_{   \cE_\ell } \|_2 +  |  \cE_\ell|^{1/2} \epsilon_\ell   \nn \\
& \leq   \| \bw^*_{  \cS } \|_2 +  s^{1/2} \epsilon_\ell +   |\cE_\ell \setminus \cS|^{1/2} \| \bw^*  \|_\infty +  |\cE_\ell \setminus \cS|^{1/2} \epsilon_\ell \nn \\
& <  \| \bw^*_{  \cS } \|_2 +  s^{1/2}\epsilon_\ell + \frac{  \| \bw^* \|_\infty + \epsilon_\ell }{   \gamma \lambda}   \|  ( \wt \bbeta^{(\ell-1 )} - \bbeta^* )_{ \cE_{\ell } \setminus \cS} \|_2  \nn \\
&  \leq \| \bw^*_{  \cS } \|_2 +  s^{1/2}\epsilon_\ell +  \frac{p'(\gamma)}{ 2 \gamma}  \|  ( \wt \bbeta^{(\ell -1)} - \bbeta^* )_{ \cE_{\ell } \setminus \cS} \|_2 . \nn
\#
Plugging the above refined bounds into \eqref{bound.stepl} yields
\#
 & \| \wt \bbeta^{(\ell)} - \bbeta^* \|_2  \nn \\
& \leq \kappa^{-1}  \big\{  \|  p_\lambda'(|\bbeta^*_\cS| -    \gamma \lambda ) \|_2 +  \| \bw^*_{  \cS} \|_2 +  s^{1/2} \epsilon_\ell  + b^*_\tau  \big\}     +  \frac{ p'(\gamma) }{  2   \gamma   \kappa  } \| ( \wt \bbeta^{(\ell-1)} - \bbeta^* )_{\cE_\ell} \|_2 . \nn
\#
Taking $\delta =  p'(\gamma)/(  2 \gamma   \kappa ) \in (0,1)$,  the contraction inequality \eqref{contraction.inequality} follows immediately. 
Finally, \eqref{contraction.inequality2} is a direct consequence of \eqref{contraction.inequality} and \eqref{bound.step1}. \qed

\subsection{Proof of Proposition~\ref{prop:oracle}}
 
By construction, the oracle estimator $\hat \bbeta^{{\rm ora}}$ is such that $\hat \bbeta^\ora_{\cS^\cc} = \textbf{0} \in \RR^{d-s}$  and $\nabla \hat  \cL_\tau(\hat \bbeta^{{\rm ora}})_{  \cS} = \textbf{0} \in \RR^s$.
With $\bw^{\ora}  = \nabla \hat  \cL_\tau(\hat \bbeta^\ora)$,  the proof strategy is similar to that in the proof of Proposition~\ref{prop:tightening} with $\epsilon_\ell=0$, because $\hat \bbeta^{(\ell)}$ are optimal solutions to $({\rm P}_\ell)$. 

Recall that  $l =  \{2 +  \frac{2}{p'(\gamma_0)}  \}  (c_0^2+1)^{1/2}  s^{1/2}$ with $c_0>0$ determined by \eqref{def.c0}.
Conditioned on the event
$$
	\big\{ \| \bw^{{\rm ora}} \|_\infty \leq 0.5 p'(\gamma_0)  \lambda  \big\} \cap \big\{ \| \hat \bbeta^{{\rm ora}}  - \bbeta^* \|_{\Sigma}  \leq r \big\} \cap \cE_2(r,    l ,  \kappa) ,
$$
following the proof of Proposition~\ref{prop:tightening} and applying Lemma~\ref{lem:rsc.oracle} with $\bbeta = \hat \bbeta^{\ora}$ and $\cE = \cE_\ell$, it can be similarly shown that 
\#
	\| \hat \bbeta^{(\ell)} - \hat \bbeta^{{\rm ora}} \|_{2} &  \leq   \kappa^{-1}  \big( \| \blambda^{(\ell-1)}_{\cS} \|_2  +    \| \bw^{\ora} _{  \cE_\ell} \|_2  \big)  \nn \\
	& \leq    \kappa^{-1}  \big(  s^{1/2} \lambda   +   |\cE_\ell|^{1/2} \| \bw^{{\rm ora}} \|_\infty  \big)  \nn \\
	 & <      \kappa^{-1} \big\{1 + 0.5 p'(\gamma_0) (c_0^2+1)^{1/2} \big\} s^{1/2}\lambda  =   c_0 \gamma_0   s^{1/2}\lambda  \leq r , \label{oracle.bound1}
\#
where similarly to \eqref{def:El} and \eqref{scaling.stepl}, $\cE_\ell = \cS \cup   \{1 \leq   j \leq d :  \lambda_j^{(\ell-1)}  < p'( \gamma_0 ) \lambda  \}$ is such that $|\cE_\ell \setminus \cS | < c_0^2 s$ and thus  $|\cE_\ell|  < (c_0^2+1)s$.
In this case, the approximation bias is hidden in $\| \bw^{\ora} \|_\infty$. Moreover, define a sequence of subsets 
$$
	\cS_\ell  = \{ 1\leq  j \leq d :  | \hat \beta_j^{(\ell)} - \beta_j^* | \geq  \gamma_0  \lambda \}, \ \ \ell =0, 1, 2, \ldots .
$$
Starting with $\hat  \bbeta^{(0)} = \textbf{0}$, it holds under the minimum signal strength condition that $\cS_0 = \cS$.

To obtain a refined upper bound on $\| \blambda^{(\ell-1)}_{\cS} \|_2 $, note that if $j\in  \cS \cap  \cS_{\ell -1}^{\cc}  $, $\lambda_j^{(\ell-1)}  = p_\lambda' ( | \hat \beta^{(\ell-1)}_j  | )   \leq p_\lambda'(|\beta_j^* | -    \gamma_0 \lambda  )$ due to monotonicity; otherwise if $j\in \cS \cap \cS_{\ell -1}$, $\lambda_j^{(\ell-1)}  \leq \lambda$. Therefore,
\#
	\| \blambda^{(\ell-1)}_{\cS} \|_2 \leq  \| p_\lambda'(|\bbeta^*_{\cS}|  - \gamma_0 \lambda  ) \|_2 + \lambda |\cS \cap \cS_{\ell - 1}|^{1/2} .  \nn
\#
Since $\| \bbeta_{\cS}^* \|_{\min} \geq (\gamma_0 + \gamma_1) \lambda$ and $p_\lambda'(t)=0$ for all $t \geq \gamma_1 \lambda$, $\| p_\lambda'(|\bbeta^*_{\cS}|  -\gamma_0 \lambda  ) \|_2$ vanishes. Turning to $ \| \bw^\ora_{ \cE_\ell }  \|_2$,  by the first-order condition of minimizing $\bbeta_{\cS} \mapsto (1/n) \sn \hat \cL_\tau(y_i - \bx_{i,\cS}^\T \bbeta_{\cS} )$, we have $\bw^\ora_{  \cS} = \textbf{0}$ and hence
$$
 \| \bw^\ora_{  \cE_\ell } \|_2  = \|\bw^\ora_{\cE_\ell \setminus \cS} \|_2 \leq \| \bw^\ora  \|_\infty |\cE_\ell \setminus \cS |^{1/2} .
$$
For each $j\in  \cE_\ell \setminus \cS$, $\beta^*_j = 0 $ and $\lambda_j^{(\ell-1)} = p_\lambda'(|\hat \beta_j^{(\ell-1)}| ) < p'(\gamma_0) \lambda =  p_\lambda'(\gamma_0 \lambda )$. Hence, $| \hat \beta_j^{(\ell-1)} - \beta^*_j| =| \hat \beta_j^{(\ell-1)} |  >  \gamma_0 \lambda$ so that $j\in   \cS_{\ell-1} \setminus \cS$. Therefore, $\cE_\ell \setminus \cS \subseteq \cS_{\ell -1} \setminus \cS$. Combined with the earlier bound, we arrive at
\#
	 \| \bw^\ora_{   \cE_\ell} \|_2  \leq   \|  \bw^\ora \|_\infty |  \cS_{\ell -1} \setminus \cS |^{1/2}   . \nn
 \#
Since $p'(\gamma_0) \leq 1/2$, substituting the above estimates into \eqref{oracle.bound1} yields
 \#
	\| \hat \bbeta^{(\ell)} - \hat \bbeta^{{\rm ora}} \|_{2} &  \leq \frac{  |\cS \cap \cS_{\ell - 1}|^{1/2} +      | \cS_{\ell -1} \setminus \cS |^{1/2}  /4   }{   \kappa  }   \lambda   \nn \\
	& \leq  \frac{\sqrt{17}}{4} \frac{   \lambda }{  \kappa } | \cS_{\ell - 1} |^{1/2}     .    \label{oracle.bound2}
\#

Next we bound $| \cS_{\ell} |$ $(\ell\geq 1)$, the cardinality of $\cS_{\ell}$.
By \eqref{oracle.constraint}, it holds  for any $j\in \cS_\ell$  that
\#
	 | \hat \beta_j^{(\ell)} - \hat \beta^{{\rm ora}}_{ j} | &  \geq    \gamma_0  \lambda - \| \hat \bbeta^{{\rm ora}} -\bbeta^*\|_\infty    \geq   1.25 \frac{\lambda}{ \delta   \kappa}  -   0.2 \frac{ \lambda}{ \delta     \kappa}    = 1.05 \frac{\lambda}{\delta   {\kappa}} . \nn
\#
In conjunction with \eqref{oracle.bound2}, this implies 
\#
	| \cS_\ell |^{1/2} \leq  \frac{  \| \hat \bbeta^{(\ell)} - \hat \bbeta^{{\rm ora}} \|_2 }{ 1.05 \lambda /( {\kappa}\delta) }   \leq   \frac{\frac{\sqrt{17}}{4} (\lambda/  \kappa) |\cS_{\ell -1 }|^{1/2}     }{  1.05 \lambda /(  {\kappa}\delta) }  <   \delta | \cS_{\ell -1} |^{1/2}      , \ \ \ell \geq 1.
\#
Recall that $\cS_0 = \cS$, we have  $| \cS_\ell |^{1/2}  <  \delta^\ell s^{1/2}$ for any $\ell\geq 1$.
As long as $\ell\geq T:=\lceil \log (s^{1/2})/\log(1/\delta) \rceil$, we are guaranteed that $| \cS_\ell |<1$, i.e. $\cS_\ell =\emptyset$. Consequently, it follows from \eqref{oracle.bound2} that $ \hat \bbeta^{(\ell)} = \hat \bbeta^{{\rm ora}} $ for all $\ell \geq T+1$. This completes the proof.  \qed

\subsection{Proof of Proposition~\ref{prop:RSC}}

The proof is based on similar  arguments that were used in the proof of Lemmas~C.3 and C.4  in \cite{SZF2017}. We only present the necessary steps in order to slightly relax the sub-Gaussian condition on $\bx_i = (x_{i1} , \ldots, x_{id})^\T$.

For any $\bbeta \in \bbeta^* + \BB(r)\cap \CC(l)$,  write $\bdelta = \bbeta - \bbeta^* $. 
Following the proof of Lemma~C.3 in \cite{SZF2017},  it can be shown under Condition~\ref{moment.cond} that 
\#
  & \langle \nabla \hat \cL_\tau(\bbeta) -  \nabla \hat \cL_\tau(\bbeta^*)  , \bbeta - \bbeta^* \rangle \nn \\
& \geq   \big\{ 1   -  (2 \sigma_2 / \tau)^2  \big\} \| \bdelta  \|_{\Sigma}^2   -   \EE (\bx^\T \bdelta )^2 I\{|\bx^\T \bdelta |  / \| \bdelta \|_2 \geq \tau  / (4r)\} - \Delta(r, l) \cdot   \| \bdelta \|_2^2 \nn \\
& \geq     \big\{ 1   -  (2 \sigma_2 / \tau)^2 -  \rho_{\bx}^2 e^{-\tau/(8 \sigma_{\bx} r )}  \big\} \cdot  \| \bdelta  \|_{\Sigma}^2   - \Delta(r, l) \cdot   \| \bdelta \|_2^2  ,   \label{bregman.div.lbd1}
\#
where the second inequality uses the bound 
\#
\EE (\bx^\T \bdelta )^2 I\{|\bx^\T \bdelta |  / \| \bdelta \|_2 \geq \tau  / (4r)\}  &\leq \big\{ \EE (\bx^\T \bdelta )^4 \big\}^{1/2} \PP\big(|\bx^\T \bdelta |  / \| \bdelta \|_2 \geq \tau  / (4r)\big)^{1/2} \nn \\
& \leq \rho_{\bx}^2  e^{-\tau/(8\sigma_{\bx} r)} \| \bdelta \|_{\Sigma}^2 \nn
\#
and $\Delta(r, l )  = \sup_{\bdelta  \in  \BB(r) \cap \CC(l) }  (1/n) \sn |f_{\bdelta}(\bx_i, \varepsilon_i) - \EE f_{\bdelta}(\bx_i, \varepsilon_i) |$ with
\$
	f_{\bdelta} (\bx_i, \varepsilon_i ) : =  I(|\varepsilon_i | \leq \tau/2) \cdot  \varphi_{\tau\| \bdelta \|_2  / (2r)}  (\bx_i^\T \bdelta )/ \| \bdelta \|_2^2  
\$
and $\varphi_c(u) := u^2 I(|u|\leq c/2 ) + (|u| - c)^2 I(c/2 < |u| \leq c )$ for $u \in \RR$ and $c \geq 0$. 
We thus let $\tau \geq     \max\{ 4 \sqrt{2} \sigma_2 ,  8 \log(8 \rho_{\bx}^2 ) \sigma_{\bx}  r \}$ so that $(2 \sigma_2 / \tau)^2 \leq 1/8$ and $ \rho_{\bx}^2 e^{-\tau/(8 \sigma_{\bx} r )}  \leq 1/8$.
It then follows from \eqref{bregman.div.lbd1} that 
\#
\langle \nabla \hat \cL_\tau(\bbeta) -  \nabla \hat \cL_\tau(\bbeta^*)  , \bbeta - \bbeta^* \rangle \geq \frac{3}{4}  \cdot \| \bdelta \|_{\Sigma} ^2 - \Delta(r, l) \cdot \|\bdelta \|_2^2  \label{RSC.lb1}
\#
holds uniformly over $\bbeta \in \bbeta^* + \BB(r) \cap \CC(l)$.

It remains to control the supremum $\Delta(r,l)$.  Note that $\varphi_c(\cdot)$ is $c$-Lipschitz continuous, and satisfies $\varphi_{b c }(b u) =b^2 \varphi_{c}(u)$ for $b, c >0$ and $u\in \RR$.  Thus,   the above $f_{\bdelta} (\bx_i, \varepsilon_i ) $ can be simplified as
\$
f_{\bdelta} (\bx_i, \varepsilon_i ) = I(|\varepsilon_i | \leq \tau/2 ) \cdot  \varphi_{\tau/ (2r)}  (\bx_i^\T \bdelta /\| \bdelta \|_2 ).
\$
Moreover, since $0\leq \varphi_c(u) \leq \min \{  (c/2)^2, u^2 \}$, we have $0\leq f_{\bdelta} (\bx_i, \varepsilon_i )  \leq \tau^2/(4r)^2$ and by \eqref{def.rhox},  $\EE f^2_{\bdelta} (\bx_i, \varepsilon_i ) \leq \EE (\bx_i^\T \bdelta/\| \bdelta \|_2 )^4 \leq \rho_{\bx} ^4 \{ \EE (\bx_i^\T \bdelta/\| \bdelta \|_2 )^2 \}^2 \leq  ( \rho_u  \rho_{\bx}^2)^2$ for all $\bdelta \in \RR^d$.
Then, applying Bousquet's version of Talagrand's inequality (see, e.g.  Theorem~7.3 in \cite{B2003}), we obtain that for every $t \geq 0$,
\#
	\Delta(r,l)  & \leq  \EE \Delta(r,l) + \sqrt{  \EE \Delta(r,l)   (\tau/ 2r)^2 t /n  +  2(  \rho_u  \rho_{\bx}^2 )^2 t /n   }  +  (\tau / 4r)^2  t / (3 n)  \nn \\
	& \leq 1.25  \EE \Delta(r,l)  +      \rho_u  \rho_{\bx}^2  \sqrt{2  t/n } + (\tau/ r)^2 t/ (3n)  \label{RSC.lb2}
\#
with probability at least $1- e^{-t}$.  For  $\EE \Delta(r,l)$,  using Rademacher symmetrization gives
\#
 \EE\Delta(r,l)   \leq 2 \EE \Bigg\{  \sup_{ \bdelta  \in  \BB(r) \cap \CC(l) }  \Bigg|   \frac{1}{n} \sn  e_i   f_{\bdelta} (\bx_i, \varepsilon_i )   \Bigg|  \Bigg\}  , \nn
\# 
where $e_1,\ldots, e_n$ are independent Rademacher random variables.
Since $\varphi_c(\cdot)$ is $c$-Lipshitz,  $ f_{\bdelta} (\bx_i, \varepsilon_i ) $ is a $(\tau /2r)$-Lipschitz function in $  \bx_i^\T \bdelta / \| \bdelta \|_2$, i.e., for any sample $(\bx_i, \varepsilon_i)$ and parameters $\bdelta , \bdelta'\in \RR^d$, 
\#
	|  f_{\bdelta} (\bx_i, \varepsilon_i ) - f_{\bdelta'} (\bx_i, \varepsilon_i ) |  \leq  \frac{\tau}{2r}    |   \bx_i^\T \bdelta / \| \bdelta \|_2 -  \bx_i^\T \bdelta' / \| \bdelta' \|_2  |  .\nn
\#
Moreover, observe that $ f_{\bdelta} (\bx_i, \varepsilon_i )=0$ for any $\bdelta$ such that $ \bx_i^\T \bdelta / \|\bdelta \|_{2} = 0$, and $I(|\varepsilon_i | \leq \tau/2 ) \in \{0, 1\}$.
Then, applying Talagrand's contraction principle (see, e.g. Theorem~4.4, Theorem~4.12 and (4.20) in \cite{LT1991}) yields
\#
 & \EE \Delta(r,l)     \leq 2 \EE \Bigg\{  \sup_{ \bdelta  \in \BB(r) \cap \CC(l) }   \Bigg|  \frac{1}{n} \sn  e_i   f_{\bdelta}(\bx_i, \varepsilon_i)   \Bigg|   \Bigg\} \nn \\
& \leq    \frac{2\tau }{r}\EE \Bigg\{  \sup_{ \bdelta  \in  \BB(r) \cap \CC(l)   }   \Bigg|  \frac{1}{n} \sn  e_i      \bx_i^\T \bdelta / \| \bdelta \|_2 \Bigg|  \Bigg\} \leq  \frac{2\tau  l }{r} \cdot  \EE  \bigg\|  \frac{1}{n} \sn e_i   \bx_i   \bigg\|_\infty  , \label{Delta.meanbound}
\# 
where  the last inequality uses the cone constraint that $\|\bdelta \|_1 \leq   l  \| \bdelta \|_{2}$.
Next, we apply a maximal inequality for sub-exponential random variables to bound the last term on the right-hand side of \eqref{Delta.meanbound}. 
For  $j=1,\ldots, d$, define partial sums  $S_j =  \sn  e_i   x_{ij}$, of which each summand satisfies  $\EE(  e_i     x_{ij} ) =0$ and $\EE(e_i x_{ij})^2 = \sigma_{jj}$.  More over, for $k=2,4,\ldots$, $\EE |e_i|^k=1$ and 
\#
	\EE |e_i     x_{ij}|^k  &\leq  \sigma_{\bx}^k   \cdot k \int_0^\infty t^{k-1} \PP(|x_{ij} | \geq \sigma_{\bx} t) \, {\rm d} t   \leq  \sigma_{\bx}^k     \cdot k \int_0^\infty t^{k-1} e^{-t} \, {\rm d} t \nn \\
	& =  k!   \sigma_{\bx}^k = \frac{k!}{2} \underbrace{ 2 \sigma_{\bx}^2 }_{\geq \sigma_{jj}} \sigma_{\bx}^{k-2}   .\nn
\#
By the symmetry of Rademacher random variables, and Bernstein's inequality (see, e.g.  Theorems~2.10 in \cite{BLM2013}), we obtain that 
$$
	 \log \EE e^{ \lambda S_j}  =  \log \EE e^{ - \lambda S_j}  \leq \psi(\lambda) :=  \frac{\nu \lambda^2 }{2(1- c \lambda)} ~~\mbox{ for all }  \lambda  \in (0, 1/c ),
$$
where $\nu = 2n \sigma_{\bx}^2$ and $c = \sigma_{\bx}$.  Following the proof of Theorems 2.5 in \cite{BLM2013},  it can be shown that 
$$
 \EE \max_{1\leq j\leq d} |S_j |   \leq   \inf_{\lambda \in (0,1/c)}  \Bigg\{  \frac{\log(2 d) + \psi(\lambda)}{\lambda} \Bigg\} = \sqrt{2v \log(2 d)} +   c \log(2 d) .
$$ 
Re-arranging terms and using \eqref{Delta.meanbound}, we find that
\#
 \EE  \Delta(r, l)  \leq    2    \sigma_{\bx}  (   \tau l / r)\Biggl\{ 2  \sqrt{  \frac{  \log(2  d)}{n}}  +   \frac{\log(2  d)}{n} \Biggr\}  . \nn
\# 
Combining this with \eqref{RSC.lb1} and \eqref{RSC.lb2} yields the bound \eqref{RSC.bound}.   \qed

\subsection{Proof of Proposition~\ref{prop:score}}

 Write $S_j = (1/n) \sn (  \xi_i x_{ij} - \EE  \xi_i x_{ij} )$ with $\xi_i := \ell_\tau'(\varepsilon_i)$ for $j=1,\ldots, d$, so that $\| \bw^* \|_\infty= \max_{1\leq j\leq d } | S_j |$.  Note that $\ell_\tau'(u) = uI(|u|\leq \tau) + \tau \sign(u) I(|u|>\tau)$, we have   $\EE(\xi_i^2 | \bx_i) \leq \sigma_2^2 $ and $|\xi_i| \leq \tau$.  It follows that $\EE(\xi_i x_{ij} )^2 \leq \sigma_2^2 \sigma_{jj}$ and under Condition~\ref{moment.cond},
\#
	 \EE |\xi_i x_{ij} |^k \leq \tau^{k-2} \sigma_2^2\cdot \EE |x_{ij} |^k \leq \tau^{k-2} \sigma_2^2 \cdot   k!  \sigma_{\bx}^k
\leq \frac{k!}{2} \cdot  2 \sigma_2^2 \sigma_{\bx}^2  \cdot ( \tau \sigma_{\bx} )^{k-2} , \nn
\#
for $k=3 ,4 ,\ldots .$
Bernstein's inequality, in conjunction with the union bound, implies that for any $x\geq 0$,
\#
\max_{1\leq j\leq d } | S_j | \leq   \sigma_{\bx}  \Biggl( 2  \sigma_2  \sqrt{\frac{2x}{n}} + \tau \frac{ x}{n} \Biggr) \nn
\#
with probability at least $1-2 d e^{-x}$. Taking $x= \log(2d) + t $ proves \eqref{score.max.bound}. Next we use a standard covering argument to prove \eqref{score.l2.bound}. For any $\epsilon \in (0,1)$, there exists an $\epsilon$-net $\cN_\epsilon$ of the unit sphere in $\RR^{s}$ with cardinality $| \cN_\epsilon | \leq (1+2/\epsilon)^s$ such that
\#
	 \|   \bw^*_{ \cS} \|_2  \leq \frac{1}{1-\epsilon}  \max_{\bu \in \cN_\epsilon } \frac{1}{n} \sn \bigl( \xi_i   \bu^\T  \bx_{i , \cS}    - \EE \xi_i  \bu^\T  \bx_{i ,  \cS}   \bigr). \label{cover.bound}
\#
For every $\bu \in \cN_\epsilon$, Bernstein's condition holds: $\EE ( \xi_i   \bu^\T  \bx_{i , \cS} )^2\leq 2 \sigma_2^2\sigma_{\bx}^2$ and for $k=3,4,\ldots $,
\#
\EE  | \xi_i    \bu^\T  \bx_{i , \cS}  |^k \leq \frac{k!}{2}  \cdot 2 \sigma_2^2 \sigma_{\bx}^2  \cdot ( \tau   \sigma_{\bx}   )^{k-2} . \nn
\#	 
Again, applying (one-sided) Bernstein's inequality yields, for any $x>0$,
\#
	 \frac{1}{n} \sn \bigl( \xi_i   \bu^\T  \bx_{i , \cS}  - \EE \xi_i   \bu^\T  \bx_{i , \cS}  \bigr)  \leq  \sigma_{\bx} \Biggl( 2 \sigma_2\sqrt{\frac{ x }{n}} + \tau\frac{ x}{n} \Biggr)  \nn
\#
with probability at least $1-e^{-x}$. Consequently, from the union bound and \eqref{cover.bound}, we have
\#
	 \| \bw(\bbeta)_{  \cS} \|_2 \leq \frac{\sigma_{\bx}}{1-\epsilon} \Biggl(2  \sigma_2\sqrt{\frac{ x }{n}} + \tau\frac{ x}{n} \Biggr) \nn
\#
with probability at least $1-e^{\log(1+2/\epsilon) s  -x}$. Taking $\epsilon=1/3$ and $x=2s  +t$ proves \eqref{score.l2.bound}.
\qed

\section{Proofs of Theorems}

\subsection{Proof of Theorem~\ref{thm:l1huber}}

We will apply Propositions~\ref{prop:contraction},  \ref{prop:RSC} and \ref{prop:score} to prove Theorem~\ref{thm:l1huber}.
To begin with, let $\lambda, r , \kappa>0$ satisfy \eqref{scaling.1}, that is,   $\lambda \geq  s^{-1/2} b^*_\tau $ and $r > 2.5 \kappa^{-1} s^{1/2} \lambda$. Applying Proposition~\ref{prop:contraction} with $\epsilon_1=0$ yields that, conditioned on the event $\cE_1(r,6 s^{1/2} ,\kappa)\cap \{ \lambda \geq 2 \| \bw^* \|_\infty\}$,   the Huber-Lasso estimator defined in \eqref{Huber-Lasso} satisfies
\#
	\| \hat \bbeta^{\hlasso} - \bbeta^* \|_2  \leq \kappa^{-1} \big( 1.5 s^{1/2} \lambda + b^*_\tau  \big) \leq  2.5\kappa^{-1} s^{1/2} \lambda .\label{thm1-2}
\#

It remains to control the above event of interest.
Taking $l=6s^{1/2}$ and $\kappa = \rho_l/2$,  it follows from Proposition~\ref{prop:RSC} that $\PP \{ \cE_1(r, l, \kappa) \} \geq 1-e^{-t}$ as long as 
$  \tau \gtrsim \max(\sigma_2, r )$ and $n  \gtrsim   (\tau  / r)^2   (s \log d +   t  )$.
Furthermore,  Proposition~\ref{prop:score} ensures that  event $\{  \lambda \geq 2 \|  \bw^* \|_\infty \}$ occurs with probability at least $1- e^{-t}$ if the  regularization parameter satisfies
\#
	\lambda \geq      \sigma_{\bx}   \Bigg(   4 \sigma_2 \sqrt{\frac{\log(2d)+ t }{  n}} +   2 \tau  \frac{\log (2d) + t }{  n}  \Biggr) .  \nn
\#

Finally,  we take $\tau \asymp \sigma_2 \sqrt{n/(\log d + t)}$ and $r \asymp \tau$.
By Lemma~\ref{lem:bias} and Condition~\ref{moment.cond},  $b^*_\tau \leq \sigma_{\bx} \sigma_2^2  \tau^{-1} \asymp \sigma_2 \sqrt{(\log d + t)/n}$.
Putting together the pieces, we conclude that under the scaling $n \gtrsim  s \log d + t $, and if $\lambda$ has magnitude of the order within the range
\#
	  \sigma_2 \sqrt{\frac{\log d+ t }{ n}}  \lesssim \lambda  \lesssim \sigma_2 s^{-1/2} \sqrt{\frac{n}{ \log d + t  }}  , \nn
\#
the event  $\cE_1(r,6s^{1/2},  \rho_l/2)\cap \{ \lambda \geq 2 \| \bw^* \|_\infty\}$ occurs with probability at least $1-2 e^{-t}$. Combined with \eqref{thm1-2}, this proves the claimed bound. \qed

\subsection{Proof of Theorem~\ref{thm:random}}

We will apply Propositions~\ref{prop:tightening}, \ref{prop:RSC} and \ref{prop:score} to prove \eqref{oracle1}.
The key is to control the random events from Proposition~\ref{prop:tightening}, which relies on a delicate combination of all the parameters.
Similarly to the proof of Theorem~\ref{thm:l1huber}, we take $\kappa = \rho_l/2$, and let $ \lambda  \geq s^{-1/2} b^*_\tau $, where $b^*_\tau  = \| \nabla \cL_\tau(\bbeta^* ) \|_2  \leq \sigma_{\bx}  \sigma_2^2 \tau^{-1}$ due to Condition~\ref{moment.cond}.
Given $\gamma_1 > \gamma_0>0$ satisfying \eqref{gamma0.cond},   define  $\delta= p'(\gamma_0)/(\rho_l\gamma_0) \in (0,1)$,  and let $c>0$ be the constant determined by $ 0.5 p'(\gamma_0) (c^2+1)^{1/2} + 4   =  c   \rho_l\gamma_0$,  which verifies \eqref{def.c} with $\kappa = \rho_l/2$.
Moreover, set $l = \{ 2 + \frac{2}{p'(\gamma_0)} \} (c^2+1)^{1/2} s^{1/2} + \frac{2}{p'(\gamma_0) }s^{1/2}$,  and let 
\#
 r \geq   c  \gamma_0  s^{1/2} \lambda =: r^{{\rm crude}}. \nn
\#
Consequently,  we apply Proposition~\ref{prop:tightening} to conclude that, conditioned on event $\cE_1(r,l,  \rho_l/2 ) \cap \{   \| \bw^*  \|_\infty + n^{-1/2} \leq 0.5 p'(\gamma_0) \lambda  \}$, the $\epsilon_\ell$-optimal solutions $\wt \bbeta^{(\ell)}$ satisfy
\#
 \| \wt \bbeta^{(\ell)} - \bbeta^* \|_2 &\leq  \delta \cdot  \| \wt \bbeta^{(\ell-1)} - \bbeta^* \|_2 \nn \\
 &~~~ +    2\rho_l^{-1} \big\{  \|  p_\lambda'(|\bbeta^*_\cS| -    \gamma_0 \lambda ) \|_2 +  \| \bw^*_{  \cS} \|_2 + (s/n)^{1/2}  + b^*_\tau \big\}       \nn
\#	
for all $\ell \geq 2$.  Under the minimum signal strength condition $\|  \bbeta^*_\cS \|_{\min} \geq (\gamma_0+\gamma_1) \lambda$, and due to the fact that $p'_\lambda(t) = 0$ for all $t\geq \gamma_1 \lambda$,  the deterministic term $\|  p_\lambda'(|\bbeta^*_\cS| -    \gamma_0 \lambda ) \|_2$ vanishes,  thus implying
\#
\| \wt \bbeta^{(\ell)} - \bbeta^* \|_2  \leq \delta^{\ell-1} r^{{\rm crude}}  +2  (1-\delta)^{-1}\rho_l^{-1} \big\{    \| \bw^*_{  \cS} \|_2 + (s/n)^{1/2} + b^*_\tau   \big\}      \label{tightening.ubd1}
\#
for all $\ell \geq 2$. 

Next, we control the event $\cE_1(r,l,  \rho_l/2 ) \cap \{   \| \bw^*  \|_\infty + n^{-1/2} \leq 0.5 p'(\gamma_0) \lambda  \}$ and the oracle error term $\| \bw^*_{  \cS} \|_2$.
Given $t\geq 0$, it follows from Proposition~\ref{prop:RSC} with the above $l \asymp s^{1/2}$ and $\kappa =\rho_l/2$ that event $\cE_1(r, l ,\kappa )$ occurs with probability at least $1- e^{-t}$ as long as  $\tau \gtrsim \max(\sigma_2, r)$ and  $n \gtrsim (\tau/r)^2 (  s \log  d  +   t)$.
We therefore take $r \asymp \tau$ throughout the proof. 
Turning to the gradient vector $\bw^* \in \RR^d$,  applying Proposition~\ref{prop:score} yields that with probability at least $1-2 e^{-t}$,
\$
 \| \bw^* \|_\infty \lesssim \sigma_2  \sqrt{\frac{\log d + t}{n}} + \tau \frac{\log d + t}{n }  ~~\mbox{ and }~~  \| \bw_{\cS}^* \|_2 \lesssim \sigma_2  \sqrt{\frac{s + t}{n}} + \tau \frac{s + t}{n } . 
\$

Based on the above analysis, we choose the regularization parameter $\lambda = C \sigma_2 \sqrt{(\log d + t) /n}$ for a sufficiently large $C$.  Under the scaling $n\gtrsim   s \log d  + t $, and if $\tau$ has magnitude of the order within the range of 
$ \sigma_2 $ to $\sigma_2\sqrt{n/(\log d + t )}$,   it follows from \eqref{tightening.ubd1} that with probability at least $1-3 e^{-t}$,
\#
\| \wt \bbeta^{(\ell)} - \bbeta^* \|_{2}  \lesssim  \delta^{\ell-1}  s^{1/2} \lambda  + \frac{1}{1-\delta}  \Bigg( \sigma_2\sqrt{\frac{s +t}{n}} + \tau \frac{s+t}{n} + b^*_\tau \Bigg) ~\mbox{ for all } \ell \geq 1  .  \nn
\# 
This leads to the claimed bound by letting $\ell \gtrsim \lceil \log(\log d +t) / \log(1/\delta) \rceil$.    \qed

\subsection{Proof of Theorem~\ref{thm:oracle}}

The proof is based primarily on Proposition~\ref{prop:oracle}, combined with complementary probabilistic analysis. 
For $t\geq 0$ and a prescribed $q\geq \max(s, \log d)$,  set $\tau = \sigma_2 \sqrt{n / (q + t)}$.
In order to apply the high-level result in Proposition~\ref{prop:oracle}, we need the following two technical lemmas to control the events in \eqref{oracle.constraint}.  
The former controls the event $\cE_2(r,  l , \kappa)$ defined in \eqref{RSC.event2} under proper sample size requirement, and the latter provides upper bounds on the $\ell_\infty$-error terms 
$\| \hat \bbeta^{{\rm ora}} - \bbeta^* \|_\infty$  and  $\| \bw^{\ora} \|_\infty = \| \nabla \hat \cL_\tau(\hat \bbeta^{{\rm ora}} ) \|_\infty$.

\begin{lemma} \label{lem:oracle.RSC}
Under the conditions of the theorem,  let $\tau, r,  l>0$ satisfy
\#   \label{oracle.scaling}
	 \tau \geq  \max(C_0 \sigma_2, C_1 r )  ~\mbox{ and }~ n \gtrsim     (\tau/ r)^2 (s+  l^2 \log d + t )  ,
\#
where $C_0$ is an absolute constant and $C_1$ depends only on $\sigma_{\bx}$.  Then, with probability at least $1- e^{-t}$, 
\# \label{oracle.rsc}
 \langle  \nabla \hat  \cL_\tau(\bbeta_1 ) -  \nabla \hat  \cL_\tau( \bbeta_2 ) , \bbeta_1 - \bbeta_2  \rangle   \geq \frac{\rho_l}{2} \| \bbeta_1 -\bbeta_2 \|_2^2   
\#
holds uniformly over $ (\bbeta_1 , \bbeta_2 ) \in \cC(r,l)$, where $\cC(r,l)$ is defined in \eqref{RSC.event2}.

\end{lemma}

The next lemma provides statistical properties of the oracle estimator $\hat \bbeta^{{\rm ora}}$ defined in \eqref{def:oracle} with $\tau = \sigma_2 \sqrt{n/( q + t)}$.  Since the oracle $\hat \bbeta^{{\rm ora}}$ has access to the true active set $\cS$, it is essentially an unpenalized Huber estimator based on $\{ (y_i, \bx_{i,  \cS}) \}_{i=1}^n$.

\begin{lemma} \label{lem:oracle.score}
Under the sample size scaling $n\gtrsim  q + t$, the following bounds
\begin{equation}
\begin{aligned}
& \| \hat \bbeta^\ora - \bbeta^* \|_{\Sigma} \lesssim \sigma_2  \sqrt{\frac{s +t }{n}}, \\
& \| \hat \bbeta^\ora - \bbeta^* \|_{\infty} \lesssim   \sigma_2   \sqrt{\frac{\log s  + t}{n}}  +  \sigma_2  \bigg(  \frac{ q + t}{n}  \bigg)^{(1+\eta)/2} ,   \label{oracle.error}
\end{aligned}
\end{equation}
and 
\#
\| \nabla \hat  \cL_\tau(\hat \bbeta^{{\rm ora}}) \|_\infty  \lesssim        \sigma_2   \sqrt{\frac{\log s  + t}{n}}  +   \sigma_2  \bigg(  \frac{q+ t}{n}  \bigg)^{(1+\eta)/2}   \label{score.oracle.ubd}
\#
hold with probability at least $1- 7 e^{-t}$.
\end{lemma}

Compared to Propositions~\ref{prop:RSC} and \ref{prop:score}, the proofs of Lemmas~\ref{lem:oracle.RSC} and \ref{lem:oracle.score},  which are placed in the following two subsections, require a more delicate analysis of the local behavior of the gradient process $\{ \nabla \hat \cL_\tau(\bbeta) , \bbeta \in \RR^d\}$ around both the underlying vector $\bbeta^*$ and the oracle estimator $\hat \bbeta^{\ora}$, with the latter being random itself.

With the above preparations, we are ready to prove the result.  In Proposition~\ref{prop:oracle}, we set $\kappa = \rho_l/2$, $\delta = 2.5/(\rho_l \gamma_0) \in (0,1)$ and $l = \{ 2 + \frac{2}{p'(\gamma_0)} \} (c_0^2 +1 )^{1/2}s^{1/2}$,  where $c_0$ is determined by \eqref{def.c0}.
Taking $r = \tau/C_1$ in Lemma~\ref{lem:oracle.RSC},  we obtain that event $\cE_2(r, l , \kappa)$ happens with probability at least $1-e^{-t}$ as long as $n\gtrsim \sigma_2$ and $n\gtrsim   s \log d + t$.  Next, let $\lambda = C \sigma_2 \sqrt{(\log d + t) / n}$ for a sufficient large constant $C$. Then, it follows from Lemma~\ref{lem:oracle.score} that the event $\{ \| \bw^{\ora} \|_\infty \leq 0.5 p'(\gamma_0) \lambda \} \cap \{ \| \hat \bbeta^{\ora} - \bbeta^* \|_\infty \leq \lambda/(5 \delta \kappa) =   \gamma_0 \lambda / 6.25\} \cap \{ \| \hat \bbeta^{\ora} - \bbeta^* \|_{\Sigma} \leq r \}$ occurs with probability at least $1-7e^{-t}$ as long as $n\gtrsim (q + t)^{1+1/\eta} (\log d )^{-1/\eta}$. 
Finally, the strong oracle property is a direct consequence of Proposition~\ref{prop:oracle}.  \qed

\subsubsection{Proof of Lemma~\ref{lem:oracle.RSC}}


By the convexity of the loss function, for any $\bbeta_1 , \bbeta_2 \in \RR^d$,
\#
	D(\bbeta_1 , \bbeta_2  ) & :=   \langle \nabla  \hat \cL_\tau( \bbeta_1 ) -  \nabla   \hat  \cL_\tau( \bbeta_2 ), \bbeta_1 - \bbeta_2 \rangle \nn \\
	& = \frac{1}{n} \sn \{  \ell_\tau'( y_i - \bx_i^\T \bbeta_2  )  -  \ell_\tau'( y_i -   \bx_i^\T \bbeta_1 )   \}    \bx_i^\T ( \bbeta_1 - \bbeta_2 ) \nn \\
	& \geq  \frac{1}{n} \sn  \{  \ell_\tau'( y_i -  \bx_i^\T  \bbeta_2   )  -  \ell_\tau'( y_i -  \bx_i^\T \bbeta_1    )  \}   \bx_i^\T(  \bbeta_1 - \bbeta_2 ) I_{\cE_i}, \label{def:T}
\#
where $I_{\cE_i}$ is the indicator function of the event $\cE_i  :=  \{ |\varepsilon_i | \leq \tau/4   \}  \cap   \{    |   \bx_i^\T ( \bbeta_2 - \bbeta^* ) |  \leq  \tau/4    \} \cap \{     |   \bx_i^\T ( \bbeta_1 - \bbeta_2 ) |     \leq  \tau \|   \bbeta_1 - \bbeta_2  \|_2 /(2r)   \}$  on which $| y_i -  \bx_i^\T \bbeta_2  | \leq  | \varepsilon_i | +  |\bx_i^\T ( \bbeta_2 - \bbeta^* ) | \leq \tau/2$ and $| y_i - \bx_i^\T \bbeta_1   |  \leq  |  \bx_i^\T (\bbeta_1 - \bbeta_2 ) |  +  |  \bx_i^\T ( \bbeta_2 - \bbeta^* ) | + |\varepsilon_i | \leq \tau$ for all $\bbeta_1  \in  \bbeta_2 + \BB(r  )$.  Similarly to the proof of Proposition~2 in \cite{L2017},  for any $R>0$,  define Lipschitz continuous  functions 
\#  
	\varphi_R(u)  &= u^2 I(|u| \leq R/2) + (|u| - R)^2 I(R/2 < | u| \leq R),  \nn\\	\mbox{ and } ~ \phi_R(u) &= I(|u|\leq R/2) + \{ 2-  (2u/R) \sign(u)\}I(R/2< |u| \leq R) , \nn
\# 
which are smoothed versions of $u \mapsto u^2 I(|u| \leq R)$ and $u \mapsto I(|u| \leq R)$, respectively. 
Moreover,  $\varphi_R(u) \leq u^2 I(|u| \leq R)$ and $\phi_R(u) \leq I(|u|\leq R)$. 
 By  \eqref{def:T} and the fact that $\ell_\tau''(u) = 1$ for $|u|\leq \tau$,  
\#
	 &D(\bbeta_1 , \bbeta_2  ) \geq  \frac{1}{n} \sn  \{   \bx_i^\T(  \bbeta_1 - \bbeta_2 )  \}^2 I_{|\bx_i^\T (\bbeta_1 - \bbeta_2)|/\| \bbeta_1 - \bbeta_2 \|_2  \leq \frac{\tau}{2r}  } I_{|\bx_i^\T (\bbeta_2 - \bbeta^*) | \leq \frac{\tau}{4}  } I_{|\varepsilon_i| \leq \frac{\tau}{4}} \nn \\
	 & \geq D_0(\bbeta_1, \bbeta_2) :=     \frac{1}{n} \sn  \varphi_{ \frac{\tau}{2 r}\|   \bbeta_1 - \bbeta_2  \|_2}(   \bx_i^\T( \bbeta_1 -\bbeta_2 )  ) \phi_{ \frac{\tau}{4} } ( \bx_i^\T ( \bbeta_2 -\bbeta^* ) ) I_{|\varepsilon_i| \leq \frac{\tau}{4}}  \nn \\
	& =  \EE  D_0(\bbeta_1, \bbeta_2) +  D_0(\bbeta_1, \bbeta_2) - \EE  D_0(\bbeta_1, \bbeta_2) .	 \label{def:g}
\#
In what follows, we deal with $\EE  D_0(\bbeta_1, \bbeta_2)$ and $ D_0(\bbeta_1,\bbeta_2) - \EE  D_0(\bbeta_1, \bbeta_2)$, separately.

Noting that $ \varphi_R(u) \geq u^2 I(|u| \leq R/2)  $ and $\phi_{R} (u) \geq   I(|u| \leq R/2)$,  we have
\#
	& \EE  D_0(\bbeta_1, \bbeta_2) \geq  \frac{1}{n} \sn \EE  \varphi_{ \frac{\tau}{2r}  \|   \bbeta_1 - \bbeta_2  \|_2}(  \bx_i^\T ( \bbeta_1 -\bbeta_2 )  )  I_{ | \bx_i^\T ( \bbeta_2 -\bbeta^* )| \leq \frac{\tau}{8} }   I_{ |\varepsilon_i | \leq \frac{\tau}{4} }  \nn \\
	& \geq    \frac{1}{n} \sn \EE \{  ( \bx_i^\T ( \bbeta_1 -\bbeta_2) \}^2 I_{ |  \bx_i^\T ( \bbeta_1 -\bbeta_2 ) | \leq \frac{\tau}{4 r} \|   \bbeta_1 - \bbeta_2  \|_2  }   \nn \\
	& ~~~~~ - \frac{1}{n} \sn  \EE \{  ( \bx_i^\T ( \bbeta_1 -\bbeta_2) \}^2  I_{  |    \bx_i^\T( \bbeta_2 -\bbeta^* ) | >  \frac{\tau}{8} }   -  \frac{1}{n} \sn    \EE \{  \bx_i^\T (  \bbeta_1 -\bbeta_2 ) \}^2  I_{ |\varepsilon_i | > \frac{\tau}{4}}. \label{mean.g.lbd1}
\#
Write $\bdelta = \bbeta_1 - \bbeta_2$ for $\bbeta_1 \in \bbeta_2 + \BB(r)$ and $\bbeta_2 \in \bbeta^* +\BB_{\Sigma}(r)$. By H\"older's inequality and \eqref{def.rhox},
$$
\EE  (\bx_i^\T \bdelta)^2 I_{ | \bx_i^\T \bdelta | > \frac{\tau}{4 r} \|   \bdelta \|_2  }  \leq  2\rho_{\bx}^2 e^{-(\tau/8 \sigma_{\bx} r)^2 } \|   \bdelta  \|_{\Sigma}^2 ,
$$
$$
 \EE ( \bx_i^\T \bdelta )^2  I_{  |   \bx_i^\T ( \bbeta_2 -\bbeta^* ) | > \frac{\tau}{8} } \leq (8/\tau)^2 \EE ( \bx_i^\T \bdelta )^2 \{ \bx_i^\T ( \bbeta_2 -\bbeta^* )\}^2 \leq (8  \rho_{\bx}^2 r/\tau)^2  \|   \bdelta  \|_{\Sigma}^2 
$$
and $ \EE( \bx_i^\T \bdelta )^2  I_{ |\varepsilon_i | > \frac{\tau}{4}} \leq (4\sigma_2/\tau)^2   \|  \bdelta  \|_{\Sigma}^2$.  Substituting these into \eqref{mean.g.lbd1} yields
\#
	 \EE D_0(\bbeta_1, \bbeta_2)  \geq  \big\{  1 -  (4\sigma_2 /\tau)^2  -2\rho_{\bx}^2 e^{-(\tau/8 \sigma_{\bx} r)^2 }   -  (8  \rho_{\bx}^2 r /\tau)^2 \big\}  \|   \bdelta  \|_{\Sigma}^2. \nn
\#
Let $\tau \geq 8 \max(  2 \sigma_2,   \{ \log(16 \rho_{\bx}^2) \}^{1/2} \sigma_{\bx} r , 4\rho_{\bx}^2 r )$, so that
$$
 (4\sigma_2 /\tau)^2 \leq 1/16, \quad 2\rho_{\bx}^2 e^{-(\tau/8 \sigma_{\bx} r)^2 } \leq 1/8 ~\mbox{ and }~ (8  \rho_{\bx}^2 r /\tau)^2 \leq 1/16.
$$
It thus follows that
\#
	 \EE D_0(\bbeta_1, \bbeta_2) \geq  \frac{3}{4}  \|   \bbeta_1 - \bbeta_2  \|_{\Sigma}^2 \geq \frac{3}{4} \rho_l \| \bbeta_1 - \bbeta_2 \|_2^2  \label{mean.g.lbd2}
\#
uniformly over $\bbeta_1 \in \bbeta_2 + \BB(r)$ and $\bbeta_2 \in \bbeta^* +  \BB_{\Sigma}(r)$.

Next,  we  will establish a high probability bound for the supremum 
$$
  \Delta(r, l) : = \sup_{(\bbeta_1, \bbeta_2 ) \in \cC(r,  l ) } \frac{ | D_0(\bbeta_1, \bbeta_2) - \EE D_0(\bbeta_1, \bbeta_2)| }{ \|  \bbeta_1 - \bbeta_2  \|_2^2 } .
$$
Note that $\varphi_{cR}(cu) = c^2 \varphi_R(u)$ for any $c>0$ and $u\in \RR$.
For each pair $(\bbeta_1, \bbeta_2)$,   we write  $\bdelta = \bbeta_1 - \bbeta_2$ and define
\$
	f_{\bbeta_1, \bbeta_2 }(\bx_i, \varepsilon_i) =   \varphi_{\frac{\tau}{2r } } (\bx_i^\T \bdelta/ \| \bdelta \|_2 )  \cdot  \phi_{\frac{\tau}{4}} (\bx_i^\T (\bbeta_2 - \bbeta^*) ) \cdot   I_{|\varepsilon_i| \leq \frac{\tau}{4}}   ,
\$
so that $\Delta(r, l)  = \sup_{(\bbeta_1, \bbeta_2 ) \in \cC(r,  l ) } |(1/n) \sn f_{\bbeta_1, \bbeta_2 }(\bx_i, \varepsilon_i) - \EE f_{\bbeta_1, \bbeta_2 }(\bx_i, \varepsilon_i) |$.
Since $0\leq \varphi_R(u ) \le \min\{(R/2)^2, u^2\}$ and $0\leq \phi_R(u) \leq 1$ for all $u\in \RR$,  we have 
$$
	0\leq f_{\bbeta_1, \bbeta_2 }(\bx_i, \varepsilon_i)  \leq   (\tau / 4 r)^2 ~~\mbox{ and }~~ \EE  f^2_{\bbeta_1, \bbeta_2 }(\bx_i, \varepsilon_i) \leq    \rho_u^2  \rho_{\bx}^4 .
$$
By Bousquet's version of Talagrand's inequality \citep{B2003} and \eqref{RSC.lb2}, for any $t\geq 0$,
\#
	  \Delta(r,l)    \leq 1.25 \EE   \Delta(r,l)  + \rho_u \rho_{\bx}^2 \sqrt{2  t /n }+  (\tau/r)^2 t/(3n) \label{Bousquet.concentration}
\#
holds with probability at least $1 -   e^{-t}$.

It suffices to bound the expected value $\EE \Delta(r,l)$.  Applying the symmetrization inequality for empirical processes and the connection between Gaussian complexity and Rademacher complexity (see, e.g. Lemma~4.5 in \cite{LT1991}), we obtain that
\#
	\EE  \Delta(r, l ) \leq 2 \cdot \sqrt{\frac{\pi}{2}} \cdot   \EE \Biggl\{ \sup_{(\bbeta_1, \bbeta_2 ) \in \cC(r,l) }  | \GG_{\bbeta_1, \bbeta_2 } |  \Biggr\} , \label{exp.ubd1}
\#
where 
$$
	\GG_{\bbeta_1, \bbeta_2 } := \frac{1}{n} \sn g_i     \varphi_{ \frac{\tau}{2r} }(  \bx_i^\T \bdelta / \| \bdelta \|_{2}  ) \phi_{ \frac{\tau}{4} } ( \bx_i^\T( \bbeta_2 -\bbeta^* ) ) I_{|\varepsilon_i| \leq \frac{\tau}{4}}
$$
with $\bdelta = \bbeta_1 - \bbeta_2$ and $g_i$'s are independent standard normal random variables that are independent of the observations. In particular,  $\GG_{\bbeta^*, \bbeta^* } $ is defined as zero. 
Let $\EE^*$ be the conditional expectation given $\{(y_i, \bx_i)\}_{i=1}^n$. By symmetry,
\#
	 \EE^* \Biggl\{ \sup_{(\bbeta_1, \bbeta_2 ) \in \cC(r,l) }  | \GG_{\bbeta_1, \bbeta_2 } |  \Biggr\}  \leq \underbrace{   \EE^*    | \GG_{\bbeta^*, \bbeta^* } |   }_{=0}   + 2  \EE^*  \Biggl\{ \sup_{(\bbeta_1, \bbeta_2 ) \in \cC(r,l) }    \GG_{\bbeta_1, \bbeta_2 }    \Biggr\}     .  \label{exp.ubd2}
\#
Next, we apply the Gaussian comparison theorem to bound  $ \EE^* \{ \sup_{(\bbeta_1, \bbeta_2 ) \in \cC(r, l)}    \GG_{\bbeta_1, \bbeta_2 }  \}$, from which an upper bound on $\EE  \{ \sup_{(\bbeta_1, \bbeta_2 ) \in \cC(r, l) }   \GG_{\bbeta_1, \bbeta_2 }   \}$ follows immediately. 
For another pair $(\bbeta_1',\bbeta_2') \in \cC(r,l)$, write  $\bdelta' = \bbeta_1'-\bbeta_2'$, and note that
\#
	&  \GG_{\bbeta_1, \bbeta_2 } -  \GG_{\bbeta'_1, \bbeta'_2 }  \nn \\
	& =  \GG_{\bbeta_1, \bbeta_2 } -  \GG_{\bbeta_2'+\bdelta, \bbeta'_2 } +  \GG_{\bbeta_2'+\bdelta , \bbeta'_2 }  - \GG_{\bbeta'_1, \bbeta'_2 }  \nn \\
	 & =   \frac{1}{n} \sn g_i \varphi_{\frac{\tau}{2r}} (  \bx_i^\T \bdelta / \| \bdelta \|_{2} )  \{ \phi_{ \frac{\tau}{4} } ( \bx_i^\T ( \bbeta_2 -\bbeta^* ) ) - \phi_{ \frac{\tau}{4} } (   \bx_i^\T ( \bbeta_2' -\bbeta^* ) )  \} I_{|\varepsilon_i| \leq \frac{\tau}{4}} \nn \\
	 & \quad ~+  \frac{1}{n}  \sn g_i   \phi_{ \frac{\tau}{4}  } ( \bx_i^\T ( \bbeta_2' -\bbeta^* ) )   \{ \varphi_{\frac{\tau}{2r}} (  \bx_i^\T \bdelta / \| \bdelta \|_{2} ) - \varphi_{\frac{\tau}{2r}} (    \bx_i^\T \bdelta' / \| \bdelta' \|_{2}  ) \} I_{|\varepsilon_i| \leq \frac{\tau}{4}} . \nn
\#
By the Lipschitz properties of $\phi_R$ and $\varphi_R$, i.e.,  $|\phi_R(u) - \phi_R(v) | \leq  \frac{2}{R}|u-v|$ and $|\varphi_R(u) - \varphi_R(v) | \leq R|u-v|$, and recall that $\varphi_R(u)\leq (R/2)^2$, we have
\#
	& 	\EE^*( \GG_{\bbeta_1, \bbeta_2 } -  \GG_{\bbeta_2'+\bdelta, \bbeta'_2 } )^2  \leq \frac{1}{n^2} \sn   \bigg( \frac{\tau}{4r} \bigg)^4  \bigg( \frac{8}{\tau} \bigg)^2   \{ \bx_i^\T ( \bbeta_2 - \bbeta_2' ) \}^2 \nn \\
&=    \bigg(\frac{\tau}{2 r^2} \bigg)^2  \frac{1}{n^2}\sn     \{ \bx_i^\T ( \bbeta_2 - \bbeta_2' ) \}^2  \label{var.ubd1}
\#
and
\#
& \EE^*( \GG_{\bbeta_2'+ \bdelta, \bbeta_2' } -  \GG_{\bbeta_1' , \bbeta'_2 } )^2 \nn \\
& \leq  \frac{1}{n^2} \sn   \bigl\{    \varphi_{\frac{\tau}{2 r}} (  \bx_i^\T \bdelta / \| \bdelta \|_{2} ) - \varphi_{\frac{\tau}{2 r}} (   \bx_i^\T  \bdelta' / \| \bdelta' \|_{2}   ) \bigr\}^2 \nn \\
& \leq   \bigg(  \frac{\tau}{2 r} \bigg)^2 \frac{1}{n^2 }\sn \bigl(    \bx_i^\T  \bdelta / \| \bdelta \|_{2}   - 
  \bx_i^\T  \bdelta' / \| \bdelta' \|_{2}     \bigr)^2 .  \label{var.ubd2}
\# 
Motivated by \eqref{var.ubd1}, \eqref{var.ubd2} and the inequality that 
$$\EE^*( \GG_{\bbeta_1, \bbeta_2 } -  \GG_{\bbeta'_1, \bbeta'_2 } )^2 \leq 2\EE^*( \GG_{\bbeta_1, \bbeta_2 } -  \GG_{\bbeta_2'+\bdelta, \bbeta'_2 } )^2 + 2\EE^*( \GG_{\bbeta_2'+ \bdelta, \bbeta_2' } -  \GG_{\bbeta_1' , \bbeta'_2 } )^2, $$ 
we define another (conditional) Gaussian process $\{ \ZZ_{\bbeta_1, \bbeta_2} , (\bbeta_1 , \bbeta_2) \in \cC(r, l)\}$ as
\#
	\ZZ_{\bbeta_1, \bbeta_2} &  = \frac{\sqrt{2}\tau}{2r^2} \cdot  \frac{1}{n} \sn g_i'  \bx_i^\T(  \bbeta_2 - \bbeta^* ) + \frac{\sqrt{2}\tau}{2r}  \cdot   \frac{1}{n} \sn g_i''    \bx_i^\T \bdelta / \| \bdelta \|_{2} \nn \\
	& =  \frac{\sqrt{2}\tau}{2 r^2} \cdot  \frac{1}{n} \sn   g_i' \bx_{i , \cS}^\T ( \bbeta_2 - \bbeta^*  )_{\cS}  + \frac{\sqrt{2}\tau}{2r}  \cdot   \frac{1}{n} \sn g_i''   \bx_i^\T \bdelta / \| \bdelta \|_2 , \nn
\#
where $g_1', g_1'', \ldots, g_n', g_n''$ are independent standard normal random variables that are independent of all  the other variables. 
We have established that $\EE^*(  \GG_{\bbeta_1, \bbeta_2 } -  \GG_{\bbeta'_1, \bbeta'_2 })^2 \leq \EE^*(  \ZZ_{\bbeta_1, \bbeta_2 } -  \ZZ_{\bbeta'_1, \bbeta'_2 })^2$. Then, applying Sudakov-Fernique's Gaussian comparison inequality (see, e.g. Theorem~7.2.11 in \cite{V2018}) yields
\#
 \EE^*  \Biggl\{ \sup_{(\bbeta_1, \bbeta_2 ) \in \cC(r,l ) }  \GG_{\bbeta_1, \bbeta_2 } \Biggr\} \leq   \EE^*  \Biggl\{ \sup_{(\bbeta_1, \bbeta_2 ) \in \cC(r, l) }  \ZZ_{\bbeta_1, \bbeta_2 } \Biggr\} ,  \label{Gaussian.sup.1}
\#
which remains valid if $\EE^*$ is replaced by $\EE$.  For the supremum of $\ZZ_{\bbeta_1, \bbeta_2 }$, it is easy to see that
\#
	 & ~~~~ \EE   \Biggl\{ \sup_{(\bbeta_1, \bbeta_2 ) \in \cC(r, l ) }  \ZZ_{\bbeta_1, \bbeta_2 } \Biggr\}  \nn \\
	&  \leq \frac{\sqrt{2} \tau    }{2  r  } \EE \bigg\| \frac{1}{n} \sn g_i' \,  (\Sigma_{\cS\cS})^{-1/2} \bx_{i , \cS} \bigg\|_2 +  \frac{\sqrt{2}  \tau  l }{2  r} \EE \bigg\| \frac{1}{n} \sn g''_i \bx_i \bigg\|_\infty \nn \\
	& \leq    \frac{\sqrt{2} \tau  }{2 r  }  \sqrt{\frac{s}{n}} +  \frac{\sqrt{2}   \tau l }{2    r} \EE \bigg\| \frac{1}{n} \sn g''_i \bx_i \bigg\|_\infty .  \label{Gaussian.sup.2}
\#
Together, \eqref{exp.ubd1},  \eqref{exp.ubd2}, \eqref{Gaussian.sup.1} and \eqref{Gaussian.sup.2} deliver the bound
\#
	\EE \Delta(r,l) \leq  2 \sqrt{\pi} \frac{\tau}{r}  \Biggl\{    \sqrt{\frac{s}{n}}+    l \,   \EE  \Biggl( \max_{1\leq j\leq d} \Biggl|   \frac{1}{n} \sn g_i x_{ij}  \Biggr| \Biggr) \Biggr\}.  \label{exp.ubd3}
\#
Finally we bound the maximum under expectation on the right-hand side of \eqref{exp.ubd3}. Write $S_j =  \sn g_i x_{ij}$ for $j=1,\ldots, d$.
Under Condition~\ref{moment.cond2}, for each $1\leq j\leq d$ and $m\geq 2$ we have
\#
	\EE |x_j |^m  & =  \sigma_{\bx}^{m} m   \int_0^\infty t^{m-1} \PP(|x_j  | \geq \sigma_{\bx} t) \, {\rm d} t \nn \\
	& \leq 2 \sigma_{\bx}^m   m   \int_0^\infty t^{m-1} e^{-t^2/2} \, {\rm d} t = 2^{ m/2} \sigma_{\bx}^m   m \Gamma(m/2). \nn  
\#
Let $g\sim {\sf N}(0,1)$ be independent of $\bx$. Using the Legendre duplication formula, i.e., $\Gamma(s) \Gamma(s+1/2) = 2^{1-2s} \sqrt{\pi} \, \Gamma(2s)$,  and some algebra, we get
\#
	\EE |g x_{j}|^m  & \leq  2^{m/2} \frac{\Gamma(\frac{m+1}{2})}{\sqrt{\pi}}   \cdot  2^{m/2} \sigma_{\bx}^m   m \Gamma(m/2)
= 2 \sigma_{\bx}^m    m! 		= \frac{m!}{2} \underbrace{  4 \sigma_{\bx}^2 }_{\geq \, \EE (g_j x_j)^2 }   \sigma_{\bx}^{m-2}  . \nn
\#
Hence,  using Bernstein's inequality and the symmetry of normal distribution yields
\$
	\log \EE e^{\lambda S_j } = \log \EE e^{-\lambda S_j} \leq   \frac{ 4 \sigma_{\bx}^2 n  \lambda^2}{2(1-  \sigma_{\bx} \lambda )}
\$
for all $\lambda \in (0,1/ \sigma_{\bx})$.  Combined with Theorem~2.5 in \cite{BLM2013}, this implies
\#
  &~~~~ \EE  \Biggl( \max_{1\leq j\leq d} \Biggl|   \frac{1}{n} \sn g_i x_{ij}  \Biggr| \Biggr) =	\EE \max_{1 \leq j\leq d} |S_j/n| \nn \\
  & \leq   \sigma_{\bx}  \Biggl\{  2  \sqrt{\frac{ 2\log(2 d)}{  n}} + \frac{  \log(2 d)}{n} \Biggr\}.   \label{exp.ubd4}
\#
 
Combining \eqref{exp.ubd3}, \eqref{exp.ubd4} with the concentration inequality  \eqref{Bousquet.concentration}, we determine that with probability at least $1-   e^{-t}$, $\Delta(r, l)\leq \rho_l/4$ as long as $n\gtrsim   (\tau/ r)^2 (s+  l^2 \log d + t )$. This, together with \eqref{def:g} and \eqref{mean.g.lbd2}, proves the claim \eqref{oracle.rsc}. \qed

\subsubsection{Proof of Lemma~\ref{lem:oracle.score}}

To begin with, consider the decomposition
$$
	\nabla\hat  \cL_\tau(\hat \bbeta^{{\rm ora}})= \bw(\hat \bbeta^{{\rm ora}}) - \bw(\bbeta^*) +    \nabla  \cL_\tau(\hat \bbeta^{{\rm ora}})  + \bw^* ,
$$
where $\bw(\bbeta) = \nabla\hat  \cL_\tau( \bbeta ) - \nabla  \cL_\tau( \bbeta)$,  $\cL_\tau(\bbeta) = \EE \hat \cL_\tau(\bbeta)$ and $\bw^* = \bw(\bbeta^*)$.  In the following, we control the $\ell_\infty$-norms of the three terms,  $ \bw(\hat \bbeta^{{\rm ora}}) - \bw(\bbeta^*)$,  $  \nabla  \cL_\tau(\hat \bbeta^{{\rm ora}}) $ and $ \bw^*$, separately.  Throughout the proof,   we take $\tau = \sigma_2 \sqrt{n/(q  + t)}$ for some $q \geq \max(s , \log d)$ and $t\geq 0$.

Applying Proposition~\ref{prop:score} to the centered gradient $\bw^* = (  (\bw^*_{\cS} )^\T, (\bw^*_{\cS^{\cc}}) )^{\T} )^\T\in \RR^d$ with slight modifications, we obtain that with probability at least $1-  2 e^{-t}$, 
\#
	 \| \bw^*_{\cS} \|_\infty \lesssim  \sigma_2 \sqrt{\frac{\log s + t}{n}}  ~~\mbox{ and }~~ \| \bw^*_{\cS^{\cc}} \|_\infty \lesssim \sigma_2 \sqrt{\frac{\log(d-s) + t}{n}} ,
	 \label{score.two.ubd}
\#
thus implying $\| \bw^* \|_\infty \lesssim \sigma_2 \sqrt{(\log d + t)/n}$ with the same probability.

 Recall that $\hat \bbeta^{{\rm ora}}$ and $\bbeta^*$ have the same support $\cS \subseteq [d]$. Define the oracle local neighborhood $\Theta^*(r) = \{ \bbeta \in \bbeta^* +\BB_{\Sigma}(r):  \bbeta_{\cS^\cc} = \textbf{0} \}$.  Then, conditioned on the event $\{  \| \hat \bbeta^{{\rm ora}}  -\bbeta^* \|_{\Sigma} \leq r \}$,
\#
 \|  \bw(\hat \bbeta^{{\rm ora}}) - \bw(\bbeta^*)  \|_\infty \leq \sup_{\bbeta \in \Theta^*(r)} \|  \bw( \bbeta ) - \bw(\bbeta^*)  \|_\infty . 
 \label{local.fluct.max}
\#
We thus focus on the supremum on the right-hand side of \eqref{local.fluct.max}.
For every $s$-sparse vector $\bbeta \in \Theta^*(r)$,  we write $\bdelta = (\bbeta - \bbeta^*)_{\cS} \in \RR^s$. 
For  $j= 1,\ldots, d$,  let $\be_j \in \RR^d$ be the coordinate vector that has 1 on its $j$-th coordinate and 0 elsewhere, and define $\Delta^0_{j} (\bdelta)= \langle \bw(\bbeta)-  \bw(\bbeta^*)  , \be_j \rangle =  (1/n) \sn(  \eta_{ij}- \EE\eta_{ij} )$ for $\bdelta \in \RR^s$, where $\eta_{ij} = x_{ij} \{\ell_\tau'(\varepsilon_i - \bx_{i ,  \cS }^\T \bdelta    ) - \ell_\tau'(\varepsilon_i) \}$. Consequently, we have
\#
 \sup_{\bbeta \in\Theta^*(r) }   \| \bw(\bbeta)-  \bw(\bbeta^*)  \|_\infty \leq  \max_{1\leq j\leq d}  \sup_{  \| \bdelta \|_{{\rm S} } \leq r  } \Delta^0_{j} (\bdelta) \bigvee \max_{1\leq j\leq d}  \sup_{  \| \bdelta \|_{ {\rm S}  } \leq r }  - \Delta^0_{j} (\bdelta) ,  \label{local.fluct.dec}
\#
where ${\rm S} = \Sigma_{\cS \cS} \in \RR^{s\times s}$.
In order to bound the local fluctuation $\sup_{\bdelta: \| \bdelta \|_{ {\rm S} } \leq r} \Delta^0_{j} (\bdelta)$, we need to control the moment generating function of $\Delta^0_{j} (\bdelta)$ for each $\bdelta \in \RR^s$. 
By the Lipschitz continuity of $\ell_\tau'(\cdot)$,  $| \EE (\eta_{ij} ) | \leq  \EE |x_{ij}  \bx_{i, \cS}^\T\bdelta | \leq \sigma_{jj}^{1/2} \| \bdelta \|_{ {\rm S} }$, $\EE(\eta_{ij}^2 | \bx_i  ) \leq  x_{ij}^2 (   \bx_{i, \cS}^\T\bdelta  )^2$ and 
\#
 \EE \{ (\eta_{ij} - \EE \eta_{ij})^2 | \bx_i \} & \leq 2  \EE(\eta_{ij}^2 | \bx_i) + 2  (\EE \eta_{ij})^2  \leq  2 x_{ij}^2 ( \bx_{i, \cS}^\T\bdelta )^2 + 2 \sigma_{jj}  \| \bdelta \|_{{\rm S} }^2. \nn
\#
The above moment inequalities, combined with the elementary inequality $|e^u-1-u| \leq (u^2/2) e^{u\vee0}$, imply that for any $\lambda  \in \RR$ and $\lambda^* = \lambda / (  \sigma_{jj}^{1/2} \| \bdelta \|_{ {\rm S} } )$,
\#
&  \EE e^{ \lambda \sqrt{n} \Delta^0_{j} (\bdelta) / ( \sigma_{jj}^{1/2}  \| \bdelta \|_{{\rm S} } )} =  \prod_{i=1}^n \EE e^{ \frac{\lambda^*}{\sqrt{n}} (\eta_{ij} - \EE \eta_{ij}) }  \nn \\
& \leq \prod_{i=1}^n  \EE \Biggl\{ 1 + \frac{\lambda^{* 2}}{  2 n }  (\eta_{ij} - \EE\eta_{ij})^2  e^{ \frac{|\lambda^* |}{  \sqrt{n}} |\eta_{ij} - \EE \eta_{ij} |   } \Biggl\} \nn \\
& \leq   \prod_{i=1}^n  \Biggl\{ 1 + \frac{\lambda^{* 2}  e^{  |\lambda |   /\sqrt{n}}  }{  2 n } \EE   (\eta_{ij} - \EE\eta_{ij})^2  e^{ \frac{|\lambda^* |}{ \sqrt{n}} | x_{ij}  \bx_{i, \cS}^\T\bdelta  |   } \Biggl\} \nn \\
& \leq  \prod_{i=1}^n  \Biggl\{ 1 + \frac{\lambda^{ 2}  }{  n } e^{  |\lambda | /\sqrt{n}}   \EE e^{ \frac{|\lambda^* |}{ \sqrt{n}} | x_{ij}   \bx_{i, \cS}^\T\bdelta |   }   +   \frac{\lambda^{ * 2}    }{  n } e^{  |\lambda | /\sqrt{n}}  \EE  x_{ij}^2 ( \bx_{i, \cS}^\T\bdelta)^2 e^{ \frac{|\lambda^* |}{ \sqrt{n}} | x_{ij}  \bx_{i, \cS}^\T\bdelta |   } \Biggl\} .  \label{mgf.ubd1}
\#
Applying H\"older's inequality to the exponential moments on the right-hand side of \eqref{mgf.ubd1}, we have
\#
 & \EE  x_{ij}^2(  \bx_{i, \cS}^\T\bdelta )^2 e^{ \frac{|\lambda^* |}{ \sqrt{n}} | x_{ij}    \bx_{i, \cS}^\T\bdelta  |   }  \nn \\
 &  \leq   \sigma_{jj} \| \bdelta \|_{{\rm S} }^2 \cdot  \Bigg\{  \EE  (  x_{ij}/   \sigma_{jj}^{1/2}  )^4 e^{\frac{|\lambda| }{\sqrt{n}}   x_{ij}^2 /  \sigma_{jj}  } \Bigg\}^{1/2}   \cdot  \Bigg(  \EE (\bdelta^\T  \bx_{i, \cS} / \| \bdelta \|_{{\rm S}})^4 e^{\frac{|\lambda |}{\sqrt{n}}  ( \bx_{i, \cS}^\T\bdelta / \| \bdelta \|_{{\rm S} })^2  }  \Bigg)^{1/2} \nn
\#
and 
\#
 \EE e^{ \frac{|\lambda^* |}{ \sqrt{n}} | x_{ij}   \bx_{i, \cS}^\T\bdelta |   }  \leq \bigg(  \EE  e^{\frac{|\lambda|}{\sqrt{n}}  x_{ij}^2 / \sigma_{jj}  } \bigg)^{1/2} \cdot  \bigg( \EE e^{\frac{|\lambda |}{\sqrt{n}}  ( \bx_{i, \cS}^\T  \bdelta / \| \bdelta \|_{ {\rm S} }  )^2  } \bigg)^{1/2} . \nn
\#
Substituting these bounds into the earlier inequality \eqref{mgf.ubd1}, we find that for any $|\lambda | \leq \sqrt{n}/C_1$,
\#
&  \EE e^{ \lambda \sqrt{n} \Delta^0_{j} (\bdelta) / ( \sigma_{jj}^{1/2} \| \bdelta \|_{ {\rm S}} )}    \leq e^{C_2^2 \lambda ^2/2} , \nn
\# 
where $C_1, C_2 >0$ depend only on $\upsilon_1$ in Condition~\ref{moment.cond2}.  A similar argument can be used to establish the same bound for each pair $(\bdelta, \bdelta')$, that is,
\#
  \EE e^{ \lambda \sqrt{n}  \{ \Delta^0_{j} (\bdelta) - \Delta^0_{j} (\bdelta')  \} / ( \sigma_{jj}^{1/2}  \|  \bdelta  - \bdelta' \|_{  {\rm S} } )}    \leq e^{C_2^2 \lambda ^2/2}  ~\mbox{ for all } |\lambda | \leq \sqrt{n}/C_1 .\nn
\#
The above inequality certifies condition ($\mathcal{E}d$) in \cite{S2012} (see Section~2 in the supplement), so that Corollary~2.2 therein applies to the process $\{ \Delta^0_{j} (\bdelta)\}_{\bdelta \in \RR^s : \| \bdelta \|_{  {\rm S} } \leq r }$: with probability at least $1- e^{-x}$,
\#
 \sup_{\bbeta \in\Theta^*(r)}	\langle  \bw(\bbeta)-  \bw(\bbeta^*)  , \be_j \rangle = \sup_{\bdelta: \| \bdelta \|_{ {\rm S}} \leq r} \Delta^0_{j} (\bdelta) \lesssim  r \,  \sqrt{\frac{s +x }{n}} \nn
\#
as long as $n \gtrsim  s + x$. Combined with \eqref{local.fluct.dec} and the union bound, we find that
\#
\sup_{\bbeta \in\Theta^*(r) }   \|  \bw(\bbeta)-  \bw(\bbeta^*)  \|_\infty   \lesssim  r \,  \sqrt{\frac{s+x }{n}}  \nn
\#
with probability at least $1- 2d e^{-x}$ provided $n \gtrsim  s + x$.  Taking $x= \log(2d) +t$,
it follows from \eqref{local.fluct.max} that conditioned  on $\{  \| \hat \bbeta^{{\rm ora}}  -\bbeta^* \|_{\Sigma} \leq r \}$,
\#
  \|  \bw(\hat \bbeta^{{\rm ora}}) - \bw(\bbeta^*)  \|_\infty \lesssim r \,  \sqrt{\frac{s+ \log d + t }{n}}  \label{score.diff.bound1}
\#
holds with probability at least  $1-e^{-t}$ as long as $n\gtrsim s + \log d  + t$.

Tuning to $\|    \nabla \cL_\tau(\hat \bbeta^{{\rm ora}}) \|_\infty$, again, we control this term conditioned on the same event above.
Following the proof of Lemma~\ref{lem:bias}, it can be similarly shown that
\#
  \|    \nabla \cL_\tau( \bbeta^* ) \|_\infty   \lesssim  ( a_\eta \sigma_2)^{2+\eta} \tau^{-1-\eta}   . \label{score.bias.ubd}
\#
For any $\bbeta\in\Theta^*(r)$, write $\bdelta = ( \bbeta -\bbeta^*)_{ \cS} \in \RR^s$, and note that
\#
&  \nabla  \cL_\tau( \bbeta )   -  \nabla \cL_\tau( \bbeta^* ) =  \EE \bigl\{  \ell_\tau'( \varepsilon  ) -  \ell_\tau'( \varepsilon     - \bx_{ \cS}^\T  \bdelta )\bigr\} \bx \nn \\
 & =      \EE\int_{     -\bx_{ \cS}^\T  \bdelta   }^{ 0} \ell''_\tau(\varepsilon + u ) \, {\rm d} u  \cdot   \bx = \EE  \ell_\tau''( \varepsilon ) \bx\bx_{ \cS}^\T  \bdelta  +  \EE\int_{ -\bx_{ \cS}^\T  \bdelta }^{ 0  }  \bigl\{ \ell''_\tau( \varepsilon    + u)  - \ell''_\tau( \varepsilon ) \bigr\} \, {\rm d} u  \cdot   \bx .  \nn  
\#
Let $\EE_{\bx}$ and $\PP_{\bx}$ be the conditional expectation and probability given $\bx$, respectively.
By the anti-concentration property \eqref{anti-concentration} of the distribution of $\varepsilon$ given $\bx$, we see that for any $u\in \RR$,
\#
 &  |  \EE_{\bx}   \{ \ell''_\tau(\varepsilon   + u)  - \ell''_\tau(\varepsilon  )  \}  |     =  | \PP_{\bx}(|\varepsilon  + u| \leq \tau) - \PP_{\bx}(|\varepsilon   | \leq \tau   )  | \leq  a_0  |u|. \nn
\#
Together, the last two displays imply
\#
 &~~~~  \|    \nabla \cL_\tau( \bbeta )  - \nabla \cL_\tau(\bbeta^*)   +  {\rm H}_{\cdot \cS} \bdelta  \|_\infty  \nn \\
 & \leq \frac{ a_0 }{2  } \max_{1\leq j\leq d} \EE  |x_{j}|    (\bx_{   \cS}^\T \bdelta)^2  \leq \frac{a_0}{2 } \max_{1\leq j\leq d} \sigma_{jj}^{1/2} \rho_{\bx}^2   \| \bdelta \|_{ {\rm S}  }^2  ,   \label{pop.score.1}
\#
where ${\rm H}_{\cdot \cS} := \EE\{ \ell''_\tau(\varepsilon  ) \bx \bx_{  \cS}^\T \} \in \RR^{d \times s}$ is the submatrix of ${\rm H} = \nabla^2 \cL_\tau(\bbeta^*) = \EE \{ \ell''_\tau(\varepsilon) \bx \bx^\T \}$.  
For the linear term ${\rm H}_{\cdot \cS} \bdelta$,  write $\Sigma_{\cdot \cS} = \EE(\bx \bx^\T_{\cS})$ and note that
\#
  \|   ( {\rm H} _{\cdot \cS} - \Sigma_{ \cdot \cS }    )  \bdelta  \|_\infty   &\leq  \max_{1\leq j\leq d}    \EE  \big\{  \PP_{  \bx} (|\varepsilon | \geq  \tau) \cdot  |  x_j \bx_{\cS}^\T \bdelta |  \big\} \nn \\
  & \leq   \max_{1\leq j\leq d} \sigma_{jj}^{1/2}  \sigma_2^2 \| \bdelta\|_{{\rm S} } \tau^{-2} .  \label{pop.score.2}
\#
Together,  \eqref{score.bias.ubd},  \eqref{pop.score.1} and \eqref{pop.score.2} imply that conditioned  on $\{  \| \hat \bbeta^{\ora}  -\bbeta^* \|_{\Sigma} \leq r \}$,
\# \label{score.diff.bound2}
   \|    \nabla \cL_\tau( \hat \bbeta^{\ora}  )  + \Sigma_{ \cdot \cS }  (\hat \bbeta^{\ora}   - \bbeta^*)_{ \cS} \|_\infty    \lesssim r^2 +    \sigma_2^{2+\eta} \tau^{-1-\eta}  + \sigma_2^2  \tau^{-2} r   .
\#

Next we consider the oracle estimator $\hat \bbeta^{\ora}$ with $\tau = \sigma_2 \sqrt{n/( q  + t )}$. Following an argument similar to that used to prove Theorem~2.1 in \cite{CZ2020}, it can be shown that with probability at least $1-3 e^{-t}$,
\#
\|   \hat \bbeta^{{\rm ora}}  -\bbeta^*     \|_{\Sigma}= \| ( \hat \bbeta^{{\rm ora}}  -\bbeta^*  )_{\cS}  \|_{{\rm S} } \lesssim  \sigma_2 \sqrt{\frac{s+ t}{n}}
 \label{oracle.l2}
\#
and
\#
\Bigg\|  {\rm S} ^{ 1/2} ( \hat \bbeta^{{\rm ora}}  -\bbeta^*  )_{\cS} - {\rm S} ^{-1/2}  \frac{1}{n}\sn \ell'_\tau(\varepsilon_i )  \bx_{i, \cS}   \Bigg\|_2  \lesssim \sigma_2 \frac{s+t }{n} , \label{oracle.bh}
\#
where ${\rm S}  = \EE ( \bx_{  \cS} \bx_{\cS}^\T) = \Sigma_{\cS \cS}$.   Note that   the linear term $(1/n)\sn \ell'_\tau(\varepsilon_i ) \bx_{i, \cS}  $ in the Bahadur representation bound \eqref{oracle.bh} can be written as  $(1/n) \sn \ell'_\tau(\varepsilon_i )  \bx_{i, \cS}   =  -  \bw_{\cS} ^* -  \nabla \cL_\tau(\bbeta^*)_{\cS}$. It follows that
\#
	& \|     \hat \bbeta^\ora - \bbeta^*  \|_\infty =  \|  ( \hat \bbeta^\ora - \bbeta^*)_{\cS} \|_\infty \nn \\ 
& \leq  \|  ( \hat \bbeta^{{\rm ora}}  -\bbeta^*  )_{\cS}   + {\rm S} ^{-1}  \bw_{\cS} ^* +    {\rm S} ^{-1}   \nabla \cL_\tau(\bbeta^*)_{\cS}    \|_\infty +   \|   {\rm S} ^{-1} \bw_{\cS} ^* +  {\rm S} ^{-1} \nabla \cL_\tau(\bbeta^*)_{\cS}  \|_\infty  \nn \\
&    \leq \|  ( \hat \bbeta^{{\rm ora}}  -\bbeta^*  )_{\cS} + {\rm S} ^{-1}  \bw_{\cS} ^* +    {\rm S} ^{-1}   \nabla \cL_\tau(\bbeta^*)_{\cS}  \|_2 +   \|    {\rm S} ^{-1}   \bw_{\cS} ^*    \|_\infty +     \|    {\rm S} ^{-1}   \nabla \cL_\tau(\bbeta^*)_{\cS}   \|_2   \nn \\
& \leq    \| {\rm S} ^{-1} \|_2^{1/2}   \big\{  \|   {\rm S} ^{ 1/2}( \hat \bbeta^{{\rm ora}}  -\bbeta^*  )_{\cS} + {\rm S} ^{-1/2}  \bw_{\cS} ^* \nn \\
&~~~~~~~~~~~~~~~~~~~~  +    {\rm S} ^{-1/2}   \nabla \cL_\tau(\bbeta^*)_{\cS}   \|_2 + \|    {\rm S} ^{-1/2}   \nabla \cL_\tau(\bbeta^*)_{\cS}    \|_2 \big\}  + \|    {\rm S} ^{-1}   \bw_{\cS} ^*    \|_\infty . \nn
\#
Similarly to Lemma~\ref{lem:bias}, we obtain that $  \|    {\rm S} ^{-1/2}   \nabla \cL_\tau(\bbeta^*)_{\cS}    \|_2 \leq  (a_\eta \sigma_2)^{2+\eta} \tau^{-1-\eta}$. For $ \|    {\rm S} ^{-1}   \bw_{\cS} ^*    \|_\infty$, following the proof of Proposition~\ref{prop:score}, it can be similarly shown that with probability at least $1- e^{-t}$,
\#
 \|   {\rm S} ^{-1}   \bw_{\cS} ^*  \|_\infty \lesssim   \sigma_2 \sqrt{\frac{\log(2s) + t }{n}}  . \nn
\#
Putting together the pieces, we conclude that the $\ell_2$-error bound \eqref{oracle.l2} and the $\ell_\infty$-error bound
\#
 \|     \hat \bbeta^\ora - \bbeta^*  \|_\infty  \lesssim  \sigma_2  \sqrt{\frac{\log s + t}{n}}  + \sigma_2  \bigg(  \frac{q + t}{n}  \bigg)^{(1+\eta)/2}\nn
\#
hold with probability $1-4 e^{-t}$ as long as $n\gtrsim  q + t$.  Combined with \eqref{oracle.l2}, this proves \eqref{oracle.error}.

Finally, it remains to deal with $\| \Sigma_{\cdot \cS} (\hat \bbeta^{\ora} - \bbeta^* )_{\cS} \|_\infty$.  Under condition~\eqref{irr},   
  \$
 \| \Sigma_{\cdot \cS} \bdelta \|_\infty \leq  \max_{  j \in  \cS^{\cc}} \| \Sigma_{j\cS} (\Sigma_{\cS\cS} )^{-1} \|_1  \cdot   \| \Sigma_{\cS\cS} \bdelta \|_\infty \leq A_0 \cdot  \| {\rm S}  \bdelta \|_\infty ~\mbox{ for any }   \bdelta \in \RR^s.
  \$
 Using the previous bounds \eqref{score.two.ubd},  \eqref{score.bias.ubd} and \eqref{oracle.bh}, we obtain that
 \$
 	 & \| {\rm S}   (\hat \bbeta^{\ora} - \bbeta^*)_{\cS} \|_\infty \\
 	 &  \leq \Bigg\| {\rm S}      (\hat \bbeta^{\ora} - \bbeta^*)_{\cS} - \frac{1}{n} \sn \ell'_\tau(\varepsilon_i) \bx_{i, \cS}  \Bigg\|_\infty +  \Bigg\|    \frac{1}{n} \sn \ell'_\tau(\varepsilon_i) \bx_{i, \cS}   \Bigg\|_\infty \\
 	 & \leq  \Bigg\| {\rm S}      (\hat \bbeta^{\ora} - \bbeta^*)_{\cS} - \frac{1}{n} \sn \ell'_\tau(\varepsilon_i) \bx_{i, \cS}  \Bigg\|_2 +    \| \bw^*_{\cS } \|_\infty + \| \nabla \cL_\tau(\bbeta^* )_{\cS} \|_\infty \\
 	 & \lesssim   \sigma_2 \frac{s+t}{n} +  \sigma_2 \sqrt{\frac{\log s + t}{n}} +   a_\eta^{2+\eta} \sigma_2  \Bigg(  \frac{ q + t}{n}  \Bigg)^{(1+\eta)/2} .
 \$ 
 Combining this bound with  \eqref{score.diff.bound1},  \eqref{score.diff.bound2}  and \eqref{oracle.l2} yields the claim \eqref{score.oracle.ubd}.    \qed

\subsection{Proof of Theorem \ref{thm:comp:1}}
For simplicity, we  write $\bbeta^{(k)} = \bbeta^{(1,k)}$, $\phi^{(k)} =\phi^{(1,k)}$ and $\blambda = \blambda^{(0)}$ throughout this section.


\subsubsection{Technical lemmas}

We first present three technical lemmas, which are the key ingredients to the proof. The first lemma provides an alternative to the stopping rule. 
\begin{lemma}\label{lemma:c.1}
$
\omega_{\blam} (\btt^{( k)})\leq L(1+\gamma_u) \|\btt^{( k)}-\btt^{( k-1)} \|_2.
$
\end{lemma}

\begin{proof}[Proof of Lemma \ref{lemma:c.1}]
For simplicity,  we write $\cL(\cdot) = \hat \cL_\tau(\cdot)$ as the loss function of interest.  Since $\btt^{(k)}$ is the exact solution at the $k$-th iteration when $\ell=1$,
	the first-order optimality condition holds:  there exists some $\bxi^{(k)}\in \partial \|\bbeta^{(k)}\|_1$ such that 
	\$\label{}
	\nabla\cL(\btt^{(k-1)})+\phi^{(k)}(\btt^{(k)}-\btt^{(k-1)})+  \blam\circ \bxi^{(k)}=\textbf{0}_d  .
	\$
	For any $\bu \in \RR^d$ such that $\| \bu\|_1=1$, we have
	\$
	&~~~~ \langle\nabla\cL(\btt^{(k)}) +   \blam\circ \bxi^{(k)},\bu\rangle \nn \\
	&=\langle\nabla\cL(\btt^{(k)}), \bu\rangle-\langle\nabla\cL(\btt^{(k-1)})+\phi^{(k)}(\btt^{(k)}-\btt^{(k-1)}),\bu \rangle \\
	&= \langle\nabla\cL(\btt^{(k)})-\nabla\cL(\btt^{(k-1)}),\bu\big\rangle - \langle\phi^{(k)}(\btt^{(k)}-\btt^{(k-1)}),\bu\rangle\\
	&\leq \|\nabla\cL(\btt^{(k)}) - \nabla\cL(\btt^{(k-1)})\|_\infty +  \phi^{(k)} \|\btt^{(k)}-\btt^{(k-1)}\|_\infty \\
	&\leq (\phi^{(k)}+L )\|\btt^{(k)}-\btt^{(k-1)}\|_2,
	\$
	where the last inequality is due to the Lipschitz continuity of $\nabla \cL(\cdot )$.  Taking the supremum over all $\bu$ satisfying $\|\bu \|_1\leq 1$, we obtain 
\$
\omega_{\blam} (\btt^{(k)})\leq  (\phi^{(k)}+L )  \|\btt^{(k)}-\btt^{(k-1)} \|_2.
\$
	
	It remains to show that $\phi^{(k)}\leq   L \gamma_u $ for any $k$. This is guaranteed by the iterative LAMM algorithm. Otherwise, if $\phi^{(k)}> L\gamma_u$, then $\phi'\equiv \phi^{(k)}/\gamma_u>L$ is the quadratic parameter in the previous iteration for searching $\phi$ such that
\$
F (\tilde\bbeta^{(k)}; \phi', \bbeta^{( k-1)} ) < \cL (\tilde\bbeta^{(k)} ),
\$
where $\tilde\btt^{(k)}$ is the new updated parameter vector under the quadratic coefficient $\phi'$.
On the other hand, it follows from the definition of $F$ and the Lipschitz continuity of $\nabla \cL$ that
	\$
&~~~~ F(\tilde\btt^{(k)};\phi', \btt^{(k-1)})+\lambda\|\btt\|_1\nn \\
&=\cL(\btt^{( k-1)})+ \langle\nabla\cL(\btt^{(k-1)}),\tilde\btt^{(k)}-\btt^{(k-1)} \rangle+\frac{\phi'}{2} \|\tilde\btt^{(k)}-\btt^{(k\!-\!1)} \|_2^2\\
&>\cL(\btt^{( k-1)})+ \langle\nabla\cL(\btt^{(k-1)}),\tilde\btt^{(k)}-\btt^{(k-1)} \rangle+\frac{L}{2} \|\tilde\btt^{(k)}-\btt^{(k\!-\!1)} \|_2^2\\
&\geq \cL(\tilde\btt^{(k)}).
	\$
	This leads to a contradiction, indicating that $\phi^{(k)} \leq L  \gamma_u$. 
\end{proof}

The second lemma is a modified version of Lemma E.4 in \cite{FLSZ2018}. We reproduce its proof here for completeness.   Let $\Psi(\btt,\blam)=\cL(\btt)+\|\blam\circ \bbeta  \|_1$ with $\blambda =  \blam^{(0)}$.

\begin{lemma}\label{lemma.c.2}
For any $\bbeta  \in \RR^d$, we have
\$
\Psi(\btt, \blam ) - \Psi(\btt^{(k)},\blam ) \geq \frac{\phi^{(k)}}{2}   \big\{ \|\btt  - \btt^{(k )}\|_2^2 -  \|\btt  - \btt^{(k-1)}\|_2^2  \big\}.
\$
\end{lemma}
\begin{proof}[Proof of Lemma \ref{lemma.c.2}]
Since $F (\bbeta; \phi^{(k)}, \bbeta^{( k-1)} )$ majorizes $\cL(\btt)$ at $\btt^{(k)}$,  we have
	\#\label{0831.1}
	\Psi(\btt,\blam )-\Psi(\btt^{(k)},\blam )\geq \Psi(\btt,\blam )- \big\{ F (\btt^{(k)}; \phi^{(k)}, \btt^{(k-1)} )+\|\blambda \circ \bbeta^{(k)}\|_1 \big\}.
	\#
	By the convexity of $\cL(\cdot)$ and $\bbeta \mapsto \|\blam\circ \bbeta \|_1$,
	\begin{align*}
	& \cL(\btt) \geq\cL(\btt^{(k-1)})+ \langle\nabla\cL(\btt^{(k-1)}),\btt-\btt^{(k-1)} \rangle ~\mbox{ and} \\
	&\|\blam \circ \bbeta \|_1 \geq \|\blam \circ  \bbeta^{(k)}\|_1+ \langle \blam \circ \bxi^{(k)}, \bbeta -\bbeta^{(k)} \rangle  
	\end{align*}
	for any $\bxi^{(k)}\in \partial \|\bbeta^{(k)}\|_1$. This further implies
	\#
	 \Psi(\btt,\blam )& \geq \cL(\btt^{(k-1)}) +  \langle\nabla\cL(\btt^{(k-1)}),\btt  - \btt^{(k-1)} \rangle \nn \\
	&\qquad + \|\blam  \circ  \bbeta^{(k)}\|_1 + \langle\blam \circ \bxi^{(k)}, \bbeta -   \bbeta^{(k)} \rangle. \label{0831.4}
	\#
	Plugging  the expression of $F (\btt^{(k)}; \phi^{(k)}, \btt^{(k-1)} )$ in \eqref{def:F}  and (\ref{0831.4})  into (\ref{0831.1}), we obtain
	\#\label{0107.7}
	& \Psi(\btt,\blam ) - \Psi(\btt^{(k)},\blam )\geq -\frac{\phi^{(k)}}{2}\|\btt^{(k)} - \btt^{(k-1)}\|_2^2\notag\\
	&\qquad + \langle\nabla\cL(\btt^{(k-1)}) ,\btt - \btt^{(k)}\rangle   + \langle \blambda \circ \bxi^{(k)} , \bbeta -\bbeta^{(k)} \rangle . 
	\#
By the first-order optimality condition,  there exists some $\bxi \in \partial \|\bbeta^{(k)}\|_1$ such that
	\$
	\nabla\cL(\btt^{(k-1)})+\phi^{(k)} (\btt^{(k)}-\btt^{(k-1)})+  \blam \circ \bxi^{(k)}=\textbf{0} .
	\$ 
	Substituting this into \eqref{0107.7} proves the claimed bound.
\end{proof}

Recall that $\Psi(\btt,\blam)=\cL(\btt)+\|\blam\circ \bbeta  \|_1$ and $\hat \bbeta^{(1)}  \in \min_{\bbeta \in \RR^d} \Psi(\btt, \blam)$ denotes the optimal solution in the contraction stage.

\begin{lemma}\label{lemma:5.3} 
For any $k\geq 1$, we have
	$$
	\Psi (\btt^{( k)},\blam )-\Psi (\hbt^{(1)},\blam )\leq  \frac{ \max_{1\leq j\leq k}\phi^{(j)} }{2k}  \|\btt^{( 0)}-\hbt^{(1)} \|_2^2.
	$$
\end{lemma}
\begin{proof}[Proof of Lemma \ref{lemma:5.3}]
For simplicity, we write $\hat \bbeta = \hat \bbeta^{(1)}$, and define $\phi_{\max}  = \max_{1\leq j\leq k} \phi^{(j)}$ and $\phi_{\min}  = \min_{1\leq j\leq k} \phi^{(j)}>0$. Taking $\btt=\hbt$ in Lemma \ref{lemma.c.2} gives
	\# 
 0 \geq 	\Psi(\hbt,\blam)-\Psi(\btt^{(j)},\blam) \geq\frac{\phi^{(j)}}{2} \{\|\hbt-\btt^{(j)}\|_2^2-\|\hbt-\btt^{(j-1)}\|_2^2 \}  \nn
	\#
	for all $j\geq 1$. Summing  over $j$ from 1 to $k$ yields
	\$
\sum_{j=1}^k  \frac{2}{\phi^{(j)}}  \{\Psi(\hbt,\blam) - \Psi(\btt^{(j)},\blam) \} \geq \sum_{j=1}^k  \{\|\btt^{(j)} - \hbt\|_2^2 -  \|\btt^{(j-1)} -  \hbt\|_2^2  \},
\$
which further implies
\#\label{0901.1.3}
\frac{2}{\phi_{\max}} \bigg\{k\Psi(\hbt,\blam)-\sum_{j=1}^k\Psi(\btt^{(j)},\blam)  \bigg\} \geq \|\btt^{(k)}-\hbt\|_2^2-\|\btt^{(0)}-\hbt\|_2^2 .
\#
Again, by Lemma~\ref{lemma.c.2} with $\btt=\btt^{(j-1)}$ and $k=j$,
\$
	\Psi(\btt^{(j-1)},\blam)-\Psi(\btt^{(j)},\blam)\geq \frac{\phi^{(j)}}{2}\|\btt^{(j)}-\btt^{(j-1)}\|_2^2 \geq  \frac{  \phi_{\min}}{2}\|\btt^{(j)}-\btt^{(j-1)}\|_2^2.
	\$
Multiplying both sides of the above inequality by $j-1$ and summing over $j$, we obtain
	\$
	&\frac{2}{\phi_{\min} }\sum_{j=1}^k \{(j - 1)\Psi(\btt^{(j - 1)},\blam) - j\Psi(\btt^{(j)},\blam) + \Psi(\btt^{(j)},\blam) \} \nn \\
	& \geq  \sum_{j=1}^k(j - 1)\|\btt^{(j)}  -  \btt^{(j-1)}\|_2^2,
	\$
or equivalently,
	\#\label{0901.1.4}
	\frac{2}{\phi_{\min} }\bigg\{ -k\Psi(\btt^{(k)},\blam)+\sum_{j=1}^k\Psi(\btt^{(j)},\blam)\bigg\}\geq \sum_{j=1}^k(j-1)\|\btt^{(j)}-\btt^{(j-1)}\|_2^2.
	\#
Together, (\ref{0901.1.3}) and (\ref{0901.1.4}) imply
\$
 & \frac{2k}{\phi_{\min} }  \{\Psi(\hbt,\blam) -  \Psi(\btt^{(k)},\blam) \} \nn \\
& \geq  \frac{\phi_{\max} }{\phi_{\min} } \|\btt^{(k)} - \hbt\|_2^2 + \sum_{j=1}^k(j - 1)\|\btt^{(j)} - \btt^{(j-1)}\|_2^2 - \frac{\phi_{\max} }{\phi_{\min} } \|\btt^{(0)} - \hbt\|_2^2,
\$
from which it follows immediately that
	\$\frac{2k}{\phi_{\max} } \{\Psi(\btt^{(k)},\blam)- \Psi(\hbt,\blam) \}\leq \|\btt^{(0)}-\hbt\|_2^2.
	\$
	This completes the proof. 
\end{proof}

\subsubsection{Proof of the theorem}

Recall that $\bbeta^{(k)} = \bbeta^{(1,k)}$ and $\phi^{(k)} =\phi^{(1,k)}$. By Lemma \ref{lemma:c.1} and its proof,
\# 
\omega_{\blam} (\btt^{(k)} ) \leq   (\phi^{(k)}+L )\|\btt^{(k)}-\btt^{(k-1)} \|_2  \leq  L(1+\gamma_u) \|\btt^{(k)}-\btt^{(k-1)} \|_2.  \nn
\#
Next, taking $\btt=\btt^{(k-1)}$ in Lemma \ref{lemma.c.2} yields
\$
\Psi (\btt^{(k-1)},\blam )-\Psi  (\btt^{(k)},\blam )\geq \frac{\phi^{(k)}}{2} \|\btt^{(k-1)}-\btt^{(k)} \|_2^2.
\$
Together, the last two displays lead to a bound for the suboptimality measure
\#\label{0107.9}
	\omega_{\blam}(\btt^{(k)}) \leq L(1 + \gamma_u)\left[   \frac{2}{\phi^{(k)}}  \big\{ \Psi (\btt^{(k-1)},\blam )-\Psi (\btt^{(k)},\blam )  \big\}   \right]^{1/2}.
\#
Recall that $\{\Psi(\btt^{(k)},\blam )\}_{k=0}^\infty$ is a non-increasing sequence, i.e.,
\# 
	\Psi(\hbt^{(1)},\blam)\leq \cdots\leq \Psi(\btt^{(k)},\blam)\leq\cdots\leq \Psi(\btt^{(0)},\blam). \nn
\#
Then, it follows from  \eqref{0107.9} and Lemma \ref{lemma:5.3} that
\# 
 \omega_{\blam}(\btt^{(k)})& \leq L(1+\gamma_u)\left[ \frac{2}{\phi^{(k)}} \{ \Psi(\btt^{(k-1)},\blambda)-\Psi(\hbt,\blambda) \} \right]^{1/2}   \nn \\
&  \leq  \frac{L(1+\gamma_u)}{\sqrt{k-1}} \sqrt{\frac{\max_{1\leq j\leq k-1} \phi^{(j)}}{\phi^{(k)}}}\|\hbt\|_2, \nn
\#
where we used the fact that $\btt^{(0)} = {\bf 0}$.  By the triangle inequality,
\$
\omega_{\blam} (\btt^{(k)} )& \lesssim \frac{L(1 + \gamma_u)}{\sqrt{k }}  \big(\|\bttc\|_2 + \|\hbt - \bttc\|_2  \big) .
\$
Therefore, in the contraction stage, we need  $k\gtrsim \{   L(1+\gamma_u)  (\|\bbeta^*\|_2+\| \hat \bbeta -\bbeta^* \|_2) / \epsilon_{{\rm c}}   \}^2$ to ensure $\omega_{\blam^{(0)}}(\btt^{(k)})\leq \epsilon_{{\rm c}}$.  This proves the stated result. \qed

\subsection{Proof of Theorem \ref{thm:comp:2}}

For convenience, we omit the index $\ell$, and use $\hat \bbeta$, $\btt^{(k)}$, $\blam$ and $\cE$ to denote $\hat \bbeta^{(\ell)}$,  $\btt^{(\ell,k)}, \blam^{(\ell-1 )}$ and $\cE_\ell$, respectively, where $\cE_\ell$ is the subset defined in \eqref{def:El} satisfying $\cS \subseteq \cE_\ell$ and $| \cE_\ell| \leq C_0 s $ for some constant $C_0>1$.
Moreover, write $\cL(\cdot)= \hat  \cL_\tau(\cdot)$, and define $\Psi(\btt,\blambda)=\cL(\btt)+\|\blambda\circ \bbeta  \|_1 =\cL(\btt)+\|\blambda^{(\ell-1)}\circ \bbeta   \|_1  $, so that $\hat \bbeta \in \min_{\bbeta} \Psi(\btt,\blambda)$.

\subsubsection{Technical lemmas}

We first provide several technical lemmas along with the proofs.  

\begin{lemma}\label{lemma:bd.rsc}
 For any  $k$-sparse ($k\geq 1$) vectors $\btt_1,\btt_2  \in  \bttc+ \BB(r)$,   we have
\begin{gather*}
\frac{1}{2} \kappa_-(2k ,r,\tau)\|\btt_1-\btt_2\|_2^2\leq D_\cL(\btt_1,\btt_2) \leq \frac{1}{2}{\kappa_+(2k,r,\tau)\|\btt_1-\btt_2\|_2^2}.
\end{gather*}
where $D_\cL(\btt_1,\btt_2) := \cL(\btt_1) - \cL(\btt_2) - \langle \nabla \cL(\btt_2) , \btt_1 - \btt_2 \rangle$ is the Bregman divergence.
\end{lemma}

\begin{proof}[Proof of Lemma \ref{lemma:bd.rsc}]
By a second-order Taylor series expansion,  there exists some $\gamma \in [0,1]$ such that  $\tilde\btt =\gamma \btt_1+(1-\gamma )\btt_2\in\btt^* +  \BB(r)$ and $ D_\cL(\btt_1,\btt_2)  = (1/2) (\btt_1 - \btt_2)^\T  \nabla^2 \cL(\tilde \btt )   (  \btt_1 - \btt_2 )$.
The stated bounds then follow directly from Definition~\ref{lse}.
\end{proof}

The next lemma converts the bound on $\Psi(\btt, \blambda)-\Psi(\bttc,\blambda)$ to that on $\|\btt-\bttc\|_2$. 
Recall that for any subset $\cE \subseteq [d]$, we write $\bbeta_{\cE}$ as a subvector of $\bbeta$ indexed by $\cE$.

\begin{lemma}\label{lemma:sc}
Assume LSE$(1)$ condition holds. Let $\cE  \subseteq [d] $ be a subset satisfying $\cS \subseteq\cE$ and $|\cE|\leq C_0s$ for some $C_0\geq 1$.   Assume further that $\lambda \geq \max\{ 4 \|\nabla\cL(\bttc)\|_\infty  , \| \blambda \|_\infty  \} $ and $\|\blambda_{\cE^{{\rm c}}  }\|_{\min}\geq \lambda/2$.  Then, for any  $\btt\in \btt^* + \mathbb B(r)$ satisfying
	$
	\|\btt_{  \cS^{{\rm c}}}\|_0\leq s'$ and $\Psi(\btt,\blam)-\Psi(\bttc,\blam)\leq C s \lambda^2 ,
	$
	we have
	\begin{gather}
	\|\btt-\bttc\|_2\leq C_1   s^{1/2}\lambda  ~~\mbox{ and }~~\|\btt-\bttc\|_1\leq C_2  \sqrt{s(s+s')}  \lambda  , \notag
	\end{gather}
	where $C_1, C_2 >0$ depend only on $C_0, C$ and localized sparse eigenvalues.
\end{lemma}
\begin{proof}[Proof of Lemma \ref{lemma:sc}]
	We omit the arguments in $\kappa_-( s +s' ,r,\tau)$ and $\kappa_+( s +s', r, \tau)$ whenever there is no ambiguity. 
 For any $\btt\in \btt^* + \mathbb B(r)$ satisfying $ \|\btt_{  \cS^{{\rm c}}}\|_0\leq s'$, note that $\|\bbeta \|_0 \leq s+s'$ and $\| \bbeta - \bbeta^* \|_0 \leq s + s'$.
	Using Lemma \ref{lemma:bd.rsc} yields
	\$
     \cL(\btt^*)+\langle\nabla\cL(\bttc), \btt-\bttc \rangle+\frac{\kappa_-}{2}\|\btt-\bttc\|_2^2\leq \cL(\btt).
	\$
	Since 
	$
	\Psi(\btt)-\Psi(\bttc)\leq C s \lambda^2 , 
	$
	or equivalently,
	\#\label{0822.1.2}
	\cL(\btt)-\cL(\bttc)+(\|\blam\circ \bbeta \|_1-\|\blam\circ \bbeta^* \|_1)\leq C s \lambda^2  ,
	\#
it follows
	\$
	\frac{\kappa_-}{2}\|\btt -\bttc\|_2^2 \leq  Cs \lambda^2  - \underbrace{   \langle\nabla\cL(\bttc),\btt -\bttc \rangle}_{\Rom{1}}+\underbrace{(\|\blam\circ \bbeta^*\|_1 -\|\blam\circ \bbeta  \|_1)}_{\Rom{2}}.
	\$
After some simple algebra, it can be derived  that
\begin{gather*}
| \Rom{1} | \leq \|(\btt-\bttc)_{ \cE^{{\rm c}}}\|_1\|\nabla\cL(\bttc)\|_\infty+\|(\btt-\bttc)_{  \cE}\|_1\|\nabla\cL(\bttc)\|_\infty,\\
\Rom{2}\leq
\lambda\|(\bbeta -\bbeta^*)_{ \cE}\|_1-  ( \lambda/2 ) \|(\bbeta -\bbeta^*)_{ \cE^{{\rm c}}}\|_1.
\end{gather*}
Combining  the above bounds  gives
	\$
	&\frac{\kappa_-}{2}\|\btt-\bttc\|_2^2+\{ \lambda/2-\|\nabla\cL(\bttc)\|_\infty \} \|(\btt-\bttc)_{  \cE^{{\rm c}}}\|_1\\
	&\qquad{}\leq \{ \lambda+\|\nabla\cL(\bttc)\|_\infty \} \|(\btt-\bttc)_{  \cE}\|_1+C s \lambda^2 , 
	\$
which further implies 
	\$
	\frac{\kappa_-}{2}\|\btt-\bttc\|_2^2&\leq \frac{5\lambda}{4}\|(\btt-\bttc)_{   \cE}\|_1+C  s \lambda^2 .
	\$
To bound the right-hand side of the above inequality, we discuss two cases regarding the magnitude of $\|(\btt-\bttc)_{  \cE}\|_1$ as compared to $s \lambda$:
\begin{itemize}
	\item If $5\lambda\|(\btt-\bttc)_{  \cE}\|_1/4\leq Cs \lambda^2$, we have
\begin{equation}
\begin{aligned}
&\frac{\kappa_-}{2}\|\btt-\bttc\|_2^2\leq 2C s \lambda^2 , ~\textnormal{and hence} \\
&\|\btt-\bttc\|_2\leq 2 (C/\kappa_-)^{1/2}  s^{1/2} \lambda . \label{0822.1.8}
\end{aligned}
\end{equation}
\item If $5\lambda\|(\btt-\bttc)_{ \cE}\|_1/4> Cs \lambda^2$, we have
\$
\frac{\kappa_-}{2}\|\btt-\bttc\|_2^2\leq  \frac{5}{2} \lambda\|(\btt-\bttc)_{  \cE} \|_1   \leq  \frac{5}{2}  \lambda (C_0 s )^{1/2} \|\btt-\btt^*\|_2 ,
\$
thus implying
\#\label{0822.1.9}
\|\btt -\btt^*\|_2\leq  5 C_0^{1/2} \kappa_-^{-1}  s^{1/2} \lambda .
\#
	\end{itemize}
Combining (\ref{0822.1.8}) and (\ref{0822.1.9}), we obtain	
\$
\|\btt-\bttc\|_2 \leq \max\big\{2 (C/ \kappa_-)^{1/2}   ,  5 C_0^{1/2} \kappa_-^{-1}  \big\}  s^{1/2}  \lambda   \asymp s^{1/2}  \lambda   . 
\$
Since $\btt-\bttc$ is at most $(s +s')$-sparse, $\|\btt-\bttc\|_1 \leq (s+s')^{1/2} \| \btt-\bttc \|_2$. The stated results then follow immediately.
\end{proof}

Recall that $\cE_\ell$ is the subset defined in \eqref{def:El} satisfying $\cS \subseteq \cE_\ell$ and $| \cE_\ell| \leq C_0 s $ for some  $C_0>1$.

\begin{lemma}\label{lemma:rsc}
Assume LSE$(C_0)$ condition  holds and $4\{ \|\nabla\cL(\btt^*)\|_\infty +\epsilon_{{\rm c}}  \vee\epsilon_{{\rm t}}  \} \leq  \lambda \lesssim   s^{-1/2}  r$. For any $\ell\geq 2$, 
the solution sequence $\{\btt^{(\ell,k)}\}_{k\geq 0}$  satisfies 
	\begin{equation}
	\begin{aligned}
	 &\|   \btt^{(\ell,k)}_{  \cE_\ell^{{\rm c}}} \|_0\leq s'  ,   \quad  \|\btt^{(\ell,k)}-\bttc\|_2\leq C_1  s^{1/2} \lambda   ~~~\mbox{ and }\\
	 & ~~~~~~~~~~~~~~~~  \|\btt^{(\ell,k)}-\bttc\|_1\leq C_2  s \lambda  ,  \label{eq:ls.set}
	\end{aligned}
	\end{equation}
where  $C_1, C_2>0$ are constants  depending only on  the localized sparse eigenvalues.
\end{lemma}
\begin{proof}[Proof of Lemma \ref{lemma:rsc}]
We prove the theorem by the method of induction on $(\ell,k)$. Throughout, $C$ denotes a constant independent of $(n,d,s)$ and may take different values at each appearance.
For the  first subproblem, directly applying  Proposition 4.1 and Lemma 5.4 in \cite{FLSZ2018} we obtain that $\|\tilde\btt^{(1)}-\btt\|_2\leq  C\kappa_*^{-1} s^{1/2} \lambda <r$, $\|\tilde\btt^{(1)}-\btt\|_1\leq  C\kappa_*^{-1}  s \lambda $ and $\tilde\btt^{(1)}$ is $(s + s')$-sparse, where $s' \leq C s$. It follows that $\btt^{(2,0)}=\tilde \btt^{(1)}$ falls in a localized sparse set. 

To apply the method of induction, first we assume that for any $k$, $\btt^{(2, k)}$ falls in a localized sparse set such that \eqref{eq:ls.set} holds. We then use Lemma E.13 in \cite{FLSZ2018} to show that $\btt^{(2,k+1)}$ also falls in a localized sparse set. To this end, we need to verify two conditions. The first one, 
$ \|\blambda^{(\ell)}_{\cE_\ell^{{\rm c}}}  \|_{\min}\geq \lambda/2$ 
 is guaranteed by Claim \eqref{scaling.stepl} in the proof of Proposition \ref{prop:tightening}, when $\gamma$ is such that $p'(\gamma) = 1/2$ and $|\cE_l| \leq C_0 s $ for some $C_0>1$.   For the second condition, it suffices to show
 \$
 \Psi(\btt^{(2,k)},\blambda^{(1)})- \Psi(\bttc,\blambda^{(1)})\lesssim (1+\zeta) \kappa_*^{-1} s \lambda^2  ,
 \$
where $\zeta=\kappa^*/\kappa_*$. Using the mean value theorem, there exists some convex combination of $\btt^{(2,k)}$ and $\bttc$, say $\tilde \bbeta$, such that 
\$
&\Psi(\btt^{(2,k)},\blambda^{(1)})- \Psi(\bttc ,\blambda^{(1)})\nn \\
& =\cL(\btt^{(2,k)})-\cL(\bttc)+ \{   \|\blambda^{(1)}\circ\btt^{(2,k)}   \|_1-  \|\blambda^{(1)}\circ\bttc \|_1 \}\\
&\leq \langle\nabla\cL(\bttc), \btt^{(2,k)}-\bttc \rangle+\frac{1}{2}(\btt^{(2,k)}-\bttc)^\T\nabla^2\cL(\tilde\btt)(\btt^{(2,k)}-\bttc)\\
&\qquad + \|\blambda^{(1)}\circ(\btt^{(2,k)}-\bttc)  \|_1\\
&\leq \|\nabla\cL(\btt^*)\|_\infty\|\btt^{(2,k)}-\btt^*\|_1+\frac{1}{2}\kappa^* \|\btt^{(2,k)}-\btt^*\|_2^2+\lambda\|\bbeta^{(2,k)}-\bbeta^*\|_1\\
&\leq \frac{C}{4}\kappa_*^{-1}s \lambda^2 +\frac{C^2}{2} \kappa^*\kappa_*^{-2} s \lambda^2  + C\kappa_*^{-1} s \lambda^2  
\lesssim (1+\zeta)\kappa_*^{-1} s \lambda^2 . 
\$
With above preparations, it follows from Lemma E.13 in \cite{FLSZ2018} with slight modification that $\|\btt_{\cE_\ell^{{\rm c}}}^{(2,k+1)}\|_0 \leq s'$ and hence $\|\btt^{(2,k+1)}\|_0\leq C_0  s+ s'$.

Next, we show that $\|\btt^{(2,k+1)}-\btt^*\|_2\lesssim \kappa_*^{-1} s^{1/2}\lambda $. Again, by Lemma  \ref{lemma.c.2},
\$
\Psi(\btt^{(2,k+1)}, \blambda^{(1)})-\Psi(\btt^{(2,k)}, \blambda^{(1)})\leq -\frac{\phi^{(2,k+1)}}{2}  \|\btt^{(2,k+1)}-\btt^{(2,k)}  \|_2 . 
\$
This  implies that $\{\Psi(\btt^{(2,k)}, \blambda^{(1)})-\Psi(\bttc, \blambda^{(1)}) \}_{k\geq 1}$ is a non-increasing sequence. By induction, it follows that
\$
&\Psi(\btt^{(2,k+1)},\blambda^{(1)})- \Psi(\bttc,\blambda^{(1)})\leq  \Psi(\btt^{(2,k)},\blambda^{(1)})- \Psi(\bttc,\blambda^{(1)})\lesssim (1+\zeta)\kappa_*^{-1}s \lambda^2 . 
\$
Combining this with  Lemma \ref{lemma:sc} gives the desired bounds on $\|\btt^{(2,k+1)}-\bttc\|_2$ and $\|\btt^{(2,k+1)}-\bttc\|_1$. 

Finally, by an argument similar to that in the proof of Lemma~5.4 in \cite{FLSZ2018}, we can derive the stated results for all $\ell\geq 3$.
\end{proof}

For $\epsilon>0$, let $\widetilde\btt$ be an  $\epsilon$-optimal solution to the program
$
\min_{\btt}  \{\cL_\tau(\btt)+\|\blambda\circ \btt \|_1  \}. 
$
The following lemma provides conditions under which $\widetilde \btt$ falls in an $\ell_1$-cone.

\begin{lemma}\label{lemma:thetahat}
Let $\cE \subseteq [d]$ be a subset satisfying $\cS \subseteq \cE$, and assume $ \lambda \geq \| \blambda \|_\infty \vee  4\{ \|\nabla\cL(\bttc)\|_\infty+\epsilon  \}$ and $\|\blambda_{\cE^{{\rm c}}}\|_{\min}\geq \lambda/2$. Then, any $\epsilon$-optimal solution $\widetilde\btt$ satisfies the cone constraint
	\$
	\|(\tbt-\bttc)_{ \cE^{{\rm c}}}\|_1&\leq \frac{\|\blam\|_\infty+\|\nabla\cL(\bttc)\|_\infty+\epsilon}{\|\blam_{\cE^{{\rm c}}}\|_{\min}- \|\nabla\cL(\bttc)\|_\infty - \epsilon }\|{(\tbt-\bttc)}_{  \cE}\|_1\leq 5\|(\tilde\btt-\bttc)_{ \cE}\|_1 .
	\$
\end{lemma}
\begin{proof}[Proof of Lemma~\ref{lemma:thetahat}]
For any $\bxi\in\partial \|\widetilde\bbeta  \|_1$, let  $\bu=\nabla\cL(\tbt)+ \blam\circ  \bxi$.
By the convexity of $\cL(\cdot)$, $\langle \nabla \cL(\wt \bbeta)  - \nabla \cL( \bbeta^*) , \wt \bbeta - \bbeta^*\rangle \geq 0$. This, together with the inequality $\langle \nabla \cL(\wt\bbeta ) + \blambda \circ \bxi , \wt \bbeta -\bbeta^* \rangle \leq \| \bu \|_\infty \| \wt \bbeta - \bbeta^* \|_1 $, implies
\#\label{akkt1.0}
0\leq \|\bu\|_\infty\|\tbt-\bttc\|_1- \underbrace{ \langle\nabla\cL(\bttc), \tbt-\bttc \rangle}_{\Rom{1}} -\underbrace{\langle\blam\circ \bxi  , \tilde \bbeta  -\bbeta^*  \rangle}_{\Rom{2}}.
\#

For \Rom{1} and \Rom{2}, note that
	$
	\Rom{1}\geq -\|\nabla\cL(\bttc)\|_\infty\|\tbt-\btt\|_1,
	$
and
	\$
	\Rom{2}= \langle\blam\circ \bxi ,   \tilde \bbeta  -\bbeta^*  \rangle
	&=  \langle(\blam\circ  \bxi )_{\cE^{{\rm c}}},(\tilde \bbeta -\bbeta^* )_{\cE^{{\rm c}}}  \rangle+ \langle(\blam\circ   \bxi)_{\cE},(\tilde \bbeta -\bbeta^*)_{\cE} \rangle \notag\\
	&\geq \|\blam_{\cE^{{\rm c}}}\|_{\min}\|(\tilde \bbeta -\bbeta^* )_{\cE^{{\rm c}}}\|_1-\|\blam_{\cE}\|_\infty\|(\tilde \bbeta -\bbeta^*)_{\cE}\|_1.
	\$
Substituting the above bounds into (\ref{akkt1.0}) and taking the infimum over $\bxi\in \partial \|\tilde \bbeta \|_1$ yields
	\$
	0&\leq  - [\|\blam_{ \cE^{{\rm c}}}\|_{\min}- \{ \|\nabla\cL(\bttc)\|_\infty+\omega_{\blam}(\tbt) \}  ] \|(\tbt-\bttc)_{  \cE^{{\rm c}}}\|_1\\
	&\hspace{1cm}+  \{ \|\blam_{\cE}\|_\infty +\|\nabla\cL(\bttc)\|_\infty+\omega_{\blam}(\tbt)  \}\|(\tbt-\bttc)_{  \cE}\|_1,
	\$
or equivalently,	
	\$
	\|(\tbt-\bttc)_{  \cE^{{\rm c}}}\|_1 &\leq \frac{ \|\blam \|_\infty  +\|\nabla\cL(\bttc)\|_\infty+\omega_{\blam}(\tbt)}{\|\blam_{\cE^{{\rm c}}}\|_{\min}-\{ \|\nabla\cL(\bttc)\|_\infty+\omega_{\blam}(\tbt) \}} \|(\tbt-\bttc)_{  \cE}\|_1.
	\$
This proves the stated result.
\end{proof}

\subsubsection{Proof of the theorem}

Restricting our attention to the $\ell$-th subproblem, we write $\phi^{(k)} = \phi^{(\ell, k)}$ for simplicity.
Define the subset $\mathbb L = \{ \alpha \hbt +(1-\alpha)\btt^{(k-1)} : 0\leq \alpha \leq 1 \}$. Due to  local majorization, we have
\$
& \Psi (\btt^{(k)},\blam ) \\
& \leq  \min_{\btt \in \mathbb L } \left\{  {\cL(\btt^{(k-1)}) + \langle\nabla\cL(\btt^{(k-1)}),\btt-\btt^{(k-1)} \rangle}  + \frac{\phi^{(k)}}{2}  \|\btt - \btt^{(k-1)} \|_2^2+ \|\blam\circ \bbeta   \|_1\right\}\notag\\
& \leq  \min_{\btt \in \mathbb L }\left\{\cL(\btt)+\frac{\phi^{(k)}}{2}\|\btt-\btt^{(k-1)} \|_2^2 +\|\blambda\circ  \bbeta \|_1\right\},
\$
where we used the convexity of $\cL(\cdot)$ in the second inequality. Since $\Psi(\btt,\blambda)=\cL(\btt)+\|\blambda\circ \bbeta  \|_1$ is minimized at $\hbt$, by convexity we have
\#
& \Psi(\btt^{(k)},\blam) \leq  \min_{ \btt \in \mathbb L }\left\{\Psi(\btt, \blam)+\frac{\phi^{(k)}}{2}\|\btt-\btt^{(k-1)}\|_2^2\right\} \notag\\
	& \leq  \min_{ 0\leq \alpha\leq 1 }\left\{ \alpha \Psi(\hbt,\blam)+(1-\alpha) \Psi(\btt^{(k-1)},\blam)+\frac{\alpha^2\phi^{(k)}}{2}\|\btt^{(k-1)}-\hbt\|_2^2\right\}\notag\\
	&  =  \min_{ 0\leq \alpha\leq 1 }\left\{\Psi(\btt^{(k-1)},\blam) - \alpha \{ \Psi(\btt^{(k-1)},\blam) - \Psi(\hbt,\blam) \} + \frac{\alpha^2\phi^{(k)}}{2}{\|\btt^{(k-1)} - \hbt\|_2^2}\right\}. \label{eq:comp:2.1}
\#

Next, we  bound the right-hand side of \eqref{eq:comp:2.1}.  By Lemma~\ref{lemma:rsc},
\begin{align*}
	&\|(\btt^{(k-1)})_{\cE_\ell^{{\rm c}}}\|_0\leq  s' ,~ \|\btt^{(k-1)}-\bttc\|_2\lesssim  s^{1/2} \lambda \lesssim  r ~~\textnormal{and}~~ \|\btt^{(k-1)}-\bttc\|_2\lesssim s \lambda .
\end{align*}
Similarly, it can be shown the the optimum $\hat \bbeta$ satisfies the same properties. Hence,
\$
\btt^{(k)},    \hbt\in   \bttc+  \BB (r)  ~~\mbox{ and }~~  \| \btt^{(k)} -  \hbt  \|_0 \leq   | \cE_\ell | + 2 s' \leq C_0 s + 2 s'.
\$
By the first-order optimality condition, there exists some $\hat\bxi   \in\partial\|\widehat\bbeta \|_1$  such that $\nabla\cL(\hbt)+  \blam\circ\widehat\bxi ={\bf 0}$.
Moreover, define $D_\cL(\btt_1,\btt_2) = \cL(\btt_1) - \cL(\btt_2) - \langle \nabla \cL(\btt_2) , \btt_1 - \btt_2 \rangle$.
Using Definition~\ref{lse}, Lemma~\ref{lemma:bd.rsc}, and the convexity of $\cL(\cdot)$ and  $\ell_1$-norm,  we obtain that
\$
 & \Psi(\btt^{(k-1)},\blam)-\Psi(\hbt,\blam)  \nn \\
 & \geq  \langle\nabla\cL(\hbt)+  \blam\circ \hat\bxi ,\btt^{(k-1)}-\hbt \rangle + D_\cL(\btt^{(k-1)}, \hat\btt ) \geq  \frac{\kappa_-}{2} \|\btt^{(k-1)}-\hbt  \|_2^2,
\$
where $\kappa_-=\kappa_-(C_0 s+2s' , r, \tau)$.
Plugging this bound into \eqref{eq:comp:2.1} yields
\$
 & \Psi(\btt^{(k)},\blam) \\ 
&\leq \min_{ 0\leq \alpha\leq 1 }\bigg[ \Psi(\btt^{(k-1)},\blam)-\alpha \{ \Psi(\btt^{(k-1)},\blam)-\Psi(\hbt,\blam) \} +\frac{\alpha^2\phi^{(k)}}{\kappa_-} \{ \Psi(\btt^{(k-1)},\blam)-\Psi(\hbt,\blam) \}\bigg] \\
&\leq \Psi(\btt^{(k-1)},\blam)-\frac{\kappa_-}{4\phi^{(k)}} \{ \Psi(\btt^{(k-1)},\blam)-\Psi(\hbt,\blam) \}.
\$
Following the proof of Lemma~\ref{lemma:c.1}, it can be similarly shown that $\phi^{(k)}\leq \gamma_u\kappa^*$ under Condition~\ref{assume:LSE}. 
Consequently,
\$
\Psi (\btt^{(k)},\blam )-\Psi(\hbt,\blam)&\leq \left(1-\frac{1}{4\gamma_u \zeta}\right)^k \{\Psi(\btt^{(0)},\blam)-\Psi(\hbt,\blam) \},
\$
where $\zeta={\kappa^*}/{\kappa_*}$.

By an argument similar to that in the proof of Lemma \ref{lemma:c.1}, we can show that, for $\ell\geq 2$,
\$
\omega_{\blambda^{\ell-1}} (\btt^{(\ell, k)} )\leq \kappa^*(1+\gamma_u) \|\btt^{(\ell, k)}-\btt^{(\ell, k-1)} \|_2 .
\$
Further, using Lemma~\ref{lemma.c.2} to bound $ \|\btt^{(\ell, k)}-\btt^{(\ell, k-1)} \|_2$ from above  and noting that $\phi^{(k)} \geq \kappa_*$, we obtain
\$
& \omega_{\blambda^{(\ell-1)}}(\btt^{(\ell, k)})  \\
& \leq   (1+\gamma_u)  \kappa^* \sqrt{ (2/\kappa_*) \{ \Psi(\btt^{(\ell,k-1)}, \blambda^{(\ell-1)})-\Psi(\btt^{(\ell, k)}, \blambda^{(\ell-1)})\} }\\
& \leq  (1+\gamma_u )\sqrt{2 \zeta \kappa^* \{  \Psi(\btt^{(\ell,k-1)}, \blambda^{(\ell-1)})-\Psi(\hbt^{(\ell)}, \blambda^{(\ell-1)})\} }\\
&\leq  (1+ \gamma_u )\sqrt{2\zeta \kappa^* \left(1-\frac{1}{4\gamma_u\zeta}\right)^{k-1} \{ \Psi(\btt^{(\ell, 0)}, \blambda^{(\ell-1)})-\Psi(\hbt^{(\ell)}, \blambda^{(\ell-1)}) \} }\\
&\leq C (1+\gamma_u )\sqrt{ \frac{ \zeta \kappa^* }{ \phi^{(\ell, 0)} }  \left(1-\frac{1}{4\gamma_u\zeta}\right)^{k-1} s \lambda^2 } \leq C (1+\gamma_u  ) \zeta \sqrt{  \left(1-\frac{1}{4\gamma_u\zeta}\right)^{k-1} s \lambda^2 } , 
\$
where the last step applies Lemmas~\ref{lemma.c.2} and \ref{lemma:rsc}.

 To make the right-hand side of the above inequality smaller than $\epsilon_{{\rm t}}$, we need $k$ to be sufficiently large  that  $k\geq C_1 \log(C_2 s^{1/2} \lambda / \epsilon_{{\rm t}} )$, where $C_1, C_2>0$ are constants depending only on localized sparse eigenvalues and $\gamma_u$. 
This completes the proof.  \qed

\end{supplement}



\end{document}